\pdfoutput=1
\documentclass[10pt]{article}
\PassOptionsToPackage{usenames,dvipsnames}{xcolor}
\usepackage[T1]{fontenc}
\usepackage{newtxtext}
\usepackage{microtype}
\usepackage[english]{babel}					\usepackage[margin = 1.2in]{geometry}
\usepackage{amssymb,cite}	

\usepackage{amsmath,amsfonts,amsthm, mathtools}
\usepackage{newtxmath}
\usepackage{mathrsfs}
\usepackage[pdftex]{graphicx}	
\usepackage[title,titletoc,page]{appendix} 
\usepackage{listings} 
\usepackage{color}
\usepackage{caption}
\usepackage{subcaption}
\usepackage{cases}
\usepackage[ruled,vlined]{algorithm2e}
\usepackage{sectsty}
\usepackage[prependcaption,textsize=tiny,textwidth=3cm]{todonotes}
\usepackage{cases}
\allsectionsfont{\centering \normalfont\scshape}

\usepackage{stmaryrd}

\usepackage{fancyhdr}
\usepackage[cal=cm]{mathalfa}
\usepackage{mathrsfs}
\usepackage[dvipsnames]{xcolor}

\usepackage{tikz}
\usetikzlibrary{shapes.geometric, shapes.misc, positioning, arrows, decorations.pathreplacing, angles, quotes,calc}
\usepackage{booktabs}
\usepackage{xfrac}
\usepackage{siunitx}
\usepackage{centernot}
\usepackage{comment}
\usepackage{chngcntr}
\usepackage{extarrows}
\usepackage{adjustbox}
\usepackage{tikz-cd}

\makeatletter
\def\thmhead@plain#1#2#3{%
  \thmname{#1}\thmnumber{\@ifnotempty{#1}{ }\@upn{#2}}%
  \thmnote{ {\the\thm@notefont#3}}}
\let\thmhead\thmhead@plain
\makeatother

\def\smallsection#1{\smallskip\noindent\textbf{#1}.}

\usepackage{hyperref}
\hypersetup{colorlinks = true,
linkcolor = MidnightBlue, anchorcolor =red,
citecolor = ForestGreen,
filecolor = red,urlcolor = red,
            pdfauthor=author}

\usepackage{titlesec}
\titleformat{\subsection}{\normalfont\large\bfseries}{\thesubsection}{1em}{}
\titleformat{\subsubsection}{\normalfont\normalsize\bfseries}{\thesubsubsection}{1em}{}

\usepackage{cleveref}
\crefname{theorem}{Theorem}{Theorems}
\crefname{lemma}{Lemma}{Lemmas}
\crefname{proposition}{Proposition}{Propositions}
\crefname{corollary}{Corollary}{Corollaries}
\crefname{definition}{Definition}{Definitions}
\crefname{remark}{Remark}{Remarks}

\usepackage{aliascnt}

\newtheorem{theorem}{Theorem}[section]

\newaliascnt{lemma}{theorem}
\newtheorem{lemma}[lemma]{Lemma}
\aliascntresetthe{lemma}

\newaliascnt{proposition}{theorem}
\newtheorem{proposition}[proposition]{Proposition}
\aliascntresetthe{proposition}

\newaliascnt{corollary}{theorem}
\newtheorem{corollary}[corollary]{Corollary}
\aliascntresetthe{corollary}

\newaliascnt{remark}{theorem}
\newtheorem{remark}[remark]{Remark}
\aliascntresetthe{remark}

\theoremstyle{definition}
\newaliascnt{definition}{theorem}
\newtheorem{definition}[definition]{Definition}
\aliascntresetthe{definition}

\newaliascnt{example}{theorem}

\aliascntresetthe{example}

\numberwithin{equation}{section}
\numberwithin{figure}{section}

\newaliascnt{condition}{theorem}

\aliascntresetthe{condition}

\newaliascnt{assumption}{theorem}
\newtheorem{assumption}[assumption]{Assumption}
\aliascntresetthe{assumption}

\newcommand{\mc}[1]{\mathcal{#1}}

\def\bb{\begin{equation}
  \left\{\ 
   \begin{aligned} }
\def\ee{   \end{aligned}
  \right.
  \end{equation}}

\def\mm{ \left[
 \begin{matrix}}
\def\nn{\end{matrix} \right] } 

\def\p{\partial}
\def \dd{\mathrm{d}}

\def\d{\delta}

\def \vp{\varphi}

\def \diag{\operatorname{diag}}

\def \sdiv {\sdiv_S}

\def \NN{\mathbb{N}}
\def \R{\mathbb{R}}

\def \C{\mathbb{C}}

\def \Z{\mathbb{Z}}

\def \i{\mathrm{i}}

\def \A{\mathcal{A}}
\def \L{\mathcal{L}}

\def \O{\mathcal{O}}
\def \S{\mathcal{S}}

\def\D{\mathcal{D}}

\def \abs#1{\mathopen| #1 \mathclose|}

\newcommand{\Ker}{\operatorname{Ker}}

\begin{document}

\title{Frequency-dependent capacitance matrix formulation for Fabry--P\'{e}rot resonances in two- and three-dimensional systems}

\author{Habib Ammari\footnote{
  ETH Z\"urich, Department of Mathematics, Rämistrasse 101, 8092 Z\"urich, Switzerland and 
  Hong Kong Institute for Advanced Study, City University of Hong Kong, Kowloon Tong, Hong Kong Special Administrative Region
  (habib.ammari@math.ethz.ch).} \and Bowen Li\footnote{Department of Mathematics,  City University of Hong Kong, Kowloon Tong, Hong Kong SAR, (bowen.li@cityu.edu.hk).}
\and Ping Liu\footnote{
School of Mathematical Sciences, Zhejiang University, Hangzhou, 310027, China, (pingliu@zju.edu.cn, shaoyingjie323@zju.edu.cn).}  \thanks{
Institute of
Fundamental and Transdisciplinary Research, Zhejiang University, Hangzhou, 310027, China.}
\and Yingjie Shao\footnotemark[3] \and Alexander Uhlmann\footnote{
  ETH Z\"urich, Department of Mathematics, Rämistrasse 101, 8092 Z\"urich, Switzerland
(alexander.uhlmann@sam.math.ethz.ch).}
}

\date{}
\maketitle

\begin{abstract}
We study scattering resonances of finite and infinite periodic two- and three-dimensional systems of high-contrast resonators beyond the subwavelength regime. At each fixed admissible nonzero reference frequency, we introduce frequency-dependent capacitance matrices and derive quantitative asymptotic expansions of hybridized Fabry--P\'erot resonant frequencies and their corresponding eigenmodes in terms of the material contrast parameter. We provide a partial differential equation (PDE) formulation of the
frequency-dependent capacitance matrix analogous to the one for the capacitance matrix in the subwavelength regime. Based on this PDE formulation, we establish key properties of the frequency-dependent capacitance matrix and estimate its
norm at high frequencies for a single smooth resonator with nontrapping exterior in three dimensions, identifying a regime in which the reduced residual remains perturbative. For infinite periodic resonator arrays, we prove bandgap opening under a uniform exterior non-resonance condition and generic Dirac cones for a honeycomb scaling family. Our results extend the use of discrete approximations as a powerful tool for characterizing the resonant properties of finite and infinite periodic systems of high-contrast resonators at arbitrarily high frequencies and understanding their anomalous localization and transport properties arising from strong coupling to a discrete set of eigenfrequencies.
\end{abstract} 

    

\tableofcontents

\section{Introduction}

The possibility of guiding and localizing waves at given frequencies is critical for many applications in nanosciences and quantum technologies. Structures achieving such properties with building blocks much smaller than the operating wavelengths are known as meta-structures. These meta-structures enjoy guiding and localization properties at subwavelength scales and have recently attracted a lot of attention both in physics and applied mathematics. Their building blocks are high-contrast resonators. Recently, functional analytic methods for studying these subwavelength resonator systems have been developed in \cite{ammari2025mathematical,ammari2024functional}. In particular, rigorous discrete approximations of scattering resonance problems have been derived in an asymptotic subwavelength regime. This has been achieved by reframing the governing equations as a nonlinear eigenvalue problem in terms of integral operators. In the subwavelength limit, resonant eigenmodes are described by the eigenmodes of a generalized capacitance matrix. Using this formulation, it is possible to describe subwavelength resonances and related exotic transport and localization phenomena \cite{anderson,ammari2019topologically,fano,ammari2020exceptional,ammari2020robust}.

Nevertheless, it turns out that in ultrasonics the design of microstructured materials or acoustic cavities with building blocks larger than the operating wavelength is important in applications and mathematically challenging; see, for instance, \cite{refhigh1,refhigh2,refhigh3,ultrasonic}. Our main aim in this part of the series of papers \cite{fabry1,fabry2,pm1} is to show that high-contrast systems of resonators can operate beyond the subwavelength regime and ensure unusual wave transport and localization properties in two- and three-dimensional settings. Parts I and II are only devoted to, respectively, finite- and periodic one-dimensional systems of high-contrast resonators.    

We consider scattering resonances of high-contrast two- and three-dimensional resonator
systems beyond the subwavelength regime, \emph{i.e.}, at wavelengths smaller than the
characteristic resonator size. We prove that scattering resonances, called Fabry--P\'erot
resonances, bifurcate from the interior Neumann eigenfrequencies, and their leading-order
frequency shifts are governed by the eigenvalues of the \emph{frequency-dependent
capacitance matrix}, which plays the same role as the generalized capacitance matrix in
the subwavelength regime. We establish key properties of the associated matrices: the
unweighted matrix $\mathscr C(\omega_0)$ is complex symmetric, while the
frequency-dependent capacitance matrix $\mc C(\omega_0)$ is similar to a
complex-symmetric matrix and is generally non-Hermitian due to radiation losses. For a single smooth resonator with nontrapping exterior in three dimensions,
we derive a high-frequency norm estimate that identifies a perturbative regime for the
reduced residual.
We then extend the formalism to infinite periodic structures. Introducing the quasiperiodic
frequency-dependent capacitance matrix, we derive leading-order asymptotic formulas for
the Bloch band functions and corresponding Bloch eigenmodes of a crystal of high-contrast
resonators in terms of the material contrast. Under natural non-resonance conditions on
the exterior domain, we prove that bandgap opening occurs as the contrast parameter tends
to zero. Assuming honeycomb lattice symmetry, we further prove that Dirac degeneracies
arise generically beyond the subwavelength regime.

The frequency-dependent capacitance matrix formalism serves a dual purpose: it provides
an efficient numerical tool for computing scattering resonances and eigenmodes of finite
and periodic systems beyond the subwavelength regime, and it furnishes a theoretical
framework for understanding the anomalous spectral and transport properties that arise
from strong coupling to the Neumann eigenvalues. In particular, it opens the door to the
systematic design of meta-structures operating at wavelengths smaller than the
characteristic resonator size.

The paper is organized as follows. In \Cref{sec:2}, we formulate the scattering resonance
problem for a finite collection of high-contrast resonators. \Cref{sec:3} is devoted to
the approximation of the leading frequency shifts of scattering resonances, as the
contrast parameter tends to zero, by the eigenvalues of the frequency-dependent
capacitance matrix. We construct the
frequency-dependent capacitance matrix from the exterior Dirichlet-to-Neumann map,
establish that its unweighted form is complex symmetric and that the frequency-dependent
capacitance matrix is similar to a complex-symmetric matrix, derive, for a single smooth
three-dimensional resonator with nontrapping exterior, a high-frequency norm estimate
that identifies a perturbative regime for the reduced residual, and present numerical results
illustrating the accuracy of the approximations. In \Cref{sec:4}, we turn to crystals of
high-contrast resonators and derive asymptotics of the Bloch band functions in terms of
material contrast. We prove bandgap opening beyond the subwavelength regime for a
three-dimensional lattice with a single resonator per unit cell under the interior regularity and uniform exterior non-resonance conditions, and
establish the generic occurrence of Dirac cones for a two-dimensional dimer honeycomb
scaling family. The appendix
provides an ODE formulation of the frequency-dependent capacitance matrix, complementing
the one-dimensional analysis of \cite{fabry1}: we show that, as in higher dimensions, it
can be computed explicitly from the exterior Dirichlet-to-Neumann map. A notable
difference is that, while this map is well defined at every positive real frequency in two
and three dimensions, it fails to be well defined in one dimension when the frequency is a
Dirichlet eigenvalue of $-\mathrm{d}^2/\mathrm{d}x^2$ in one of the separating intervals
between the resonators.

Note that most of the derivations and results in this paper hold for both two- and three-dimensional cases in the same manner. This is in contrast with the subwavelength regime, where the Dirichlet-to-Neumann map has a logarithmic singularity at zero frequency, while beyond the subwavelength regime, it is analytic in frequency. Nevertheless, for simplicity of presentation, we only consider the three-dimensional case unless otherwise indicated.

\section{Fabry--P\'{e}rot scattering resonances} \label{sec:2}

In this section, we formulate the scattering resonance problem for a finite collection of high-contrast resonators in $\mathbb{R}^3$ and reduce it, via the layer potential method, to the study of characteristic values of an operator-valued family.

Let $D = \cup_{j=1}^N D_j$, where $D_1,\dots,D_N\subset \mathbb{R}^3$ are bounded connected domains with $C^\infty$ boundaries and pairwise disjoint closures; the resonators
need not be identical. We assume that $\mathbb{R}^3\setminus\overline D$ is connected. The background medium has wave speed $v>0$, while the resonator $D_j$ has 
wave speed $v_j>0$. We write
\begin{equation*}
    k:=\frac{\omega}{v}, \qquad k_j:=\frac{\omega}{v_j}, \quad j=1,\dots,N,
\end{equation*}
for the exterior and interior wave numbers. Assume that the density contrasts $\delta_j\in\mathbb{C}$ satisfy
\begin{equation*}
    \delta_j = \O(\delta), \qquad \delta\to 0, \quad j=1,\dots,N,
\end{equation*}
and are encoded by the piecewise-constant function $\delta_{\partial D}(x):=\delta_j$ 
for $x\in\partial D_j$.

For a complex frequency $\omega\in \mathbb{C}$, we consider the resonance problem
\begin{equation}
\label{eq:multi-resonator-scattering-problem}
\begin{dcases}
\Delta u + k^2 u = 0 & \text{in } \mathbb{R}^3\setminus \overline{D},\\
\Delta u + k_j^2 u = 0 & \text{in } D_j, \quad j=1,\dots,N,\\
u\big|_+ = u\big|_- & \text{on } \partial D,\\
\delta_j \dfrac{\partial u}{\partial \nu}\Big|_+ - \dfrac{\partial u}{\partial \nu}\Big|_- = 0 
    & \text{on } \partial D_j, \quad j=1,\dots,N,\\
\text{$u$ is outgoing.}
\end{dcases}
\end{equation}
Here, $\nu$ is the unit outward normal to $\partial D$, and the subscripts $+$ and $-$ denote traces taken from the exterior and interior of $D$, respectively. Throughout this paper, we use interchangeably $\partial /\partial \nu$ and  $\partial_\nu$  to denote the normal derivative. 
When $k\neq 0$ is real, the outgoing  radiation condition is understood with respect to the exterior wave number $k$, namely,
\begin{equation} \label{sommerfeld}
\frac{\partial u}{\partial |x|}-\mathrm{i}k u=o(|x|^{-1})
\qquad \text{as } |x|\to \infty,
\end{equation}
uniformly in the angular variable $x/|x|$. For nonreal $k$, the outgoing condition is defined by meromorphic continuation of the outgoing resolvent from $\operatorname{Im} k>0$, equivalently through the outgoing Dirichlet-to-Neumann (DtN) map on a sphere containing $D$, away from its poles; see, for instance, \cite{ammari2025mathematical}.

Next, we introduce the layer-potential operators associated with the disconnected boundary 
$\partial D = \cup_{j=1}^N \partial D_j$. For $\omega\in \mathbb{C}$, let
\begin{equation*}
G^{\omega}(x):=-\frac{e^{\mathrm{i}\omega |x|}}{4\pi |x|},
\qquad x\neq 0,
\end{equation*}
be the outgoing Helmholtz Green function in $\mathbb{R}^3$. The corresponding single-layer 
potential and Neumann--Poincar\'e operator are defined by 
\begin{equation*}
\mathcal{S}_D^{\omega}[\varphi](x)
:=
\int_{\partial D} G^{\omega}(x-y)\,\varphi(y)\,\mathrm{d}\sigma(y):\ L^2(\partial D) \to H^1_{\rm loc}(\R^3)\,,
\end{equation*}
\begin{equation*}
\mathcal{K}_D^{\omega,*}[\varphi](x)
:=
\int_{\partial D} \frac{\partial}{\partial \nu_x}G^{\omega}(x-y)\,\varphi(y)\,\mathrm{d}\sigma(y):\ L^2(\partial D) \to L^2(\partial D)\,.
\end{equation*}
Here, $H^1_{\rm loc}(\R^3)$ is the set of functions that together with their weak derivatives are locally square integrable in $\R^3$.

With abuse of notation, we use the same notation for the boundary operator $\mathcal{S}_D^{\omega}: L^2(\partial D) \to H^1(\partial D)$, where $H^1(\partial D)$ denotes the space of $L^2$-functions on $\partial D$ with weak $L^2$
tangential derivatives.  
For convenience, we also write these operators in the following block forms:
\begin{equation*}
\mathcal{S}_D^{\omega}=\bigl(\mathcal{S}_{D_i,D_j}^{\omega}\bigr)_{i,j=1}^N,
\qquad
\mathcal{K}_D^{\omega,*}=\bigl(\mathcal{K}_{D_i,D_j}^{\omega,*}\bigr)_{i,j=1}^N,
\end{equation*}
where, for a density $\varphi_j$ supported on $\partial D_j$,
\begin{equation*}
\mathcal{S}_{D_i,D_j}^{\omega}[\varphi_j]
:=
\left.\mathcal{S}_D^{\omega}[\varphi_j]\right|_{\partial D_i},
\qquad
\mathcal{K}_{D_i,D_j}^{\omega,*}[\varphi_j]
:=
\left.\mathcal{K}_D^{\omega,*}[\varphi_j]\right|_{\partial D_i}.
\end{equation*}
Since the interior wave numbers may vary from one resonator to another, we introduce the following piecewise operators, as in \cite[Section 2.1]{ammari2024functional}, for $\psi \in L^2(\partial D)$ and $j=1,\dots,N$,
\begin{equation*}
\widetilde{\mc S}_D^{\omega}[\psi](x)
:=
\mc S_D^{k_j}[\psi](x),
\qquad x\in D_j,
\end{equation*}
and
\begin{equation*}
\widetilde{\mc K}_D^{\omega,*}[\psi](x)
:=
\mc K_D^{k_j,*}[\psi](x),
\qquad x\in \partial D_j.
\end{equation*}

We now introduce the layer potential representation for the scattering resonance \eqref{eq:multi-resonator-scattering-problem}. 
We seek a resonant field in the form
\begin{equation}
\label{eq:multi-resonator-layer-ansatz}
u(x)=
\begin{cases}
\mc S_D^{k}[\phi](x), & x\in \R^3\setminus \overline D,\\
\widetilde{\mc S}_D^{\omega}[\psi](x), & x\in D,
\end{cases}
\end{equation}
for some densities $\phi,\psi\in L^2(\partial D)$. The exterior part $S_D^{k}[\phi]$ of \eqref{eq:multi-resonator-layer-ansatz} automatically satisfies the outgoing radiation condition, while each interior component solves the corresponding Helmholtz equation in $D_j$. The jump relations for the single-layer potential yield \cite{ammari2009layer}:
\begin{equation*}
\mc S_D^{k}[\phi]\big|_+ = \mc S_D^{k}[\phi]\big|_-,
\qquad
\frac{\partial}{\partial \nu}\mc S_D^{k}[\phi]\Big|_{\pm}
=
\left(\pm \frac12 I + \mc K_D^{k,*}\right)[\phi],
\end{equation*}
on $\partial D$ and, on each $\partial D_j$,
\begin{equation*}
\widetilde{\mc S}_D^{\omega}[\psi]\big|_- = \mc S_D^{k_j}[\psi]\big|_{\partial D_j},
\qquad
\frac{\partial}{\partial \nu}\widetilde{\mc S}_D^{\omega}[\psi]\Big|_-
=
\left(-\frac12 I + \widetilde{\mc K}_D^{\omega,*}\right)[\psi].
\end{equation*}
Therefore, the transmission conditions in \eqref{eq:multi-resonator-scattering-problem} are equivalent to the following system of boundary integral equations \cite[Lemma 2.1]{ammari2025mathematical}:
\begin{equation} \label{eq:multi-resonator-bie-system}
\begin{cases}
\widetilde{\mc S}_D^{\omega}[\psi]-\mc S_D^{k}[\phi]=0\\
\left(-\dfrac12 I + \widetilde{\mc K}_D^{\omega,*}\right)[\psi]
-\delta_{\partial D}\left(\dfrac12 I + \mc K_D^{k,*}\right)[\phi]=0
\end{cases}
\qquad \text{on } \partial D.
\end{equation}
This can be written more compactly as
\begin{equation}
\label{eq:multi-resonator-operator-equation}
\mc A(\omega,\delta_{\partial D})
\begin{pmatrix}
\psi\\
\phi
\end{pmatrix}
=
\begin{pmatrix}
0\\
0
\end{pmatrix},
\end{equation}
where the block operator $\mc A(\omega,\delta_{\partial D}):L^2(\partial D)\times L^2(\partial D) \to H^1(\partial D)\times L^2(\partial D)$ is
\begin{equation*}
\mc A(\omega,\delta_{\partial D}) = 
\begin{pmatrix}
\widetilde{\mc S}_D^{\omega} & -\mc S_D^{k}\\
-\dfrac12 I+\widetilde{\mc K}_D^{\omega,*} & -\delta_{\partial D}\left(\dfrac12 I+\mc K_D^{k,*}\right)
\end{pmatrix}.
\end{equation*}

We thus obtain the equivalence between the scattering resonance problem  \eqref{eq:multi-resonator-scattering-problem} and the operator-valued characteristic equation \eqref{eq:multi-resonator-operator-equation}, that is, provided that 
$\mathcal{S}_D^k$ and $\widetilde{\mathcal{S}}_D^\omega$ are invertible, a nontrivial 
solution of \eqref{eq:multi-resonator-scattering-problem} is equivalent to a nontrivial 
pair $(\psi,\phi)$ solving \eqref{eq:multi-resonator-operator-equation}; see also \cref{lem:layerrep} below. Note that the condition therein that $k_j^2$ is not a Dirichlet eigenvalue of $-\Delta$ in $D_j$ for each $j = 1, \dots, N$ is automatically satisfied, since the resonant frequencies satisfy $\mathrm{Im}\, \omega < 0$ in the case of real material parameters, which implies $k_j^2 = \omega^2 / v_j^2 \notin (0,+\infty)$, while the Dirichlet eigenvalues of $-\Delta$ in $D_j$ have to be strictly positive. In the subwavelength regime $\omega \to 0$, all the above assumptions are naturally satisfied. Those $\omega \in \mathbb{C}$ for which \eqref{eq:multi-resonator-scattering-problem} admits nontrivial solutions are scattering resonances; in this paper, those that bifurcate, as $\delta\to0$, from nonzero interior Neumann eigenfrequencies are called \emph{Fabry--P\'erot resonances}~\cite{fabry1,li2025high}.

\section{Frequency-dependent capacitance matrix approximation} \label{sec:3}

Having reduced the scattering resonance problem to the characteristic equation \eqref{eq:multi-resonator-operator-equation}, we shall derive a finite-dimensional approximation for the Fabry--P\'{e}rot resonant frequencies. For this, we first identify the characteristic values $\omega_0$ of the limiting operator $\mathcal{A}(\omega,0)$ with the interior Neumann eigenvalue problem in Section \ref{sec:limit}, and then establish a pole-pencil decomposition of $\mathcal{A}(\omega,0)^{-1}$ near $\omega_0$ in Section \ref{sec:pole-pencil}. In Section \ref{sec:finite-reduction}, we apply a Lyapunov--Schmidt reduction to the full operator family $\mathcal{A}(\omega,\delta)$ to show that the resonant frequencies near $\omega_0$ are, to the leading order in $\delta$, determined by the eigenvalues of a \emph{frequency-dependent capacitance matrix} $\mc C(\omega_0)$. We will also discuss the general properties of $\mc C(\omega_0)$ in \Cref{sec:general_properties}. In particular, we derive a norm estimate of $\mc C(\omega_0)$ and identify a high-frequency regime in which the reduced residual remains perturbative, as shown at the end of this section.

\subsection{The limiting operator and the interior Neumann problem} \label{sec:limit}

We begin by analyzing the limiting operator $\mathcal{A}(\omega,0)$ obtained as $\delta\to 0$, and identify its characteristic values with the interior Neumann eigenvalues of $-\Delta$ in the individual resonators.

Setting $\delta_{\partial D}=0$ in \eqref{eq:multi-resonator-scattering-problem} gives 
the limiting problem
\begin{equation}
\label{eq:limiting-pde-problem}
\begin{dcases}
\Delta u + k^2 u = 0 & \text{in } \mathbb{R}^3\setminus \overline{D},\\
\Delta u + k_j^2 u = 0 & \text{in } D_j, \quad j=1,\dots,N,\\
u\big|_+ = u\big|_-  & \text{on } \partial D,\\
\dfrac{\partial u}{\partial \nu}\Big|_- = 0  & \text{on } \partial D_j, \quad j=1,\dots,N,\\
\text{$u$ is outgoing,}
\end{dcases}
\end{equation}
in which the exterior field is coupled to the resonators only through the continuity of the trace, while the interior flux condition reduces to a homogeneous Neumann condition on each $\partial D_j$. In particular, the above equivalence between the resonance problem \eqref{eq:multi-resonator-scattering-problem} and the boundary integral equation
\eqref{eq:multi-resonator-operator-equation} 
extends to the limiting problem \eqref{eq:limiting-pde-problem}, as we now state. We refer, for example, to \cite{ammari2025mathematical} for its proof. Let $\sigma_{\rm Neu}(-\Delta;D_j)$ and $\sigma_{\rm Dir}(-\Delta;D_j)$ denote the Neumann eigenvalues and Dirichlet eigenvalues of $-\Delta$ in $D_j$, respectively. 

\begin{lemma}[{(Special case of \cite[Lemma 1.2]{ammari2025mathematical})}] \label{lem:layerrep}
Let
\begin{equation*}
    \mathcal A_0(\omega):=\mathcal A(\omega,0)
= \begin{pmatrix}
\widetilde{\mathcal S}_D^\omega & -\mathcal S_D^k\\[1mm]
-\dfrac12 I+\widetilde{\mathcal K}_D^{\omega,*} & 0
\end{pmatrix}.
\end{equation*}
Assume that
\begin{enumerate}
    \item[(i)] $\mathcal S_D^k:L^2(\partial D)\to H^1(\partial D)$ is invertible;
    \item[(ii)] $\widetilde{\mathcal S}_D^\omega:L^2(\partial D)\to H^1(\partial D)$ is invertible;
    \item[(iii)] for each $j=1,\dots,N$, $k_j^2 \notin \sigma_{\rm Dir}(-\Delta;D_j)$.  
\end{enumerate}
Then, equation \eqref{eq:limiting-pde-problem} is equivalent to the existence of a frequency $\omega \in \C$ and a nontrivial density $(\psi,\phi)\in L^2(\partial D)\times L^2(\partial D)$ such that 
\begin{equation*}
     \mathcal A_0(\omega)
    \binom{\psi}{\phi}=0.
\end{equation*}
In this case,  $u$ given by
\begin{equation*}
    u(x)=
\begin{cases}
\mathcal S_D^k[\phi](x), & x\in \mathbb R^3\setminus \overline D,\\[1mm]
\widetilde{\mathcal S}_D^\omega[\psi](x), & x\in D,
\end{cases}
\end{equation*}
is a solution to \eqref{eq:limiting-pde-problem}.
\end{lemma}

We can now identify the characteristic values of the limiting operator. Note that for real nonzero $k$, the exterior Dirichlet problem with the outgoing radiation
condition is uniquely solvable. Consequently,
the limiting PDE \eqref{eq:limiting-pde-problem} has a nontrivial solution if and only if
\begin{equation*}
    \omega^2\in \bigcup_{j=1}^N v_j^2\,\sigma_{\rm Neu}(-\Delta;D_j).
\end{equation*}
Moreover, when $\mathcal{S}_D^k$ is invertible, equivalently, when $k^2 = \omega^2/v^2$  is not a Dirichlet eigenvalue of $-\Delta$ in $D$, then $\mathcal{A}_0(\omega)$ shares its characteristic values with $-\tfrac{1}{2} I + \widetilde{\mathcal{K}}_D^{\omega,*}$. Combining these facts with \cref{lem:layerrep}, we have the following lemma.  

\begin{lemma}
\label{thm:limiting-char-values-sharp}
Under the assumptions in \cref{lem:layerrep},
the following statements are equivalent:
\begin{enumerate}
\item[(i)] $\omega$ is a characteristic value of the limiting operator $\mc A_0$;
\item[(ii)] $\omega$ is a characteristic value of $-\dfrac12 I+\widetilde{\mathcal K}_D^{\omega,*}$;
\item[(iii)] there exists $j\in\{1,\dots,N\}$ such that $k_j^2$ is a Neumann eigenvalue of $-\Delta$ in $D_j$.
\end{enumerate}
That is, 
\begin{equation}
\label{eq:limiting-char-values-sharp}
\omega \text{ is a characteristic value of } \mc A_0
\quad\Longleftrightarrow\quad
\omega^2\in \bigcup_{j=1}^N v_j^2\,\sigma_{\rm Neu}(-\Delta;D_j). 
\end{equation}
Moreover, the geometric multiplicity of $\omega$ as a characteristic value of $\mc A_0$ is
\begin{equation}
\label{eq:multiplicity-limiting}
\dim \Ker \mc A_0(\omega)
=
\sum_{j=1}^N
\dim \ker\bigl(-\Delta_{{\rm Neu},D_j}-k_j^2\bigr).
\end{equation}
\end{lemma}

\begin{remark}
It is known that, generically, the Neumann and Dirichlet spectra of $-\Delta$ 
on a connected domain are disjoint; see \cite{generic,henry2005perturbation}. 
Consequently, for generic resonators $D_j$, the assumption that $k_j^2 = \omega^2/v_j^2$ 
is not a Dirichlet eigenvalue of $-\Delta$ in $D_j$ is satisfied whenever $k_j^2$ is a Neumann eigenvalue and therefore can be omitted. 
\end{remark}

\begin{remark}
The existence of Fabry--P\'erot resonances then follows from \cref{thm:limiting-char-values-sharp} and the asymptotic Gohberg--Sigal theory~\cite{gohberg1971operator}.
\end{remark}

\subsection{Pole-pencil decomposition} \label{sec:pole-pencil}

In this subsection, we fix a positive characteristic value $\omega_0 > 0$ of the limiting operator
\begin{equation*}
\A_0(\omega):=\A(\omega,0).
\end{equation*}
We write
\begin{equation*}
    k_0:=\frac{\omega_0}{v},
    \qquad
    k_{j,0}:=\frac{\omega_0}{v_j},\qquad j=1,\dots,N,
\end{equation*}
and introduce the index set
\begin{equation*}
    \mathcal J:=\Bigl\{j\in\{1,\dots,N\}: k_{j,0}^2\in \sigma_{\rm Neu}(-\Delta;D_j)\Bigr\}.
\end{equation*}

For each $j\in \mathcal J$, let $m_j$ denote the multiplicity of the Neumann eigenvalue $k_{j,0}^2$ in $D_j$, and choose a real-valued $L^2(D_j)$-orthonormal basis
\begin{equation*}
    \{u_{j,\ell}\}_{\ell=1}^{m_j}
    \subset \ker\bigl(\Delta+k_{j,0}^2\bigr)
\end{equation*}
of the corresponding Neumann eigenspace, \emph{i.e.},
\begin{equation} \label{eq:sharp-neumann-eig}
\begin{cases}
(\Delta+k_{j,0}^2)u_{j,\ell}=0 & \text{in }D_j,\\[1mm]
\partial_\nu u_{j,\ell}=0 & \text{on }\partial D_j.
\end{cases}
\end{equation}
To connect the Neumann problem \eqref{eq:sharp-neumann-eig} with $\ker \mc{A}_0(\omega_0)$, we denote by $g_{j,\ell}\in H^1(\partial D)$ the boundary trace extended by zero to the other components, that is, 
\begin{equation} \label{eq:tracegu}
    g_{j,\ell}|_{\partial D_j}=u_{j,\ell}|_{\partial D_j},
    \qquad
    g_{j,\ell}|_{\partial D_i}=0\quad \text{for }i\neq j.
\end{equation}

\begin{remark}
When the resonators are identical, the situation is analogous to the subwavelength case: the material speeds $v_j$ coincide, the Neumann spectra of the components are the same, and hence $k_{j,0}^2$ is a Neumann eigenvalue of every resonator whenever it is one for a single component. For $\omega_0=0$ (\emph{i.e.} in the subwavelength regime), no identical-shape assumption is needed, since zero is automatically a Neumann eigenvalue of every resonator.
\end{remark}

\begin{lemma}\label{lem:kernel-A0-omega0}
Assume that $\omega_0\neq 0$ is an isolated characteristic value of $\A_0$, and that
\begin{equation*}
    \S_D^{k_0}:L^2(\partial D)\to H^1(\partial D),
    \qquad
    \widetilde \S_D^{\omega_0}:L^2(\partial D)\to H^1(\partial D)
\end{equation*}
are invertible, and that $k_{i,0}^2\notin\sigma_{\rm Dir}(-\Delta;D_i)$ for all $i=1,\ldots,N$. For each $j\in\mathcal J$ and $\ell=1,\dots,m_j$, define
\begin{equation} \label{eq:field2}
    \psi_{j,\ell}:=(\widetilde \S_D^{\omega_0})^{-1}[g_{j,\ell}],
    \qquad
    \phi_{j,\ell}:=(\S_D^{k_0})^{-1}[g_{j,\ell}],
\end{equation}
and set
\begin{equation} \label{eq:field}
    \Psi_{j,\ell}:=\binom{\psi_{j,\ell}}{\phi_{j,\ell}}
    \in L^2(\partial D)\times L^2(\partial D).
\end{equation}
Then, we have
\begin{equation*}
    \ker \A_0(\omega_0) =  \mathrm{span}\bigl\{\Psi_{j,\ell}: j\in\mathcal J,\ 1\leq \ell\leq m_j\bigr\}.
\end{equation*}
In particular, it follows that
\begin{equation*}
    \dim \ker \A_0(\omega_0)=m,
    \qquad
    m:=\sum_{j\in\mathcal J} m_j.
\end{equation*}
\end{lemma}

For an analytic Fredholm family of operators $\mc{A}(\omega,\delta)$, an isolated characteristic value gives rise to a meromorphic inverse whose principal part is finite-rank. 

\begin{proposition} \label{prop:simplepole}
Let $\omega_0\neq 0$ be a characteristic value of the limiting operator $\A_0(\omega)$.
Assume that in a neighborhood of $\omega_0$, $\widetilde{\S}_D^\omega$ and $\S_D^k$ are invertible, and 
\begin{equation*}
    k_j^2 \notin \sigma_{\rm Dir}(-\Delta;D_j),
\qquad j = 1, \dots, N. 
\end{equation*}
Then, $\omega_0$ is a pole of order one of $\A_0(\omega)^{-1}$. 
\end{proposition}

\begin{proof}
We define the interior DtN map by
\begin{equation*} \Lambda_{\mathrm{int}}^\omega := \left(-\frac12 I+\widetilde{\mc K}_D^{\omega,*}\right)
\bigl(\widetilde{\mc S}_D^\omega\bigr)^{-1},
\end{equation*}
which depends analytically on $\omega$ in a sufficiently small neighborhood of $\omega_0$.
A direct block elimination shows that $\mc A_0(\omega)$ is analytically equivalent to
\begin{equation*}
    \begin{pmatrix}
I&0\\[1mm]
0&\Lambda_{\mathrm{int}}^\omega
\end{pmatrix}.
\end{equation*}
Indeed, one can compute 
\begin{equation*}
    \begin{pmatrix}
I&0\\[1mm]
-\Lambda_{\mathrm{int}}^\omega&I
\end{pmatrix}
\mc A_0(\omega)
\begin{pmatrix}
(\widetilde{\mc S}_D^\omega)^{-1}&(\widetilde{\mc S}_D^\omega)^{-1}\\[1mm]
0&(\mc S_D^k)^{-1}
\end{pmatrix}
=
\begin{pmatrix}
I&0\\[1mm]
0&\Lambda_{\mathrm{int}}^\omega
\end{pmatrix}.
\end{equation*}
Hence, the pole singularity of $\mc A_0(\omega)^{-1}$ is the same as that of $(\Lambda_{\mathrm{int}}^\omega)^{-1}$. 
By \cref{lem:layerrep}, the operator $\Lambda_{\mathrm{int}}^\omega$ is the direct sum of the DtN maps on the individual resonators:
\begin{equation*}
    \Lambda_{\mathrm{int}}^\omega
= \bigoplus_{j=1}^N \Lambda_j\!\left(\frac{\omega^2}{v_j^2}\right),
\end{equation*}
where $\Lambda_j(\lambda)$ is the DtN map for
\begin{equation*}
    (\Delta+\lambda)u=0\quad \text{in }D_j.
\end{equation*}
Consequently, we have the meromorphic function:
\begin{equation} \label{eq:block_decom}
    (\Lambda_{\mathrm{int}}^\omega)^{-1}
= \bigoplus_{j=1}^N N_j\!\left(\frac{\omega^2}{v_j^2}\right),
\end{equation}
where $N_j(\lambda):=\Lambda_j(\lambda)^{-1}$ is the Neumann-to-Dirichlet (NtD) map.

Since the Laplacian with Neumann boundary conditions on $D_j$ is self-adjoint, its resolvent has only simple poles at isolated eigenvalues. Therefore, for each $j\in\mc J$, the operator $N_j(\lambda)$ has a simple pole at $\lambda=k_{j,0}^2$. Since $\omega_0\neq 0$, the map $\omega\mapsto \omega^2/v_j^2$ is locally invertible near $\omega_0$, with nonzero derivative. Hence, each singular block remains a simple pole after composition with $\omega\mapsto \omega^2/v_j^2$. Summing over $j\in\mc J$ shows that $(\Lambda_{\mathrm{int}}^\omega)^{-1}$, and therefore also $\mc A_0(\omega)^{-1}$, has only a simple pole at $\omega_0$.
\end{proof}

Next, we give the pole-pencil decomposition of $\A_0(\omega)^{-1}$. A direct block computation shows that, for $\omega$ away from the characteristic values,
\begin{equation} \label{eq:inverseA0}
    \A_0(\omega)^{-1}
=
\begin{pmatrix}
0 & \bigl(\widetilde{\S}_D^\omega\bigr)^{-1}\bigl(\Lambda_{\rm int}^\omega\bigr)^{-1}\\[1mm]
-\bigl(\S_D^k\bigr)^{-1} &
\bigl(\S_D^k\bigr)^{-1}\bigl(\Lambda_{\rm int}^\omega\bigr)^{-1}
\end{pmatrix}. 
\end{equation}

\begin{proposition}\label{prop:A0-pole-pencil-reduction}
Let $\omega_0\neq 0$ be a characteristic value of $\A_0(\omega)$. Then, under the assumptions of \cref{prop:simplepole}, the Laurent expansion of $\A_0(\omega)^{-1}$ near $\omega_0$ is of the
form
\begin{equation} \label{eq:A0-pole-pencil-reduction}
    \A_0(\omega)^{-1}
= -\frac{1}{\omega-\omega_0}\,\mc P_{\omega_0}
+ \mc R_0(\omega), 
\end{equation}
where $\mc R_0(\omega)$ is analytic near $\omega_0$, and the finite-rank residue operator $\mc P_{\omega_0}$ is given by
\begin{equation*}
    \mc P_{\omega_0}
= \begin{pmatrix}
 0 & \bigl(\widetilde{\S}_D^{\omega_0}\bigr)^{-1}\Pi_{\omega_0}\\[1mm]
 0 & \bigl(\S_D^{k_0}\bigr)^{-1}\Pi_{\omega_0}
\end{pmatrix}.    
\end{equation*}
Moreover, let $\mathcal G_{\omega_0} :=\operatorname{span}\{g_{j,\ell}:j\in\mathcal J,\ 1\leq \ell\leq m_j\} \subset H^1(\partial D).$ Then $\Pi_{\omega_0}:L^2(\partial D)\to\mathcal G_{\omega_0}\subset H^1(\partial D)$ is the finite-rank operator
\begin{equation*}
    \Pi_{\omega_0}[h]
= \sum_{j\in\mc J}\frac{v_j^2}{2\omega_0} \sum_{\ell=1}^{m_j}
\langle h,g_{j,\ell}\rangle_{\partial D}\,g_{j,\ell},
\end{equation*}
and the same formula extends continuously to
\begin{equation*}
\Pi_{\omega_0}:H^{-1/2}(\partial D)
\longrightarrow\mathcal G_{\omega_0}\subset H^1(\partial D).
\end{equation*}
Here, $H^{1/2}(\partial D)$ is the trace space of functions in $H^1(D)$, $H^{-1/2}(\partial D)$ is its dual, and the bilinear duality pairing is taken component-wise. Therefore, it follows that 
\begin{equation} \label{eq:Pw0-rankone-reduction}
    \mc P_{\omega_0}\binom{f}{g}
=
\sum_{j\in\mc J}\frac{v_j^2}{2\omega_0}
\sum_{\ell=1}^{m_j}
\langle g,g_{j,\ell}\rangle_{\partial D}\,\Psi_{j,\ell},
\end{equation}
where $\Psi_{j,\ell}$ are defined in \eqref{eq:field}. 
\end{proposition}

\begin{proof}
It suffices to identify the principal part of the NtD map
$(\Lambda_{\rm int}^\omega)^{-1}$, according to \eqref{eq:inverseA0}. Recalling \eqref{eq:sharp-neumann-eig} and \eqref{eq:block_decom}, by the spectral resolution of the Neumann Laplacian on $D_j$, we have, near the eigenvalue $k_{j,0}^2$, 
\begin{equation*}
    N_j(\lambda) =
\frac{\Pi_j}{k_{j,0}^2-\lambda} + H_j(\lambda), \qquad 
\Pi_j[h] = \sum_{\ell=1}^{m_j} \langle h,g_{j,\ell}\rangle_{\partial D_j}\,g_{j,\ell}.
\end{equation*} 
where $H_j$ is analytic. Composing with $\lambda=\omega^2/v_j^2$ and using
\begin{equation*}
    k_{j,0}^2-\frac{\omega^2}{v_j^2} = -\frac{2\omega_0}{v_j^2}(\omega-\omega_0)+\O\bigl((\omega-\omega_0)^2\bigr),
\end{equation*}
we obtain
\begin{equation*}
N_j\!\left(\frac{\omega^2}{v_j^2}\right) = -\frac{1}{\omega-\omega_0}\frac{v_j^2}{2\omega_0}\Pi_j
+\widetilde H_j(\omega), 
\end{equation*}
with $\widetilde H_j$ analytic near $\omega_0$. Therefore, it follows that 
\begin{equation*}
\bigl(\Lambda_{\rm int}^\omega\bigr)^{-1} = -\frac{1}{\omega-\omega_0}\Pi_{\omega_0} + \mc H(\omega), \qquad \Pi_{\omega_0}[h] = \sum_{j\in\mc J}\frac{v_j^2}{2\omega_0}\sum_{\ell=1}^{m_j} \langle h,g_{j,\ell}\rangle_{\partial D}\,g_{j,\ell},
\end{equation*}
where $\mc H$ is analytic. Finally, $\widetilde{\S}_D^\omega$ and $\S_D^k$ are invertible and depend
analytically on $\omega$ near $\omega_0$. We finish the proof by using expression \eqref{eq:inverseA0} and \cref{lem:kernel-A0-omega0}. 
\end{proof}

\subsection{Finite-dimensional reduction} \label{sec:finite-reduction}

We now derive the finite-dimensional reduction of the full operator family
\begin{equation*}
    \A(\omega,\delta_{\partial D})=\A_0(\omega)+\L(\omega,\delta_{\partial D}),    
\end{equation*}
near a nonzero characteristic value $\omega_0$ of the limiting operator $\A_0$, which shall give the \emph{frequency-dependent capacitance matrix} in \cref{def:freq-cap-matrix-general-updated}.  Here, the operator $\mc{L}$ is defined by
\begin{equation} \label{def:lomega}
    \L(\omega,\delta_{\partial D}) :=
\begin{pmatrix}
0 & 0\\[1mm]
0 & -\delta_{\partial D}\Bigl(\dfrac12 I+\mc K_D^{k,*}\Bigr)
\end{pmatrix}.
\end{equation}
For convenience, we also write 
\begin{equation} \label{eq:dimensionIm}
    \mc I:=\{(j,\ell): j\in \mc J,\ 1\leq \ell\leq m_j\},
\qquad
m =|\mc I|=\sum_{j\in \mc J}m_j.
\end{equation}

\smallsection{Exterior fields and the matrix coefficients}
For each $(j,\ell)\in \mc I$, let $V_{j,\ell}$ be the unique outgoing solution of the exterior Dirichlet problem
\begin{equation}\label{eq:def-Vjl-updated}
\begin{cases}
(\Delta+k_0^2)V_{j,\ell}=0 & \text{in }\R^3\setminus\overline D,\\[1mm]
V_{j,\ell}=g_{j,\ell} & \text{on }\partial D.
\end{cases}
\end{equation}
Then,  $V_{j,\ell}$ admits the following layer potential representation: 
\begin{equation*}
V_{j,\ell}=\S_D^{k_0}\bigl[(\S_D^{k_0})^{-1}[g_{j,\ell}]\bigr]
\qquad\text{in }\R^3\setminus\overline D,
\end{equation*}
with 
\begin{equation*}
    \partial_\nu V_{j,\ell}\big|_+
=
\Bigl(\frac12 I+\mc K_D^{k_0,*}\Bigr)(\S_D^{k_0})^{-1}[g_{j,\ell}].    
\end{equation*}
Here, $(\S_D^{k_0})^{-1}$ denotes the inverse of $\S_D^{k_0}$ as an operator from $L^2(\partial D)$ into $H^1(\partial D)$.

\begin{definition}
\label{def:freq-cap-matrix-general-updated}
Let $\omega_0 \in \R$ be a characteristic value of $\mc{A}_0(\omega)$ (equivalently, the interior DtN map). The associated \emph{frequency-dependent capacitance matrix} at $\omega_0$ is the $m\times m$ matrix $\bigl(\mc C_{(i,r),(j,\ell)}(\omega_0, \delta_{\partial D})\bigr)_{(i,r),(j,\ell)\in \mc I}$, given by
\begin{equation}\label{eq:def-freq-cap-general-updated}
\mc C_{(i,r),(j,\ell)}(\omega_0, \delta_{\partial D})
:= -\frac{v_i^2}{2\omega_0}
\Bigl\langle \delta_{\partial D}\,\partial_\nu V_{j,\ell}\big|_+,g_{i,r}\Bigr\rangle_{\partial D}.
\end{equation}
Here $\langle f, g \rangle_{\partial D} = \int_{\partial D} f g \, {\rm d} \sigma$ denotes the bilinear duality pairing on $\partial D$. 

For simplicity, in what follows, we omit $\delta_{\partial D}$, and write  $\mc C_{(i,r),(j,\ell)}(\omega_0, \delta_{\partial D})$ as $\mc C_{(i,r),(j,\ell)}(\omega_0)$. 
\end{definition}

\begin{remark}
This is the non-subwavelength analogue of the abstract capacitance-matrix reduction in the subwavelength regime \cite{ammari2025mathematical}. The traces $g_{j,\ell}$ of the interior Neumann eigenmodes play the role of the constant boundary values $\chi_{\partial D_j}$ from the subwavelength theory, while the outgoing exterior solutions $V_{j,\ell}$ replace the harmonic capacitary potentials. 
\end{remark}

Equivalently, since $g_{i,r}$ is supported on $\partial D_i$ by \eqref{eq:tracegu}, one can also write 
\begin{equation}\label{eq:def-freq-cap-general-component-updated}
\mc C_{(i,r),(j,\ell)}(\omega_0)
=
-\frac{\delta_i v_i^2}{2\omega_0}
\int_{\partial D_i} V_{i,r}\,\partial_\nu V_{j,\ell}\big|_+\,\dd\sigma, 
\end{equation}
due to $u_{i,r} = g_{i,r} = V_{i,r}$ on $\partial D_i$. We will see that the above matrix appears naturally by applying the finite-rank operator $\mc P_{\omega_0}$ to the perturbation term. 

\begin{lemma}\label{lem:PL-on-kernel-basis}
For each $(j,\ell)\in \mc I$, one has
\begin{equation*}
    \L(\omega_0,\delta_{\partial D})\Psi_{j,\ell}
= \binom{0}{-\delta_{\partial D}\,\partial_\nu V_{j,\ell}|_+},
\end{equation*}
and consequently,
\begin{equation}\label{eq:PL-on-kernel-basis-explicit}
\mc P_{\omega_0}\L(\omega_0,\delta_{\partial D})\Psi_{j,\ell}
=
\sum_{(i,r)\in\mc I}
\mc C_{(i,r),(j,\ell)}(\omega_0)\,\Psi_{i,r}.
\end{equation}
\end{lemma}

\begin{proof}
By the definitions of $\L(\omega,\delta_{\partial D})$ and $\Psi_{j,\ell}$ in \eqref{def:lomega} and \eqref{eq:field}, respectively, we have 
\begin{equation*}
    \L(\omega_0,\delta_{\partial D})\Psi_{j,\ell}
=
\binom{0}{-\delta_{\partial D}\Bigl(\dfrac12 I+\mc K_D^{k_0,*}\Bigr)(\S_D^{k_0})^{-1}[g_{j,\ell}]}
=
\binom{0}{-\delta_{\partial D}\,\partial_\nu V_{j,\ell}|_+}.
\end{equation*}
Applying \eqref{eq:Pw0-rankone-reduction} with $f=0$ and $g=-\delta_{\partial D}\,\partial_\nu V_{j,\ell}|_+$ gives
\begin{equation*}
    \mc P_{\omega_0}\L(\omega_0,\delta_{\partial D})\Psi_{j,\ell}
=
-\sum_{i\in\mc J}\frac{v_i^2}{2\omega_0}\sum_{r=1}^{m_i}
\Bigl\langle \delta_{\partial D}\,\partial_\nu V_{j,\ell}|_+,g_{i,r}\Bigr\rangle_{\partial D}\,\Psi_{i,r},
\end{equation*}
which is exactly \eqref{eq:PL-on-kernel-basis-explicit} by Definition~\ref{def:freq-cap-matrix-general-updated}.
\end{proof}

\smallsection{Reduction to a finite-dimensional matrix pencil}
We now apply the pole-pencil expansion of $\A_0(\omega)^{-1}$ to the
characteristic value problem for $\A(\omega,\delta_{\partial D})$ and derive asymptotic expansions of the resonant frequencies of \eqref{eq:multi-resonator-scattering-problem} and their corresponding eigenmodes as $\delta = \|\delta_{\partial D}\|_{L^\infty(\partial D)}$ goes to zero.  

\begin{theorem}\label{prop:finite-dimensional-pencil-reduction}
Suppose that the assumptions of \cref{prop:A0-pole-pencil-reduction} hold in a neighborhood of the nonzero characteristic value $\omega_0$ of $\A_0$. Then there exists an $m\times m$ matrix-valued function $\mathsf F(\omega,\delta_{\partial D})$, analytic jointly in $\omega$ and $(\delta_1,\ldots,\delta_N)$ near $(\omega_0,0)$, such that the characteristic values of $\A(\omega,\delta_{\partial D})$ near $\omega_0$ are exactly the zeros of
\begin{equation*}
    \det \mathsf F(\omega,\delta_{\partial D})=0,
\end{equation*}
with the same algebraic multiplicities, and we have 
\begin{equation*}
    \mathsf F(\omega,\delta_{\partial D})  = (\omega-\omega_0) I_m - \mc C(\omega_0) + \O\bigl(|\omega-\omega_0|^2+\delta|\omega-\omega_0|+\delta^2\bigr).
\end{equation*}
In particular, there exist constants $R>0$ and $\delta_0>0$ such that, for every contrast profile satisfying
$0<\delta=\|\delta_{\partial D}\|_{L^\infty(\partial D)}<\delta_0$, the disk
\begin{equation*}
    \bigl\{\omega\in\C:|\omega-\omega_0|<R\delta\bigr\}
\end{equation*}
contains exactly $m$ characteristic values of $\A(\mathord\cdot,\delta_{\partial D})$, counted with algebraic multiplicity, and these exhaust the characteristic values converging to $\omega_0$. Consequently, every characteristic value converging to $\omega_0$ satisfies $|\omega-\omega_0|=\O(\delta)$, and the above expansion reduces to $ \mathsf F(\omega,\delta_{\partial D}) = (\omega - \omega_0) I_m-\mc C(\omega_0)+\O(\delta^2)$. Hence, the leading-order characteristic equation is 
\begin{equation*}
    \det\bigl((\omega - \omega_0) I_m-\mc C(\omega_0)\bigr)=0.
\end{equation*}

For any such characteristic value $\omega$ and any $X=\binom{\psi}{\phi}\neq0$ satisfying
\begin{equation*}
    \A(\omega,\delta_{\partial D})X=0,
\end{equation*}
let $\mc X_\perp$ be the orthogonal complement of $\ker\A_0(\omega_0)$ in $L^2(\partial D)\times L^2(\partial D)$ and write
\begin{equation} \label{eq:decompositionofx}
    X=\sum_{(j,\ell)\in\mc I}a_{j,\ell}\Psi_{j,\ell}+X^\perp,
    \qquad X^\perp\in\mc X_\perp,
\end{equation}
where
\begin{equation*}
    a:=\bigl(a_{j,\ell}\bigr)_{(j,\ell)\in\mc I}\in\C^m.
\end{equation*}
Then $a\neq0$ and
\begin{equation}\label{eq:complement-estimate-finite-pencil}
    X^\perp=\O(\delta\|a\|_2),
\end{equation}
as well as
\begin{equation}\label{eq:finite-pencil-reduction}
    \bigl((\omega-\omega_0)I_m-\mc C(\omega_0)\bigr)a
    =\O(\delta^2\|a\|_2),
\end{equation}
where $\|a\|_2$ denotes the $\ell^2$ norm of $a$.

If $\lambda(\omega_0)$ is an eigenvalue of $\mc C(\omega_0)$ of algebraic multiplicity $p$ whose largest Jordan block has size $q$, then exactly $p$ corresponding characteristic values, counted with algebraic multiplicity, satsify
\begin{equation} \label{eq:approximation_omega}
    \omega= \underbrace{\omega_0 + \lambda(\omega_0)}_{:= \omega^{\mathrm{lead}}} + \O(\delta^{1+1/q}).
\end{equation}
In the semisimple case, \emph{i.e.}, $q=1$, this reduces to the sharper expansion
\begin{equation*}
        \omega = \omega_0 + \lambda(\omega_0) + \O(\delta^2).
\end{equation*}
\end{theorem}

\begin{proof}
We recall estimates
\begin{equation} \label{est:lomegadelta}
    \L(\omega,\delta_{\partial D})=\O(\delta),
    \qquad
    \L(\omega,\delta_{\partial D})
    =
    \L(\omega_0,\delta_{\partial D})
    +\O(\delta|\omega-\omega_0|),
\end{equation}
uniformly for $\omega$ near $\omega_0$. We first construct an exact finite-dimensional reduction. Define
\begin{equation*}
    \mathscr X:=L^2(\partial D)\times L^2(\partial D),
    \qquad
    \mathscr Y:=H^1(\partial D)\times L^2(\partial D),
\end{equation*}
and the mapping 
\begin{equation*}
\iota_{\omega_0}:\C^m\longrightarrow\mathscr X,
    \qquad
    \iota_{\omega_0}a
    :=
    \sum_{(j,\ell)\in\mc I}a_{j,\ell}\Psi_{j,\ell}.
\end{equation*}
We also define $\gamma_{\omega_0}:\mathscr Y\to\C^m$ by
\begin{equation*}
    \left(
    \gamma_{\omega_0}\binom{f}{g}
    \right)_{j,\ell}
    :=
    \frac{v_j^2}{2\omega_0}
    \langle g,g_{j,\ell}\rangle_{\partial D}.
\end{equation*}
Then \eqref{eq:Pw0-rankone-reduction} gives the factorization
\begin{equation*}
    \mc P_{\omega_0}
    = \iota_{\omega_0}\gamma_{\omega_0}.
\end{equation*}
For simplicity, we write 
\begin{equation*}
    A_{00}:=\A_0(\omega_0),
    \qquad A_{01}:=\partial_\omega\A_0(\omega_0),
    \qquad
    R_{00}:=\mc R_0(\omega_0).
\end{equation*}
Comparing the coefficients of $(\omega-\omega_0)^{-1}$ and the constant terms in
\begin{equation*}
    \A_0(\omega)^{-1}\A_0(\omega)=I_{\mathscr X},
    \qquad
    \A_0(\omega)\A_0(\omega)^{-1}=I_{\mathscr Y},
\end{equation*}
and using \eqref{eq:A0-pole-pencil-reduction}, we obtain the Laurent identities
\begin{align*}
    \mc P_{\omega_0}A_{00}&=0,
    &-\mc P_{\omega_0}A_{01}+R_{00}A_{00}&=I_{\mathscr X},\\
    A_{00}\mc P_{\omega_0}&=0,
    &-A_{01}\mc P_{\omega_0}+A_{00}R_{00}&=I_{\mathscr Y}.
\end{align*}
These identities imply $\operatorname{ran}\mc P_{\omega_0} =\ker A_{00}=\operatorname{ran}\iota_{\omega_0}$. Hence, $\gamma_{\omega_0}$ is surjective, and the factorization
$\mc P_{\omega_0}=\iota_{\omega_0}\gamma_{\omega_0}$ also gives
\begin{equation} \label{eq:laurent-identity}
    \gamma_{\omega_0}A_{00}=0,
    \qquad
    \gamma_{\omega_0}A_{01}\iota_{\omega_0}=-I_m.
\end{equation}

Let $R_+:\mathscr X\to\C^m$ be the coefficient map associated with \eqref{eq:decompositionofx}, namely,
\begin{equation*}
    R_+\bigl(\iota_{\omega_0}a+X^\perp\bigr)=a,
\end{equation*}
and set $R_-:=-A_{01}\iota_{\omega_0}  :\C^m\rightarrow\mathscr Y$.
Then $R_+\iota_{\omega_0}=I_m$, $\ker R_+=\mc X_\perp$, and $\gamma_{\omega_0}R_-=I_m$. Moreover,
\begin{equation*}
    \operatorname{ran}A_{00}=\ker\gamma_{\omega_0}.
\end{equation*}
Indeed, the inclusion from left to right follows from
\eqref{eq:laurent-identity}. Conversely, if $Y\in\ker\gamma_{\omega_0}$,
then $\mc P_{\omega_0}Y=0$, and the last Laurent identity gives
$Y=A_{00}R_{00}Y$. It follows that
\begin{equation*}
    A_{00}\big|_{\ker R_+}:\ker R_+\rightarrow
    \ker\gamma_{\omega_0}
\end{equation*}
is an isomorphism. For surjectivity, one may take, for $Y\in\ker\gamma_{\omega_0}$,
\begin{equation*}
    R_{00}Y-\iota_{\omega_0}R_+R_{00}Y\in\ker R_+,
\end{equation*}
whose image under $A_{00}$ is $Y$; injectivity follows from $\ker A_{00}=\operatorname{ran}\iota_{\omega_0}$ and $R_+\iota_{\omega_0}=I_m$. Denote the inverse of this restriction by $B_0$. Consequently, the Grushin operator
\begin{equation*}
    \mathsf G(\omega,\delta_{\partial D})
    :=
    \begin{pmatrix}
        \A(\omega,\delta_{\partial D}) & R_-\\
        R_+ & 0
    \end{pmatrix}
    :
    \mathscr X\times\C^m
    \longrightarrow
    \mathscr Y\times\C^m
\end{equation*}
is invertible at $(\omega_0,0)$, with inverse
\begin{equation*}
    \mathsf G(\omega_0,0)^{-1}
    =
    \begin{pmatrix}
        B_0(I_{\mathscr Y}-R_-\gamma_{\omega_0}) & \iota_{\omega_0}\\[1mm]
        \gamma_{\omega_0} & 0
    \end{pmatrix}.
\end{equation*}
It therefore remains analytically invertible in a neighborhood of this point. Introduce 
\begin{equation*}
    \mathsf G(\omega,\delta_{\partial D})^{-1}
    =
    \begin{pmatrix}
        E & E_+\\
        E_- & \mathsf F
    \end{pmatrix},
\end{equation*}
where
\begin{equation*}
    E_+(\omega_0,0)=\iota_{\omega_0},
    \qquad
    E_-(\omega_0,0)=\gamma_{\omega_0},
    \qquad
    \mathsf F(\omega_0,0)=0.
\end{equation*}
For $h\in\C$ and $\eta=(\eta_1,\ldots,\eta_N)\in\C^N$, let $\eta_{\partial D}|_{\partial D_j}=\eta_j$. Then
\begin{equation*}
    D\A(\omega_0,0)[h,\eta]
    =hA_{01}+\L(\omega_0,\eta_{\partial D}).
\end{equation*}
Differentiating $\mathsf G^{-1}\mathsf G=I$ in the direction $(h,\eta)$ gives
\begin{equation*}
    D(\mathsf G^{-1})(\omega_0,0)[h,\eta]
    =-\mathsf G(\omega_0,0)^{-1}
    D\mathsf G(\omega_0,0)[h,\eta]
    \mathsf G(\omega_0,0)^{-1}.
\end{equation*}
Taking the lower-right block and using \eqref{eq:laurent-identity} and \cref{lem:PL-on-kernel-basis}, we obtain
\begin{equation*}
    D\mathsf F(\omega_0,0)[h,\eta]
    =-\gamma_{\omega_0}
    \bigl(hA_{01}+\L(\omega_0,\eta_{\partial D})\bigr)
    \iota_{\omega_0}
    =hI_m-\mc C(\omega_0,\eta_{\partial D}).
\end{equation*}
Taylor expansion therefore yields
\begin{equation*}
    \mathsf F(\omega,\delta_{\partial D})
    = (\omega-\omega_0)I_m-\mc C(\omega_0)
    + \O\bigl(|\omega-\omega_0|^2
        +\delta|\omega-\omega_0|
        +\delta^2
    \bigr).
\end{equation*}

The block identities for $\mathsf G^{-1}\mathsf G=I$ give the analytic factorization
\begin{equation*}
\mathsf G^{-1}
\begin{pmatrix}
    \A & 0\\
    0 & I_m
\end{pmatrix}
=
\begin{pmatrix}
    I_{\mathscr X} & 0\\
    0 & \mathsf F
\end{pmatrix}
\begin{pmatrix}
    I_{\mathscr X} & E_+\\
    0 & I_m
\end{pmatrix}
\begin{pmatrix}
    I_{\mathscr X} & 0\\
    -R_+ & I_m
\end{pmatrix}.
\end{equation*}
Here and below, the dependence on $(\omega,\delta_{\partial D})$ is suppressed. The Grushin operator and the two triangular factors in this identity are analytic and invertible. Thus, by the analytic equivalence underlying the Grushin reduction \cite{gohberg1971operator}, the characteristic values of $\A$ near $\omega_0$ are exactly the zeros of $\det\mathsf F$, with the same algebraic multiplicities.

We next count the characteristic values. Since $\|\mc C(\omega_0)\|\leq c_0\delta$, fix $R>c_0$ and set $\Gamma_\delta:=\{z\in\C:|z|=R\delta\}$. On $\Gamma_\delta$, with $z=\omega-\omega_0$,
\begin{equation*}
    \bigl\|\bigl(zI_m-\mc C(\omega_0)\bigr)^{-1}\bigr\|
    \leq \frac{1}{(R-c_0)\delta},
    \qquad
    \bigl\|\mathsf F(\omega_0+z,\delta_{\partial D})
    -(zI_m-\mc C(\omega_0))\bigr\|
    \leq C_R\delta^2.
\end{equation*}
Consequently,
\begin{equation*}
    \bigl\|(zI_m-\mc C(\omega_0))^{-1}
    [\mathsf F(\omega_0+z,\delta_{\partial D})
    -(zI_m-\mc C(\omega_0))]\bigr\|
    \leq \frac{C_R\delta}{R-c_0}<1
\end{equation*}
for $\delta$ sufficiently small. The finite-dimensional Rouch\'e theorem
therefore shows that $\det\mathsf F$ and $\det(zI_m-\mc C(\omega_0))$ have the same number
$m$ of zeros inside $\Gamma_\delta$, counted with algebraic multiplicity. Since $\mathsf F(\omega_0+z,0)=zI_m+\O(z^2)$, we may choose a fixed $r_0>0$ inside the Grushin neighborhood such that 
\begin{equation*}
    \bigl\|z^{-1}\mathsf F(\omega_0+z,0)-I_m\bigr\|<1,
    \qquad |z|=r_0.
\end{equation*}
The matrix Rouch\'e theorem gives exactly $m$ zeros in $|z|<r_0$ at $\delta_{\partial D}=0$, and uniform convergence on this fixed circle gives the same count for all sufficiently small $\delta$. Hence the annulus $R\delta\leq|z|<r_0$ contains no characteristic values, so the $R\delta$-disk exhausts those converging to $\omega_0$ and
\begin{equation}\label{eq:omega-close-to-omega0}
    |\omega-\omega_0|=\O(\delta).
\end{equation}

Let $X\neq0$ be a characteristic vector for any one of these values and set $a:=R_+X$. Then
\begin{equation*}
    \mathsf G(\omega,\delta_{\partial D})
    \binom{X}{0}=\binom{0}{a}.
\end{equation*}
Applying $\mathsf G^{-1}$ gives
\begin{equation*}
    X=E_+a,
    \qquad
    \mathsf F(\omega,\delta_{\partial D})a=0.
\end{equation*}
In particular, $a\neq0$. The lower-right block of $\mathsf G\mathsf G^{-1}=I$ gives $R_+E_+=I_m$, while analyticity gives, in $\mathcal L(\C^m,\mathscr X)$,
\begin{equation*}
    E_+(\omega,\delta_{\partial D})  =\iota_{\omega_0}+\O(|\omega-\omega_0|+\delta)
    =\iota_{\omega_0}+\O(\delta).
\end{equation*}
Therefore,
\begin{equation*}
    X^\perp  =X-\iota_{\omega_0}a
    =(E_+-\iota_{\omega_0})a = \O(\delta\|a\|_2),
\end{equation*}
which proves \eqref{eq:complement-estimate-finite-pencil}. Using the Taylor expansion of $\mathsf F$ and the identity $\mathsf Fa=0$, we obtain \eqref{eq:finite-pencil-reduction}:
\begin{equation}\label{est:lomegadelta2}
    \bigl((\omega-\omega_0)I_m-\mc C(\omega_0)\bigr)a
    =\O(\delta^2\|a\|_2). 
\end{equation}

It remains to prove the cluster refinement. The Jordan resolvent estimate, after scaling by $\delta$, gives constant $C_J, K_J>0$ independent of $\delta$ such that
\begin{equation*}
    \bigl\|(zI_m-\mc C(\omega_0))^{-1}\bigr\|
    \leq
    K_J\frac{\delta^{q-1}}
    {|z-\lambda(\omega_0)|^q}
\end{equation*}
whenever $0<|z-\lambda(\omega_0)|\leq C_J \delta/2$. On the circle
\begin{equation*}
    |z-\lambda(\omega_0)|=M\delta^{1+1/q},
\end{equation*}
the last bound is $K_J/(M^q\delta^2)$, whereas the Taylor remainder in $\mathsf F$ is bounded by $K_E\delta^2$. Choose $M$ so that $K_JK_E/M^q<1$, and then take $\delta$ sufficiently small that $M\delta^{1/q} < C_J/2$. The matrix Rouch\'e theorem shows that the disk bounded by this circle contains exactly $p$ zeros of $\det\mathsf F$, counted with algebraic multiplicity. This proves \eqref{eq:approximation_omega}; for $q=1$, its radius is $M\delta^2$, giving the stated semisimple estimate.
\end{proof}

From \eqref{eq:finite-pencil-reduction}, we deduce the leading-order
structure of the corresponding eigenmodes.

\begin{corollary}\label{cor:eigenmodes}
Let $\omega$ be any characteristic value in the local cluster described in \cref{prop:finite-dimensional-pencil-reduction}, and let $X(\delta)=(\psi,\phi)^\top\neq 0$ be a density pair with the decomposition \eqref{eq:decompositionofx}, where
$a(\delta)=(a_{j,\ell}(\delta))_{(j,\ell)\in\mc I}$ satisfies the normalization $\|a(\delta)\|_2=1$. Then the associated eigenmode of \eqref{eq:multi-resonator-scattering-problem} satisfies
\begin{equation*}
u_\delta(x) =
\sum_{(j,\ell)\in\mc I}a_{j,\ell}(\delta)\,
\begin{cases}
V_{j,\ell}(x),       & x\in\R^3\setminus\overline D,\\[2pt]
u_{j,\ell}(x),       & x\in D_j,\\[2pt]
0,                    & x\in D_i,\quad i\neq j,
\end{cases}
+\O(\delta \|a(\delta)\|_2)
\end{equation*}
in $H^1_{\rm loc}(\R^3)$, where $V_{j,\ell}$ is defined in \eqref{eq:def-Vjl-updated}.  

Moreover, suppose that $\omega$ belongs to the cluster in \eqref{eq:approximation_omega} associated with $\lambda(\omega_0)$, whose largest Jordan block has size $q$. Then there exists
\begin{equation*}
    a^{(0)}(\delta)\in
    \ker\bigl(\mc C(\omega_0)-\lambda(\omega_0)I_m\bigr)
\end{equation*}
such that $a(\delta)=a^{(0)}(\delta)+\O(\delta^{1/q})$. Consequently,
\begin{equation} \label{eq:fieldexpansion}
u_\delta(x)
=
\sum_{(j,\ell)\in\mc I}a^{(0)}_{j,\ell}(\delta)\,
\begin{cases}
V_{j,\ell}(x),       & x\in\R^3\setminus\overline D,\\[2pt]
u_{j,\ell}(x),       & x\in D_j,\\[2pt]
0,                    & x\in D_i,\quad i\neq j,
\end{cases}
+\O(\delta^{1/q})
\end{equation}
in $H^1_{\rm loc}(\R^3)$. If this eigenspace is one-dimensional, then, after choosing the phase of $X(\delta)$, one may replace $a^{(0)}(\delta)$ by a fixed unit eigenvector of $\mc C(\omega_0)$.
\end{corollary}

\begin{proof}
By the layer-potential ansatz \eqref{eq:multi-resonator-layer-ansatz}, the eigenmode $u_\delta$ at frequency $\omega$ is parametrized by a density pair $X(\delta)=(\psi,\phi)^\top\in L^2(\partial D)^2$ satisfying
$\mc A(\omega,\delta_{\partial D})X(\delta)=0$. Recalling the
decomposition \eqref{eq:decompositionofx} of $X(\delta)$, the estimate
\eqref{eq:complement-estimate-finite-pencil} gives
\begin{equation}\label{eq:densities-leading-order}
\psi
=
\sum_{(j,\ell)\in\mc I}a_{j,\ell}(\delta)\,\psi_{j,\ell}
+\O(\delta \|a(\delta)\|_2),
\qquad
\phi
=
\sum_{(j,\ell)\in\mc I}a_{j,\ell}(\delta)\,\phi_{j,\ell}
+\O(\delta \|a(\delta)\|_2),
\end{equation}
in $L^2(\partial D)$. 
The single-layer operators $\mc S_D^k:L^2(\partial D)\to H^1_{\rm loc}(\R^3)$ and $\widetilde{\mc S}_D^\omega:L^2(\partial D)\to H^1(D)$ depend analytically on $\omega$ in a neighborhood of $\omega_0$. Combined with \eqref{eq:omega-close-to-omega0}, this gives the operator-norm estimates
\begin{equation*}
\mc S_D^k=\mc S_D^{k_0}+\O(\delta),
\qquad
\widetilde{\mc S}_D^\omega=\widetilde{\mc S}_D^{\omega_0}+\O(\delta).
\end{equation*}
Substituting \eqref{eq:densities-leading-order} into the ansatz \eqref{eq:multi-resonator-layer-ansatz} and using the linearity and boundedness of $\mc S_D^{k_0}$ and $\widetilde{\mc S}_D^{\omega_0}$, we obtain
\begin{equation*}
u_\delta(x)
=
\sum_{(j,\ell)\in\mc I}a_{j,\ell}(\delta)\,
\begin{cases}
\mc S_D^{k_0}[\phi_{j,\ell}](x),                  & x\in\R^3\setminus\overline D,\\[2pt]
\widetilde{\mc S}_D^{\omega_0}[\psi_{j,\ell}](x), & x\in D,
\end{cases}
+\O(\delta \|a(\delta)\|_2),
\end{equation*}
with the error measured in $H^1_{\rm loc}(\R^3)$. By \eqref{eq:field2} and the definition of $V_{j,\ell}$ in \eqref{eq:def-Vjl-updated}, $\mc S_D^{k_0}[\phi_{j,\ell}]=V_{j,\ell}$ in $\R^3\setminus\overline D$. Moreover, $\widetilde{\mc S}_D^{\omega_0}[\psi_{j,\ell}]$ has boundary trace $g_{j,\ell}$; hence \eqref{eq:tracegu} and uniqueness of the interior Dirichlet problem give
\begin{equation*}
\widetilde{\mc S}_D^{\omega_0}[\psi_{j,\ell}](x)
=
\begin{cases}
u_{j,\ell}(x), & x\in D_j,\\[2pt]
0,              & x\in D_i,\quad i\neq j.
\end{cases}
\end{equation*}
If $\omega(\delta)$ is associated with $\lambda(\omega_0)$, then \eqref{eq:finite-pencil-reduction} and \eqref{eq:approximation_omega} give
\begin{equation*}
    \bigl(\mc C(\omega_0)-\lambda(\omega_0)I_m\bigr)a(\delta)
    =\O(\delta^{1+1/q}).
\end{equation*}
Let $P_\lambda$ be the orthogonal projection onto the kernel of
$\mc C(\omega_0)-\lambda(\omega_0)I_m$. Since $\delta^{-1}(\mc C(\omega_0)-\lambda(\omega_0)I_m)$ is fixed, its restriction to the orthogonal complement of its kernel is bounded below. Hence
\begin{equation*}
    \|(I_m-P_\lambda)a(\delta)\|_2
    \lesssim
    \delta^{-1}\|\bigl(\mc C(\omega_0)
    -\lambda(\omega_0)I_m\bigr)a(\delta)\|_2
    =\O(\delta^{1/q}).
\end{equation*}
Taking $a^{(0)}(\delta):=P_\lambda a(\delta)$ and replacing $a_{j,\ell}(\delta)$ by $a^{(0)}_{j,\ell}(\delta)$ in the preceding field expansion proves \eqref{eq:fieldexpansion}. The proof is complete.
\end{proof}

\begin{remark}
The approximations of the non-subwavelength resonant frequencies remain valid for complex material parameters and can therefore be used to detect exceptional-point degeneracies in parity-time symmetric systems; see \cite{ammari2020exceptional}. The frequency-dependent capacitance matrix formulation also extends to systems of resonators with an imaginary gauge potential, where it captures the non-Hermitian skin effect beyond the subwavelength regime; see \cite{skin3d, alex}.
\end{remark}

\begin{remark}
In three dimensions, the leading-order subwavelength resonant frequencies are of order $\sqrt{\delta}$, that is, $\omega_{\rm sub}^{\rm lead}=\O(\sqrt{\delta})$. This differs from the non-subwavelength regime considered here, where \eqref{eq:approximation_omega} gives $\omega^{\rm lead}=\omega_0+\lambda(\omega_0)$ with $\lambda(\omega_0)=\O(\delta)$. Moreover, $\omega_{\rm sub}^{\rm lead}$ is real, whereas the leading-order non-subwavelength approximation is generally complex. This is because the radiation condition is negligible at leading order in the subwavelength regime, but contributes at leading order for non-subwavelength resonances \cite{alex}. 
\end{remark}

\subsection{General properties of the frequency-dependent capacitance matrix} \label{sec:general_properties}

In the subwavelength regime, $2 \omega_0 \mc C(\omega_0)|_{\omega_0 = 0}$ reduces to the classical generalized capacitance matrix, obtained as the product of the capacitance matrix $C$ with a diagonal matrix $V$ whose entries are the weights $\delta_i v_i^2/|D_i|$, \emph{i.e.}, $V C = 2 \omega_0 \mc C(\omega_0)|_{\omega_0 = 0}$.
The (subwavelength) capacitance matrix $C$ has several well-known properties: it is real symmetric positive definite, its off-diagonal entries are negative, and it is diagonally dominant \cite{ammari2025mathematical,ammari2024functional}. Moreover, the subwavelength resonant frequencies and the corresponding eigenmodes are approximated by the eigenvalues and eigenvectors of the generalized capacitance matrix $V C$. Note that 
when $\delta_i v_i^2$ are real and positive, $V C$ is always diagonalizable.

In this section, we discuss which of these properties extends to the non-subwavelength regime, \emph{i.e.},  at a nonzero reference frequency. Let $\omega_0\neq 0$ be a characteristic value of the limiting operator $\mc{A}_0(\omega)$, and recall that
\begin{equation*}
    \mc I=\{(j,\ell):j\in\mc J,\ 1\le \ell\le m_j\}.
\end{equation*}
For $\alpha=(j,\ell)\in\mc I$, write $g_\alpha:=g_{j,\ell}$. We define
\begin{equation*}
\mathscr C_{\alpha\beta}(\omega_0) := -\bigl\langle \Lambda_{\rm ext}^{\omega_0}[g_\beta],\,g_\alpha\bigr\rangle_{\partial D}, \qquad \alpha,\beta\in\mc I,
\end{equation*}
where $\Lambda_{\rm ext}^{\omega_0}$ denotes the exterior DtN map at frequency $\omega_0$. With this convention, the frequency-dependent capacitance matrix can be written as
\begin{equation}\label{eq:mcC-from-scrC}
   \mc C(\omega_0)=\frac{1}{2\omega_0}\,\widehat D\,\mathscr C(\omega_0),
\end{equation}
by \cref{def:freq-cap-matrix-general-updated} with the equation \eqref{eq:def-Vjl-updated}, where $\widehat D:=\diag(\delta_i v_i^2)$ is the diagonal matrix whose entry, corresponding to a basis element $g_\alpha$ with $\alpha=(i,r)$, is $\widehat D_{\alpha\alpha}=\delta_i v_i^2$. Hence, $\mathscr C(\omega_0):= 2\omega_0\,\widehat D^{-1}\,\mc C(\omega_0)$ can be viewed as a generalization of the capacitance matrix $C$ to the non-subwavelength regime.

\begin{proposition}\label{prop:complex}
The matrix $\mathscr C(\omega_0)$ is complex symmetric, that is,
\begin{equation*}
    \mathscr C(\omega_0)^\top=\mathscr C(\omega_0),
\end{equation*}
or equivalently,
\begin{equation*}
    \bigl\langle \Lambda_{\rm ext}^{\omega_0}[g_\beta],\,g_\alpha\bigr\rangle_{\partial D}
=
\bigl\langle \Lambda_{\rm ext}^{\omega_0}[g_\alpha],\,g_\beta\bigr\rangle_{\partial D}
\qquad \text{for all }\alpha,\beta\in\mc I.
\end{equation*}
Consequently, $\mc C(\omega_0)$ is similar to the complex symmetric matrix
$\frac{1}{2\omega_0}\,\widehat D^{1/2}\,\mathscr C(\omega_0)\,\widehat D^{1/2}$:
\begin{equation*}
    \mc C(\omega_0)
=\frac{1}{2\omega_0}\,\widehat D\,\mathscr C(\omega_0)
\sim
\frac{1}{2\omega_0}\,\widehat D^{1/2}\,\mathscr C(\omega_0)\,\widehat D^{1/2}.
\end{equation*}
\end{proposition}

\begin{proof}
For each $\alpha\in\mc I$, let $V_\alpha$ be the outgoing exterior solution of
\begin{equation*}
    (\Delta+k_0^2)V_\alpha=0 \quad \text{in }\R^3\setminus\overline D,
\qquad
V_\alpha|_{\partial D}=g_\alpha,
\end{equation*}
so that $\Lambda_{\rm ext}^{\omega_0}[g_\alpha]=\partial_\nu V_\alpha|_{\partial D}$
with $\nu$ pointing out of $D$ into the exterior. Choose $R>0$ large enough so that $\overline D\subset B_R$, where $B_R$ is the open ball of radius $R$ centered at the origin, and let $\Omega_R:=B_R\setminus\overline D$. Since $V_\alpha$ and
$V_\beta$ satisfy the same Helmholtz equation in $\Omega_R$, we have 
\begin{equation*}
\int_{\partial\Omega_R}\bigl(V_\alpha\,\partial_{\nu_R} V_\beta-V_\beta\,\partial_{\nu_R} V_\alpha\bigr)\, {\rm d} \sigma
=
\int_{\Omega_R}\bigl(V_\alpha\Delta V_\beta-V_\beta\Delta V_\alpha\bigr)\, {\rm d} x
=0,
\end{equation*}
where $\nu_R$ is the outward normal of $\Omega_R$. The radiation
condition \eqref{sommerfeld} for $V_\alpha$ and $V_\beta$ forces the integral on $\partial B_R$ to vanish as $R\to\infty$, while on $\partial D$
the relation $\nu_R=-\nu$ converts the remaining boundary integral into
\begin{equation*}
\int_{\partial D} g_\alpha\,\Lambda_{\rm ext}^{\omega_0}[g_\beta]\,\dd \sigma
=
\int_{\partial D} g_\beta\,\Lambda_{\rm ext}^{\omega_0}[g_\alpha]\,\dd \sigma,
\end{equation*}
which is the asserted symmetry.
\end{proof}

\Cref{prop:complex} shows that $\mathscr C(\omega_0)$ is complex symmetric but, in general, \emph{not} Hermitian: the outgoing radiation condition contributes in the leading order in $\delta$ and produces an imaginary part. As a consequence, the classical subwavelength properties of the capacitance matrix, such as positive definiteness, diagonal dominance, and more generally the Minkowski property \cite{ammari2025mathematical}, no longer hold beyond the subwavelength regime. Two mechanisms are responsible. First, radiation losses make $\mathscr C(\omega_0)$ (and hence $\mc C(\omega_0)$) typically complex-valued, as proved in \Cref{prop:complex}. Second, the basis traces of $g_\alpha$ are no longer nonnegative constants on each resonator, and the maximum principle fails for the interior Helmholtz operators $\Delta+\omega_0^2/v_i^2$ on $D_i$; entry-wise sign properties are therefore not basis-invariant. 

We next derive an estimate for the matrix norm of the
frequency-dependent capacitance matrix \eqref{eq:def-freq-cap-general-updated} for \emph{large enough real} $\omega > 0$ such that $\omega^2/v_i^2$ is a Neumann eigenvalue of $-\Delta$ in $D_i$, where $g_{i,r}$ denotes the boundary trace of an $L^2$-normalized Neumann eigenfunction. 

We begin with boundary-trace estimates for the normalized Neumann eigenfunctions of a single resonator. Throughout this subsection, the
notation $A\lesssim B$ means that $A\le CB$ for some constant $C>0$ depending
only on the geometry of the resonators.

\begin{lemma}\label{lem:sharp-trace}
Assume that each $D_i$ is of class $C^\infty$. Then, for every Neumann eigenfunction $u_{i,\ell}(\omega,\cdot)$ of \eqref{eq:sharp-neumann-eig}
normalized by $\|u_{i,\ell}(\omega,\cdot)\|_{L^2(D_i)}=1$, its boundary trace
$g_{i,\ell}(\omega):= u_{i,\ell}(\omega,\cdot)\big|_{\partial D_i}$ satisfies, for $ s\in\{0,\tfrac12,1\}$,
\begin{equation} \label{est:trace}
    \|g_{i,\ell}(\omega)\|_{H^{s}(\partial D_i)}
\lesssim \Bigl(1+\frac{\omega}{v_i}\Bigr)^{s+\frac{1}{2}},
\end{equation}
uniformly in the Neumann eigenfrequency $\omega > 0$.
\end{lemma}

\begin{proof}
Let $u=u_{i,\ell}(\omega,\cdot)$ and $\lambda_i:=\omega^2/v_i^2$ be the corresponding Neumann eigenvalue of $-\Delta$ on $D_i$. Since $-\Delta u=\lambda_i u$ in $D_i$ with $\partial_\nu u=0$ on $\partial D_i$,
Green's identity gives
\begin{equation*}
    \|\nabla u\|_{L^2(D_i)}^2
= \int_{D_i}(-\Delta u)\,\overline u\, \dd x
- \int_{\partial D_i}(\partial_\nu u)\,\overline u\, \dd \sigma
= \lambda_i,
\end{equation*}
where the boundary term vanishes by the Neumann condition and the last
equality uses $\|u\|_{L^2(D_i)}=1$. Hence, we find
\begin{equation}\label{eq:H1-eigfn-bound}
    \|u\|_{H^1(D_i)}\lesssim (1+\lambda_i)^{1/2}\lesssim 1+\frac{\omega}{v_i}.
\end{equation}
By the standard trace theorem and \eqref{eq:H1-eigfn-bound}, it follows that 
\begin{equation*}
    \|g_{i,\ell}(\omega)\|_{H^{1/2}(\partial D_i)}
\lesssim \|u\|_{H^1(D_i)}
\lesssim 1+\frac{\omega}{v_i}.
\end{equation*}

The $\varepsilon$-trace inequality gives, for every $\varepsilon\in(0,1]$,
\begin{equation*}
    \|g_{i,\ell}(\omega)\|_{L^2(\partial D_i)}^2
\le \varepsilon\,\|\nabla u\|_{L^2(D_i)}^2
+ C(\varepsilon)\,\|u\|_{L^2(D_i)}^2,
\qquad
C(\varepsilon)\lesssim 1+\varepsilon^{-1}.
\end{equation*}
Choosing $\varepsilon=(1+\lambda_i)^{-1/2}$ and using
$\|\nabla u\|_{L^2(D_i)}^2=\lambda_i$ and $\|u\|_{L^2(D_i)}=1$, we obtain
\begin{equation*}
    \|g_{i,\ell}(\omega)\|_{L^2(\partial D_i)}^2
\lesssim (1+\lambda_i)^{1/2},
\end{equation*}
and hence
\begin{equation*}
    \|g_{i,\ell}(\omega)\|_{L^2(\partial D_i)}
\lesssim (1+\lambda_i)^{1/4}
\lesssim \Bigl(1+\frac{\omega}{v_i}\Bigr)^{1/2}.
\end{equation*}

Since $\partial_\nu u=0$ on $\partial D_i$ and $D_i$ is smooth, elliptic
regularity for the Neumann problem yields
\begin{equation*}
    \|u\|_{H^2(D_i)}
\lesssim \|\Delta u\|_{L^2(D_i)}+\|u\|_{H^1(D_i)}
\lesssim 1+\lambda_i,
\end{equation*}
which further gives, by the trace theorem,
\begin{equation*}
    \|g_{i,\ell}(\omega)\|_{H^{3/2}(\partial D_i)}
\lesssim \|u\|_{H^2(D_i)}
\lesssim 1+\lambda_i.
\end{equation*}
Interpolating between $H^{1/2}(\partial D_i)$ and $H^{3/2}(\partial D_i)$ \cite{mclean2000strongly}, we have 
\begin{equation*}
    \|g_{i,\ell}(\omega)\|_{H^1(\partial D_i)}
\lesssim \|g_{i,\ell}(\omega)\|_{H^{1/2}(\partial D_i)}^{1/2}\,
        \|g_{i,\ell}(\omega)\|_{H^{3/2}(\partial D_i)}^{1/2}
\lesssim (1+\lambda_i)^{3/4}
\lesssim \Bigl(1+\frac{\omega}{v_i}\Bigr)^{3/2}.
\end{equation*}
This completes the proof.
\end{proof}
\begin{remark}
The estimates in \Cref{lem:sharp-trace} are not sharp. In particular, Tataru's sharp trace estimate \cite{tataru} for solutions of the Helmholtz equation on smooth domains gives
\begin{equation*}
    \|g_{i,\ell}(\omega)\|_{L^2(\partial D_i)}
\lesssim \Bigl(1+\frac{\omega}{v_i}\Bigr)^{1/3},
\end{equation*}
improving the estimate \eqref{est:trace} with $s = 0$; see also \cite{barnett1,barnett2} for related results. In particular, on the unit ball, a direct computation,  using whispering-gallery modes, indicates the sharp scaling
\begin{equation*}
    \|g_{i,\ell}(\omega)\|_{H^s(\partial D_i)}
\lesssim \Bigl(1+\frac{\omega}{v_i}\Bigr)^{1/3+s},
\qquad s\in\{0,\tfrac12,1\},
\end{equation*}
which is expected to hold in a more general setting. 
\end{remark}

The next lemma gives an $L^2$-estimate for the exterior DtN map at high frequencies. To obtain an explicit bound on the wavenumber dependence, we need to assume that the number of resonators $N=1$ and that the exterior $\R^3\setminus\overline D$ is \emph{nontrapping} \cite{spence1,spence2}. Physically, the nontrapping condition states that incident rays illuminating and reflecting on $D$ according to the law of geometric optics exit any bounded set in finite time, which is the case, for instance, if $D$ is convex. The following estimate follows from \cite[Theorem 1.4]{spence1}; see also \cite[Lemma 4.2]{spence3}.

\begin{lemma}\label{lem:dtn-l2}
Assume that $N = 1$, and $D$ is of class $C^\infty$ with nontrapping $\R^3\setminus\overline D$. Then, for any $f\in H^1(\partial D)$ and
$\omega>0$, we have 
\begin{equation*}
    \|\Lambda_{\rm ext}^{\omega}f\|_{L^2(\partial D)}
\lesssim \|f\|_{H^1(\partial D)}+\left(1+\frac{\omega}{v}\right)\|f\|_{L^2(\partial D)}.
\end{equation*}
\end{lemma}

We are now in a position to derive a high-frequency spectral norm bound for the frequency-dependent capacitance matrix.

\begin{proposition}\label{prop:sharp-matrix}
Under the assumptions of \Cref{lem:dtn-l2}, let $m = |\mc{I}|$ be defined as in \eqref{eq:dimensionIm}, and recall the decomposition \eqref{eq:mcC-from-scrC}. Then, for any Neumann eigenfrequency $\omega > 0$, 
\begin{equation}\label{eq:scrC-bound}
    \|\mathscr C(\omega)\|_{\ell^2(\C^m) \to \ell^2(\C^m)}
\lesssim m \Bigl(1+\frac{\omega}{v_{\min}}\Bigr)^2,
\qquad v_{\min}:=\min\{v,v_1\}.
\end{equation}
Consequently, we have
\begin{equation}\label{eq:mcC-bound}
    \|\mc C(\omega)\|_{\ell^2(\C^m) \to \ell^2(\C^m)}
\lesssim m \frac{\|\widehat D\|_{\ell^2(\C^m) \to \ell^2(\C^m)}}{2\omega}\Bigl(1+\frac{\omega}{v_{\min}}\Bigr)^2,
\end{equation}
and, in particular,
\begin{equation}\label{eq:mcC-high-freq}
    \|\mc C(\omega)\|_{\ell^2(\C^m) \to \ell^2(\C^m)} \lesssim \delta\,\omega\, m , 
\qquad\text{as }\omega\to\infty.
\end{equation}
\end{proposition}

\begin{proof}
For $a=(a_1,\dots,a_m)^\top\in\mathbb C^m$, set
$g_a(\omega):=\sum_{\beta \in \mc{I}} a_\beta\,g_\beta(\omega)\in H^1(\partial D)$,
and analogously $g_b(\omega)$ for $b\in\mathbb C^m$. By the definition of $\mathscr C(\omega)$, we have 
\begin{equation*}
    b^\top\mathscr C(\omega)\,a
= -\bigl\langle \Lambda_{\rm ext}^{\omega}[g_a(\omega)],\,g_b(\omega)\bigr\rangle_{\partial D},
\end{equation*}
and Cauchy's inequality on $L^2(\partial D)$ gives
\begin{equation}\label{eq:scrC-pairing-bound}
    |b^\top\mathscr C(\omega)\,a|
\le \|\Lambda_{\rm ext}^{\omega}[g_a(\omega)]\|_{L^2(\partial D)}\,
    \|g_b(\omega)\|_{L^2(\partial D)}.
\end{equation}

\Cref{lem:dtn-l2} gives
\begin{equation*}
    \|\Lambda_{\rm ext}^{\omega}[g_a(\omega)]\|_{L^2(\partial D)}
\lesssim \|g_a(\omega)\|_{H^1(\partial D)}+\left(1+\frac{\omega}{v}\right)\,\|g_a(\omega)\|_{L^2(\partial D)}.
\end{equation*}
The triangle inequality, the elementary bound
$\sum_\beta |a_\beta|\le m^{1/2} \|a \|_2$, and \Cref{lem:sharp-trace} yield
\begin{equation*}
    \|g_a(\omega)\|_{H^1(\partial D)}
\lesssim m^{1/2}\Bigl(1+\frac{\omega}{v_{\min}}\Bigr)^{3/2}\|a\|_2,
\qquad \|g_a(\omega)\|_{L^2(\partial D)}
\lesssim m^{1/2}\Bigl(1+\frac{\omega}{v_{\min}}\Bigr)^{1/2}\|a\|_2.
\end{equation*}
Combining the above two estimates implies 
\begin{equation*}
    \|\Lambda_{\rm ext}^{\omega}[g_a(\omega)]\|_{L^2(\partial D)} \lesssim m^{1/2}\Bigl(1+\frac{\omega}{v_{\min}}\Bigr)^{3/2}\|a\|_2.
\end{equation*}
The same argument with $g_b$ in place of $g_a$ yields
\begin{equation*}
    \|g_b(\omega)\|_{L^2(\partial D)}
\lesssim m^{1/2}\Bigl(1+\frac{\omega}{v_{\min}}\Bigr)^{1/2}\|b\|_2.
\end{equation*}

Substituting into \eqref{eq:scrC-pairing-bound} and taking the supremum over $\|a\|_2 = \|b\|_2 = 1$ gives \eqref{eq:scrC-bound}. The bound \eqref{eq:mcC-bound}
follows from $\mc C(\omega)=\frac{1}{2\omega}\widehat D\,\mathscr C(\omega)$ and
$\|\widehat D\|_{\ell^2(\C^m)\to\ell^2(\C^m)}=|\delta_1|v_1^2\lesssim \delta v_1^2$. Finally, \eqref{eq:mcC-high-freq} follows from $(1+\omega/v_{\min})^2/(2\omega)\sim\omega/(2v_{\min}^2)$
as $\omega\to\infty$.
\end{proof}

As a consequence of \cref{prop:sharp-matrix}, we finally discuss the
high-frequency dependence of the reduced residual in
\eqref{eq:finite-pencil-reduction}. We first recall the estimate from
\cite[Theorem A.1]{jfa}: for $k_0=\omega_0/v$ and
$\phi\in L^2(\partial D)$,
\begin{equation}\label{eq:KD-SD-highfreq}
 \Bigl\|\bigl(\tfrac12 I+\mathcal K_D^{k_0,*}\bigr)[\phi]\Bigr\|_{L^2(\partial D)}
\lesssim \Theta(k_0)\,\|\phi\|_{L^2(\partial D)},
\end{equation}
where 
\begin{equation}\label{eq:Theta-def}
    \Theta(k_0):=
    \begin{cases}
        1+k_0^{1/4}\log(2+k_0), & \text{in the general smooth case,}\\[3pt]
        1+k_0^{1/6}\log(2+k_0), & \text{if }\partial D\text{ has nonvanishing curvature.}
    \end{cases}
\end{equation}
We also need a local Lipschitz bound for $\mathcal K_D^{k,*}$. A direct
differentiation gives
\begin{equation*}
     \partial_k\partial_{\nu(x)}G^k(x,y) = \frac{k}{4\pi}e^{ik |x-y|}\,
    \frac{\nu(x)\cdot(x-y)}{|x-y|}.
\end{equation*}
Since $\partial D$ is smooth and compact, we have
$|\nu(x)\cdot(x-y)|\le C|x-y|$ for $x,y\in\partial D$. Hence, the kernel of
$\partial_k\mathcal K_D^{k,*}$ is bounded by $Ck$, and Schur's test gives
\begin{equation*}
    \|\partial_k\mathcal K_D^{k,*}\|_{L^2(\partial D)\to L^2(\partial D)}
    \lesssim 1 + k.
\end{equation*}
It follows that, for $k$ in a fixed neighborhood of $k_0$,
\begin{equation*}
    \|\mathcal K_D^{k,*}-\mathcal K_D^{k_0,*}\|_{L^2(\partial D)\to L^2(\partial D)}
    \lesssim (1+k_0)|k-k_0|, 
    \qquad k=\omega/v.
\end{equation*}
By the definition \eqref{def:lomega} of
$\mathcal L(\omega,\delta_{\partial D})$, the estimates
\eqref{est:lomegadelta} can therefore be refined to
\begin{equation*}
    \|\mathcal L(\omega,\delta_{\partial D})\|_{L^2(\partial D)\to L^2(\partial D)}
    \lesssim
    \delta\,\Theta(k_0)+\delta\,(1+k_0)|\omega-\omega_0|.
\end{equation*}

Tracking the wavenumber dependence in the Taylor remainder of the exact
Grushin matrix gives, for characteristic values in the local reduction
neighborhood,
\begin{equation*}
    \bigl\|\bigl((\omega-\omega_0)I_m-\mc C(\omega_0)\bigr)a\bigr\|_2
    \lesssim
    \delta\,\Upsilon(k_0)\,r(\omega_0)
    \bigl(|\omega-\omega_0|+\delta\,\Upsilon(k_0)\bigr)\|a\|_2,
\end{equation*}
where $\Upsilon(k_0):=\Theta(k_0)+1+k_0\lesssim 1+k_0$, and
$r(\omega_0)\geq1$ denotes a reduction constant controlling the residue,
the finite-dimensional coordinate maps, and the local $C^1$-norm of
$\mc R_0$ in that neighborhood.
Therefore, in the regime
\begin{equation}\label{eq:uniform-regime}
\delta(1+k_0)r(\omega_0)\ll 1, \qquad
\delta\,m(\omega_0)\,\omega_0\ll1,
\end{equation}
we have $\|\mc C(\omega_0)\|_{\ell^2(\C^m)\to\ell^2(\C^m)}\ll1$ by
\eqref{eq:mcC-high-freq}, and the residual estimate above is perturbative.
For fixed $v$ and large $\omega_0$, a sufficient condition implying the two conditions in \eqref{eq:uniform-regime} can be written as
\begin{equation*}
    \delta\,\omega_0\max\{r(\omega_0),m(\omega_0)\}\ll1.
\end{equation*}
We emphasize that this estimate only locates the resonance shifts in the pseudospectrum of $\mc C(\omega_0)$. Since $\mc C(\omega_0)$ is generally nonnormal, the estimate alone does not yield uniform eigenvalue or Jordan-block expansions. Such expansions would additionally require uniform
control of the exact Grushin remainder and of the spectral separation and conditioning of $\mc C(\omega_0)$; we do not pursue these estimates here.

\subsection{Numerical illustrations} 
We numerically illustrate the leading-order approximation derived above. We focus on spherical resonators in three dimensions, for which the Neumann eigenvalues of $-\Delta$ are explicitly characterized. For simplicity, throughout this subsection we set $v=v_i=1$.
\begin{lemma}
    Let $B_R \subset \R^3$ denote the ball of radius $R>0$ centered at $0$. The positive Neumann eigenvalues of $-\Delta$ on $B_R$ are
    \begin{equation*}
        \omega^2_{\ell,n} = \left(\frac{\beta_{\ell,n}}{R}\right)^2 \quad \text{for }\ell,n\in\NN \cup \{0\},
    \end{equation*}
    where $\{\beta_{\ell,n}\}_{n\geq0}$ are the positive zeros of the derivative of the spherical Bessel function of the first kind, indexed in increasing order, \emph{i.e.}, $j'_\ell(\beta_{\ell,n}) = 0$. Writing $x = \rho\phi$, where $\rho=\abs{x}$ and $\phi\in \mathbb S^2$, the corresponding eigenspace is spanned by
    \begin{equation*}
        u(\rho,\phi) = j_\ell(\beta_{\ell,n}\frac{\rho}{R})Y^m_\ell(\phi) \quad \text{for }m=-\ell, \dots, \ell,
    \end{equation*}
    where $Y^m_\ell$ denotes the spherical harmonic of degree $\ell$ and order $m$.
\end{lemma}
The restriction $\beta_{\ell,n}>0$ excludes the zero Neumann eigenfrequency, from which the subwavelength resonance bifurcates.

\begin{figure}[h]
    \centering
    \includegraphics[width=0.6\linewidth]{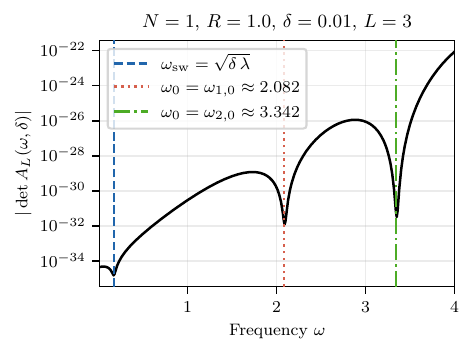}
    \caption{The characteristic determinant $\omega\mapsto \abs{\det A_L(\omega, \delta)}$ for a single spherical resonator. The dips identify the subwavelength resonant frequency of order $\sqrt{\delta}$, determined by the single eigenvalue $\lambda$ of the subwavelength capacitance matrix, as well as the resonance clusters near the two lowest positive Neumann eigenfrequencies.}
    \label{fig:det_scan}
\end{figure}

We next numerically validate \eqref{eq:approximation_omega}. We expand the layer potentials in spherical harmonics and truncate at degrees $\ell<L$. Discretizing the operator $\mathcal{A}(\omega,\delta)$ in \eqref{eq:multi-resonator-operator-equation} then yields a matrix $A_L(\omega,\delta)$ of size $2NL^2\times 2NL^2$. Using M\"uller's method to locate the zeros of $\omega\mapsto\det A_L(\omega,\delta)$, we approximate the resonant frequencies of \eqref{eq:multi-resonator-scattering-problem}.

\Cref{fig:det_scan} illustrates this determinant criterion. The subwavelength resonance and the two lowest non-subwavelength resonance clusters near the Neumann eigenfrequencies $\omega_{1,0}$ and $\omega_{2,0}$ are clearly resolved.

Near any fixed Neumann eigenfrequency $\omega_0$, we compare the numerically computed resonant frequencies $\{\omega_n^{\mathrm{exact}}\}$, obtained from the zeros of $\omega\mapsto\det A_L(\omega,\delta)$, with the leading-order approximations $\{\omega_n^{\mathrm{lead}}\}$ from \eqref{eq:approximation_omega}. Here and below, the superscript $\mathrm{exact}$ distinguishes the roots of the discretized characteristic equation from their leading-order approximations. As shown in \Cref{fig:thm37_convergence}, the observed Hausdorff set distance\footnote{Defined as $d(A,B) \coloneqq \max \{\sup_{a\in A}\operatorname{dist}(a,B),\sup_{b\in B}\operatorname{dist}(b,A)\}$, where $\operatorname{dist}(a,B)=\inf_{b\in B}|a-b|$.} scales as $\O(\delta^2)$ as $\delta\to 0$. This is consistent with the semisimple case of \eqref{eq:approximation_omega}.

\begin{figure}[h]
    \centering
    \includegraphics[width=0.9\linewidth]{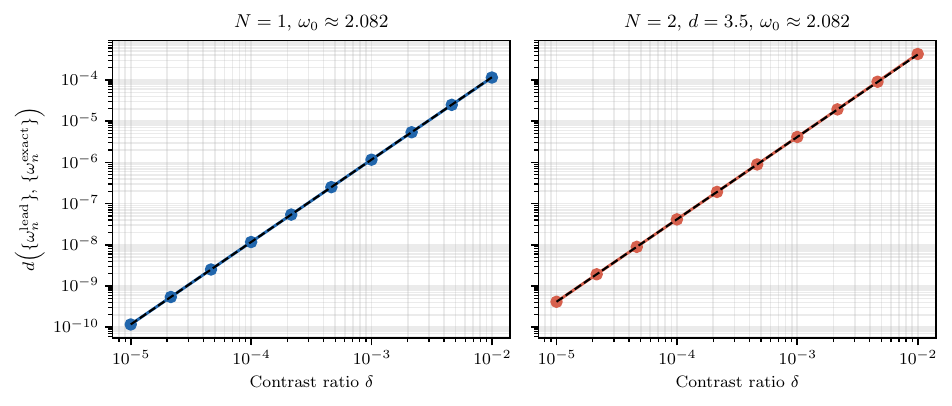}
    \caption{Hausdorff set distance between the leading-order approximate frequencies $\{\omega_n^{\mathrm{lead}}\}$ and the numerically computed resonant frequencies $\{\omega_n^{\mathrm{exact}}\}$ near the Neumann eigenfrequency $\omega_0\approx2.082$ of multiplicity $3$. For both a single spherical resonator ($N=1$ and $R=1$) and a dimer of identical spherical resonators ($N=2$, $R=1$, and separation distance $d=3.5$), the observed set distance scales as $\O(\delta^2)$.}
    \label{fig:thm37_convergence}
\end{figure}

\section{Periodic case} \label{sec:4}

As in the previous sections, let $D_1,\dots,D_N\subset\R^3$ be pairwise disjoint bounded domains with $C^\infty$ boundaries, and set $D:=\bigcup_{i=1}^N D_i$. We periodically repeat this configuration along a lattice $\Lambda$. For simplicity, we assume throughout this section that all interior wave speeds are equal, $v_i=v_b$ for $i=1,\dots,N$, and that all contrast parameters are equal, $\delta_i=\delta$.

We let $l_1,l_2, l_3\in \R^3$ denote the lattice vectors generating the lattice $\Lambda$:
$$\Lambda := \bigr\{ n_1 l_1+n_2 l_2 +n_{3} l_{3} ~|~ n_i \in \mathbb{Z} \bigr\}, $$
and denote by $Y$ a fundamental domain of $\Lambda$:
$$ Y := \bigr\{ c_1 l_1+c_2 l_2 +c_{3} l_{3} ~|~ 0 \le c_1,c_2,c_{3} \le 1 \bigr\}.$$
We recall that the dual lattice of $\Lambda$, denoted by $\Lambda^*$, is generated by $\alpha_1,\alpha_2,\alpha_{3}$ satisfying $ \alpha_i\cdot l_j = 2\pi \delta_{ij}$ (with $\delta_{ij}$ being the Kronecker symbol),  for $i,j = 1,2,3,$ and the Brillouin zone $Y^*$ is defined as $Y^*:= \R^3/\Lambda^*$. Moreover, for $\alpha \in Y^*$, we say that a smooth function $u$ is $\alpha$-quasiperiodic if 
$u(x+ m)= e^{\i \alpha \cdot m} u(x)$ for all $x\in \R^3$ and $m \in \Lambda$. 

We introduce the quasiperiodic Green function (or fundamental
solution) $G^{\alpha,\omega}$, which satisfies
 \begin{equation} \label{fGpoisson}
 (\Delta + \omega^2) G^{\alpha,\omega} (x,y) =
 \sum_{m \in \Lambda} \delta_0(x-y - m) e^{\i m \cdot \alpha} \quad \text{in } \R^3,
 \end{equation}
 where $\delta_0$ denotes the Dirac mass at $0$. 
 
 If $\omega \neq |q + \alpha|, \forall\; q \in
\Lambda^*$, then, using Poisson's summation formula
 \begin{equation} \label{poisson}
 \frac{1}{|Y|} \sum_{q \in \Lambda^*}
 e^{\i (q +\alpha) \cdot x} = \sum_{m \in \Lambda} \delta_0(x-m) e^{\i m \cdot\alpha},
 \end{equation}
where $|Y|$ denotes the volume of $Y$, the quasiperiodic
Green function $G^{\alpha,\omega}$ can be represented as the sum of augmented plane waves over the 
reciprocal lattice:  \begin{equation} \label{rephelm1}
G^{\alpha,\omega} (x,y) = \frac{1}{|Y|} \sum_{q\in \Lambda^*}
\frac{e^{\i (q+ \alpha) \cdot (x-y)}}{\omega^2 - |q + \alpha|^2};\end{equation}
see, for instance, \cite{photonic2018}.  

Whenever the quasiperiodic layer-potential representation below is used, we assume that $\omega \neq |q + \alpha|$
for all $q \in \Lambda^*$. For $\omega
> 0$, let $\mc S_D^{\alpha,\omega}$ and $\mc D_D^{\alpha,\omega}$ be the
quasiperiodic single- and double-layer potentials
associated with $G^{\alpha,\omega}$ on $D$; that is, for a given
density $\vp \in L^2(\p D)$,
\begin{align*}
\mc S_D^{\alpha, \omega}[\varphi] (x) & = \int_{\p D} G^{\alpha,\omega}(x,y)
\varphi (y) \, \mathrm{d} \sigma (y), \quad x \in \R^3,
\\ \mc D_D^{\alpha,\omega} [\varphi] (x) & = \int_{\p D}
\frac{\partial G^{\alpha,\omega} (x,y)}{\partial \nu(y)} \varphi (y) \, \mathrm{d} \sigma (y), \quad x
\in \R^3 \setminus \partial D.
\end{align*}
We also recall  the jump relations 
obeyed by the double-layer potential  and
by the normal derivative of the  single-layer potential  on $\partial D$:
\begin{align*}
\frac{\partial (\mc S_D^{\alpha,\omega} [\varphi])}{\partial \nu} \bigg |_{\pm}(x) & = \bigg
(\pm \frac{1}{2} I + (\mc K_D^{-\alpha,\omega})^*\bigg )[\varphi] (x),
\quad  x \in \partial D, \\
 (\mc D_D^{\alpha,\omega}[\varphi]) \bigg|_{\pm} (x) & = \bigg(\mp \frac{1}{2}
I + \mc K_D^{\alpha,\omega} \bigg)[\varphi] (x), \quad x \in
\partial D, 
\end{align*}
for $\varphi \in L^2(\partial D)$, where $\mc K_D^{\alpha,\omega}$ is the
operator on $L^2(\partial D)$ defined by
 $$
 \mc K_D^{\alpha,\omega}[\varphi] (x) = 
 \int_{\partial D} \frac{\partial G^{\alpha,\omega}(x,y)}{\partial \nu(y)} \varphi (y) \, \mathrm{d} \sigma (y)
 $$
and $(\mc K_D^{-\alpha,\omega})^*$ is the $L^2$-adjoint operator of
$\mc K_D^{-\alpha,\omega}$, which is given by
 $$
 (\mc K_D^{-\alpha,\omega})^*[\varphi] (x) = \int_{\partial D}
 \frac{\partial G^{\alpha,\omega}(x,y)}{\partial \nu(x)} \varphi (y) \, \mathrm{d} \sigma (y);
 $$
see again \cite{photonic2018}.

Finally, we emphasize that, in this section, boundary pairings $\langle \cdot, \cdot \rangle_{\partial D}$ are understood in the Hilbert-space sesquilinear sense. This convention differs from the bilinear pairing used in the free-space frequency-dependent capacitance matrix  \cref{def:freq-cap-matrix-general-updated} and its complex symmetric property in \cref{prop:complex}.

\subsection{Asymptotics of the Bloch band functions and associated Bloch eigenmodes} \label{sec:asymperiodic}

In this subsection, we derive the leading-order asymptotics of the Bloch eigenfrequencies for a high-contrast periodic scattering problem. To do so, we fix $\alpha\in Y^\ast$ and assume that the reference exterior wave
number
$\omega_0/v$ is not a diffraction threshold, that is, 
\begin{equation}\label{eq:three-regimes-no-threshold}
\frac{\omega_0}{v} \neq |q+\alpha| \qquad \text{for all } q\in \Lambda^\ast.
\end{equation}
Thus, the $\alpha$-quasiperiodic Green function is well defined for wave numbers near $\omega_0/v$. Then, we let
\begin{equation*}
    \Omega:=Y\setminus \overline D,
\end{equation*}
and consider, by the Floquet--Bloch theory, the $\alpha$-quasiperiodic
transmission problem
\begin{equation}\label{eq:periodic-bloch-eigenproblem}
\begin{dcases}
\Delta u + \dfrac{\omega^2}{v_b^2}u = 0 & \text{in } D,\\
\Delta u + \dfrac{\omega^2}{v^2}u = 0 & \text{in } \Omega,\\
u \big|_- = u \big|_+ & \text{on } \partial D,\\
\partial_\nu u \big|_- = \delta\,\partial_\nu u \big|_+ & \text{on } \partial D,\\
u \text{ is } \alpha\text{-quasiperiodic}.
\end{dcases}
\end{equation}

As in the derivation of \eqref{eq:multi-resonator-operator-equation}, whenever the relevant DtN maps are defined, the transmission problem \eqref{eq:periodic-bloch-eigenproblem} is equivalent to a boundary equation for the trace $g=u|_{\partial D}$. More precisely,  throughout this subsection, we assume that
\begin{equation*}
    \frac{\omega_0^2}{v_b^2}\notin \sigma_{\rm Dir}(-\Delta;D),
\end{equation*}
so that the interior DtN map is analytic near $\omega_0$. We also assume that $\omega^2/v^2$ is not an $\alpha$-quasiperiodic exterior Dirichlet eigenvalue in $\Omega$. Then \eqref{eq:periodic-bloch-eigenproblem} is equivalent to
\begin{equation}\label{eq:Tdelta-alpha}
\mathcal T_\delta^\alpha(\omega)g=0, \qquad \mathcal T_\delta^\alpha(\omega):= \Lambda_{\rm in}^{\omega}-\delta \Lambda_{\rm ext}^{\alpha,\omega}, 
\end{equation}
where 
\begin{equation*}
    \Lambda_{\rm in}^\omega
    :=
    \Bigl(-\frac12 I+(\mathcal K_D^{k_b})^*\Bigr)(\mathcal S_D^{k_b})^{-1}: H^{1/2}(\partial D)\to H^{-1/2}(\partial D),
\end{equation*}
and
\begin{equation*}
    \Lambda_{\rm ext}^{\alpha,\omega}
    :=
    \Bigl(\frac12 I+(\mathcal K_D^{-\alpha,k})^*\Bigr)
    (\mathcal S_D^{\alpha,k})^{-1}:  H^{1/2}(\partial D)\to H^{-1/2}(\partial D),
\end{equation*}
with $k_b:=\omega/v_b$ and $k:=\omega/v$. Here,
$\Lambda_{\rm in}^{\omega}$ is the interior DtN map for the
Helmholtz operator in $D$ at the wave number $k_b$, while
$\Lambda_{\rm ext}^{\alpha,\omega}$ is the meromorphic
$\alpha$-quasiperiodic exterior DtN map in $\Omega$ at
wave number $k$, with poles at the $\alpha$-quasiperiodic Dirichlet eigenvalues
of $\Omega$.

The possible limiting regimes of \eqref{eq:Tdelta-alpha} as $\d \to 0$ are determined by two mechanisms: the zeros of the interior DtN map $\Lambda_{\rm in}^\omega$ and the poles of the exterior quasiperiodic DtN map $\Lambda_{\rm ext}^{\alpha,\omega}$. The former correspond to interior Neumann eigenfrequencies, while the latter correspond to $\alpha$-quasiperiodic Dirichlet eigenfrequencies in the exterior cell $\Omega=Y\setminus\overline D$. Thus, unlike the finite-resonator problem in free space, the periodic setting has three nontrivial local regimes near a reference frequency $\omega_0$: an interior Neumann frequency with a regular exterior problem, an interior Neumann frequency that coincides with an exterior $\alpha$-quasiperiodic Dirichlet frequency, and an exterior $\alpha$-quasiperiodic Dirichlet frequency at which the interior problem is regular.

\begin{remark}[(Comparison with the free-space setting)]
In the two- and three-dimensional free-space finite-resonator problems, the exterior outgoing Dirichlet problem is well posed at positive real frequencies. Hence, near any real frequency $\omega_0>0$, the exterior DtN map is analytic and does not contribute a pole to the limiting regime. The one-dimensional setting is different; see \cite{fabry1} and Appendix \ref{sec:reduct}. In the periodic problem, the exterior domain is instead the bounded cell $\Omega=Y\setminus\overline D$ with $\alpha$-quasiperiodic boundary conditions. Its exterior DtN map is meromorphic, with real poles at the $\alpha$-quasiperiodic Dirichlet eigenfrequencies, and these poles are responsible for the additional regimes (Case 2 and Case 3 below).
\end{remark}

We discuss these cases in turn.  We denote by $\sigma_{\rm Neu}(-\Delta;D)$ the Neumann eigenvalues in $D$, and by $\sigma_{\alpha, \mathrm{Dir}}(-\Delta;\Omega)$ the $\alpha$-quasiperiodic Dirichlet eigenvalues in $\Omega$. 

\subsubsection*{Case 1: Interior Neumann eigenvalue, exterior regular}

Assume that
\begin{equation}\label{eq:case1-assumption}
\frac{\omega_0^2}{v_b^2}\in \sigma_{\rm Neu}(-\Delta;D), \qquad
\frac{\omega_0^2}{v^2}\notin \sigma_{\alpha,\mathrm{Dir}}(-\Delta;\Omega)
\quad \text{for \emph{a fixed} } \alpha\in Y^*.
\end{equation}
That is, the exterior $\alpha$-quasiperiodic Dirichlet problem is well posed at $\omega_0$.

Let
\begin{equation*}
    \mathscr G(\omega_0):=\ker \Lambda_{\rm in}^{\omega_0}, \qquad
    m:=\dim \mathscr G(\omega_0),
\end{equation*}
and choose an $L^2(D)$-orthonormal basis
$u_1,\dots,u_m$ of the interior Neumann eigenspace at $\omega_0$:
\begin{equation*}
    -\Delta u_p=\frac{\omega_0^2}{v_b^2}u_p \quad \text{in } D, \qquad
\partial_\nu u_p=0 \quad \text{on } \partial D, \qquad 
\int_D u_p \overline{u_q}\,dx=\delta_{pq}.
\end{equation*}
Again, similarly to \eqref{eq:tracegu}, we set
\begin{equation} \label{defgp}
    g_p:= u_p|_{\partial D}\in H^{1/2}(\partial D), \qquad p=1,\dots,m.
\end{equation}

Since the exterior problem is regular, $\Lambda_{\rm ext}^{\alpha,\omega}$ is analytic near $\omega_0$. On the other hand, differentiating the interior Helmholtz problem with respect to $\omega$ gives
\begin{equation}\label{eq:interior-derivative-normalization}
\left\langle \partial_\omega \Lambda_{\rm in}^{\omega_0}[g_q],\,g_p\right\rangle_{\partial D}
= -\frac{2\omega_0}{v_b^2}\,\delta_{pq}.
\end{equation}
Hence, on $\mathscr G(\omega_0)$, we have 
\begin{equation}\label{eq:Lambda-in-case1}
\Lambda_{\rm in}^{\omega}
= (\omega-\omega_0)\mathcal M(\omega_0)
+ \O((\omega-\omega_0)^2),
\text{ where } \mathcal M(\omega_0):=-\frac{2\omega_0}{v_b^2}I.
\end{equation}
Following \cref{def:freq-cap-matrix-general-updated}, we introduce the \emph{regular $\alpha$-quasiperiodic frequency-dependent capacitance matrix} as
\begin{equation}\label{eq:case1-cap-matrix}
\mc{C}^{\mathrm{reg}}_{\alpha}(\omega_0) =    \bigl((\mc{C}^{\mathrm{reg}}_{\alpha})_{pq}(\omega_0)\bigr)_{1\le p,q\le m},
\qquad
(\mc{C}^{\mathrm{reg}}_{\alpha})_{p,q}(\omega_0) :=
    -\frac{v_b^2}{2\omega_0}
    \left\langle
    \Lambda_{\rm ext}^{\alpha,\omega_0}[g_q],\,g_p
    \right\rangle_{\partial D}.
\end{equation}
Since the exterior $\alpha$-quasiperiodic Dirichlet problem is self-adjoint at real frequencies away from its spectrum, $\Lambda_{\rm ext}^{\alpha,\omega_0}$ is self-adjoint. Hence, we have the following result. 

\begin{lemma} \label{lem:hermitan}
Assuming \eqref{eq:case1-assumption}, $\mc C^{\mathrm{reg}}_{\alpha}(\omega_0)$ is Hermitian.     
\end{lemma}

\begin{theorem}
\label{thm:case1-asymptotics}
Assume that \eqref{eq:case1-assumption} holds. Let
$\lambda_1^\alpha,\dots,\lambda_m^\alpha$ be the eigenvalues of
$\mc C^{\mathrm{reg}}_{\alpha}(\omega_0)$, counted with their multiplicities. Then, the
Bloch eigenfrequencies converging to $\omega_0$ can be labeled so that
\begin{equation} \label{approx:alphaj}
    \omega_j^\alpha(\delta)
    =
    \omega_0+\delta\lambda_j^\alpha+\O(\delta^2),
    \qquad j=1,\dots,m,
\end{equation}
as $\delta \to 0$, for the fixed $\alpha$ in \eqref{eq:case1-assumption}.
\end{theorem}

\begin{proof}
Decompose
\begin{equation*}
    H^{1/2}(\partial D)=\mathscr G(\omega_0)\oplus \mathscr G(\omega_0)^\perp, 
\end{equation*}
and write
\begin{equation*}
    g=g_N+g_\perp.
\end{equation*}
Since $\Lambda_{\rm in}^{\omega_0}$ is invertible on $\mathscr G(\omega_0)^\perp$,
Lyapunov--Schmidt reduction applied to \eqref{eq:Tdelta-alpha} gives
\begin{equation*}
    \|g_\perp\|
    =
    \O\bigl((\delta+|\omega-\omega_0|)\|g_N\|\bigr).
\end{equation*}
Projecting \eqref{eq:Tdelta-alpha} onto $\mathscr G(\omega_0)$ and using \eqref{eq:Lambda-in-case1} yields
\begin{align*} 
    \Bigl((\omega-\omega_0)\mathcal M(\omega_0) - \delta\,P_{\mathscr G}\Lambda_{\rm ext}^{\alpha,\omega_0}P_{\mathscr G}
\Bigr)a = \O\bigl((\delta|\omega-\omega_0|+|\omega-\omega_0|^2+\delta^2) \|a\|_2\bigr), 
\end{align*}
where $g_N=\sum_{p=1}^m a_p g_p$ and $P_{\mathscr G}$ denotes the projection onto $\mathscr G(\omega_0)$. Characteristic values in this regime satisfy $|\omega-\omega_0|=\O(\delta)$. Multiplying by $\mathcal M(\omega_0)^{-1}$ and using \eqref{eq:case1-cap-matrix}, we obtain
\begin{equation} \label{approx:bloch}
    \Bigl((\omega-\omega_0)I-\delta \mc{C}^{\mathrm{reg}}_{\alpha}(\omega_0)
\Bigr)a = \O(\delta^2 \|a\|_2).
\end{equation}
Since $\mc C^{\mathrm{reg}}_{\alpha}(\omega_0)$ is Hermitian, it is
diagonalizable with real eigenvalues. Applying the spectral theorem to the
finite-dimensional pencil above gives the stated $\O(\delta^2)$ expansion. 
\end{proof}

\subsubsection*{Case 2: Interior Neumann eigenvalue and exterior Dirichlet eigenvalue}

Assume that
\begin{equation}\label{eq:case2-assumption}
\frac{\omega_0^2}{v_b^2}\in \sigma_{\rm Neu}(-\Delta;D),
\qquad
\frac{\omega_0^2}{v^2}\in \sigma_{\alpha,\mathrm{Dir}}(-\Delta;\Omega),
\end{equation}
for \emph{some} fixed $\alpha\in Y^*$. Let
\begin{equation*}
    u_1,\dots,u_m,\qquad g_p=u_p|_{\partial D},
\end{equation*}
be the interior Neumann eigenfunction basis and boundary traces introduced in
Case~1. Let $v_1^\alpha,\dots,v_r^\alpha$ be an $L^2(\Omega)$-orthonormal
basis of the exterior $\alpha$-quasiperiodic Dirichlet eigenspace at
$\omega_0$, that is,
\begin{equation*}
    -\Delta v_s^\alpha=\frac{\omega_0^2}{v^2}v_s^\alpha
    \quad \text{in } \Omega, \qquad
    v_s^\alpha=0 \quad \text{on } \partial D,
    \qquad
    \int_\Omega v_s^\alpha \overline{v_t^\alpha}\,dx=\delta_{st}.
\end{equation*}
We also introduce 
\begin{equation*}
    \psi_s^\alpha:=\partial_\nu v_s^\alpha|_{\partial D}\in H^{-1/2}(\partial D).
\end{equation*}
Then the exterior DtN map has a simple pole at $\omega_0$ with the Laurent expansion
\begin{equation}\label{eq:Lambda-ext-pole}
\Lambda_{\rm ext}^{\alpha,\omega}
= \frac{1}{\omega-\omega_0}\,\mathcal R_{-1}^\alpha + \mathcal R_0^\alpha
+ \O(\omega-\omega_0),
\end{equation}
where
\begin{equation}\label{eq:residue-explicit}
\mathcal R_{-1}^\alpha = - \frac{v^2}{2\omega_0}
\sum_{s=1}^r \psi_s^\alpha\otimes \psi_s^\alpha.
\end{equation}
Here, $\psi_s^\alpha\otimes \psi_s^\alpha$ is the mapping defined by  $(\psi_s^\alpha\otimes\psi_s^\alpha)[g] := \langle g,\psi_s^\alpha\rangle_{\partial D} \psi_s^\alpha$. The minus sign
comes from the convention that the normal on $\partial D$ points out of the resonators. 

Define the coupling matrix between the interior and the exterior modes, $g_p$ and $v^\alpha_s$, by
\begin{equation}\label{eq:Gamma-alpha}
\Gamma_\alpha=
\bigl(\Gamma_\alpha^{sp}\bigr)_{1\le s\le r}^{1\le p\le m},
\qquad
\Gamma_\alpha^{sp}
:=
\langle g_p,\psi_s^\alpha\rangle_{\partial D},
\end{equation}
and set
\begin{equation}\label{eq:scaled-Gamma-alpha}
    \widetilde\Gamma_\alpha
    :=
    \frac{v_bv}{2\omega_0}\Gamma_\alpha.
\end{equation}
We introduce the two \emph{singular $\alpha$-quasiperiodic frequency-dependent capacitance matrices}
\begin{align}
    \mc C^{\mathrm{sing}}_\alpha(\omega_0) &:=
    \widetilde\Gamma_\alpha^\ast\widetilde\Gamma_\alpha =
    \frac{v_b^2v^2}{4\omega_0^2}\, \Gamma_\alpha^\ast \Gamma_\alpha \in \mathbb{C}^{m \times m},
    \label{eq:case2-singular-matrix}\\ 
    \mc C^{\mathrm{sing,ext}}_\alpha(\omega_0) &:=
    \widetilde\Gamma_\alpha\widetilde\Gamma_\alpha^\ast =
    \frac{v_b^2v^2}{4\omega_0^2}\,   \Gamma_\alpha\Gamma_\alpha^\ast  \in \mathbb{C}^{r \times r}.
    \label{eq:case2-exterior-singular-matrix}
\end{align}
where the superscript $\ast$ denotes the Hermitian adjoint.
The first acts on the interior Neumann trace coefficients, while the second acts on the exterior Dirichlet eigenspace coefficients.

The following results hold.

\begin{lemma}
Assuming \eqref{eq:case2-assumption}, both $\mc C^{\mathrm{sing}}_\alpha(\omega_0)$ and $\mc C^{\mathrm{sing,ext}}_\alpha(\omega_0)$ are Hermitian positive semi-definite. Moreover, they have the same nonzero spectra. 
\end{lemma}

\begin{theorem}
\label{thm:case2-asymptotics}
Assume that \eqref{eq:case2-assumption} holds, and set
$p_\alpha:=\operatorname{rank}\Gamma_\alpha$. Let
$\lambda_1^\alpha,\dots,\lambda_{p_\alpha}^\alpha$ be the common strictly positive
eigenvalues of $\mc C^{\mathrm{sing}}_{\alpha}(\omega_0)$ and
$\mc C^{\mathrm{sing,ext}}_\alpha(\omega_0)$, counted with their
multiplicities. Then the Bloch eigenfrequency branches associated with the
nonzero singular directions satisfy
\begin{equation*}
    \omega_{j,\pm}^\alpha(\delta) =  \omega_0\pm\delta^{1/2}\sqrt{\lambda_j^\alpha}+\O(\delta),
    \qquad j=1,\dots,p_\alpha\,, \quad \text{as $\delta\to 0$.}
\end{equation*}
The remaining $m+r-2p_\alpha$ Bloch eigenfrequency branches, corresponding at leading order to the space
$\ker\widetilde\Gamma_\alpha \oplus \ker\widetilde\Gamma_\alpha^\ast$, satisfy
\begin{equation*}
    \omega^\alpha(\delta) = \omega_0 + \O(\delta). 
\end{equation*}
All branches are counted with algebraic multiplicity.
\end{theorem}

\begin{remark}
\label{rem:case2-degenerate}
Before proving \Cref{thm:case2-asymptotics}, we note that the extreme case $\mc C^{\mathrm{sing}}_\alpha(\omega_0)=0$ is possible. In fact,  $\Gamma_\alpha^{sp}=\langle g_p,\psi_s^\alpha\rangle_{\partial D}$ may vanish identically, e.g., when a symmetry of $D$ forces a parity mismatch between the interior Neumann boundary traces $g_p$ and the exterior $\alpha$-quasiperiodic Dirichlet normal traces $\psi_s^\alpha$. In this degenerate situation, the $\delta^{1/2}$ splitting is absent. 
\end{remark}

\begin{proof}
Set $z:=\omega-\omega_0$ and $\varepsilon:=\delta^{1/2}$. As in Case~1, we decompose
\begin{equation*}
    H^{1/2}(\partial D)=\mathscr G(\omega_0)\oplus \mathscr G(\omega_0)^\perp,
    \qquad
    g=g_N+g_\perp,
    \qquad
    g_N=\sum_{p=1}^m a_p g_p.
\end{equation*}
By \eqref{eq:Lambda-in-case1}, the interior DtN map satisfies
\begin{equation*}
    \Lambda_{\rm in}^{\omega_0+z}
    =
    z\,\mathcal M(\omega_0)+\O(z^2)
    \quad \text{on } \mathscr G(\omega_0),
    \qquad
    \mathcal M(\omega_0)=-\frac{2\omega_0}{v_b^2}I_m.
\end{equation*}
Moreover, by \eqref{eq:Lambda-ext-pole}--\eqref{eq:residue-explicit}, we have 
\begin{equation*}
    \Lambda_{\rm ext}^{\alpha,\omega_0+z}
    =
    \frac{1}{z}\,\mathcal R_{-1}^\alpha+\mathcal R_{\rm reg}^\alpha(z),
    \qquad
    \mathcal R_{-1}^\alpha
    =
    -\frac{v^2}{2\omega_0}\sum_{s=1}^r \psi_s^\alpha\otimes \psi_s^\alpha,
\end{equation*}
where $\mathcal R_{\rm reg}^\alpha$ is analytic near $z=0$.

Introducing the finite-rank map:
\begin{equation*}
    \mathcal B_\alpha:H^{1/2}(\partial D)\to \mathbb C^r,
    \qquad
    \mathcal B_\alpha f := \bigl(\langle f,\psi_s^\alpha\rangle_{\partial D}\bigr)_{s=1}^r,
\end{equation*}
then we can write 
\begin{equation*}
    \mathcal R_{-1}^\alpha = -\frac{v^2}{2\omega_0} \mathcal B_\alpha^\ast\mathcal B_\alpha.
\end{equation*}
Hence, for $z\neq 0$, the boundary equation \eqref{eq:Tdelta-alpha} is
equivalent to the augmented system
\begin{equation*}
    \begin{cases}
    \Lambda_{\rm in}^{\omega_0+z}g
    +
    \varepsilon \frac{v}{\sqrt{2\omega_0}}\,\mathcal B_\alpha^\ast\beta
    -
    \varepsilon^2 \mathcal R_{\rm reg}^\alpha(z)g
    =0,\\[2mm]
    z\beta-\varepsilon \frac{v}{\sqrt{2\omega_0}}\,\mathcal B_\alpha g=0,
    \end{cases}
\end{equation*}
with unknowns $g\in H^{1/2}(\partial D)$ and $\beta\in\mathbb C^r$. Indeed,
eliminating $\beta$ from the second equation gives
\begin{equation*}
    \beta=\frac{\varepsilon v }{z \sqrt{2\omega_0}}\,\mathcal B_\alpha g,
\end{equation*}
and substituting this into the first equation recovers
\eqref{eq:Tdelta-alpha}. Thus, the coupled system is an analytic desingularization of the local characteristic equation at $z=0$.

Since $\Lambda_{\rm in}^{\omega_0}$ is invertible on
$\mathscr G(\omega_0)^\perp$, Lyapunov--Schmidt reduction applied to the first
equation gives
\begin{equation*}
    g_\perp
    =
    \O\bigl((|z|+\varepsilon)(|a|+|\beta|)\bigr).
\end{equation*}
Projecting the first equation onto $\mathscr G(\omega_0)$ and using
$\mathcal B_\alpha g_N=\Gamma_\alpha a$, we obtain
\begin{equation*}
    z\,\mathcal M(\omega_0)a
    +
    \varepsilon \frac{v}{\sqrt{2\omega_0}}\,\Gamma_\alpha^\ast\beta
    =
    \O\bigl((z^2+\varepsilon|z|+\varepsilon^2)(|a|+|\beta|)\bigr).
\end{equation*}
The second equation gives
\begin{equation*}
    z\beta-\varepsilon \frac{v}{\sqrt{2\omega_0}}\,\Gamma_\alpha a
    =
    \O\bigl((\varepsilon|z|+\varepsilon^2)(|a|+|\beta|)\bigr).
\end{equation*}

Now, setting
\begin{equation*}
    b:= \frac{v_b}{\sqrt{2\omega_0}} \beta,
\end{equation*}
the reduced system in terms of $(a,b)$ becomes
\begin{equation} \label{eq:leading-coupling}
    \Bigl(z I_{m+r} - \varepsilon H_\alpha + \mathcal E_\alpha(z,\varepsilon) \Bigr) \binom{a}{b} = 0, \qquad   H_\alpha := \begin{pmatrix}
    0 & \widetilde\Gamma_\alpha^\ast\\
    \widetilde\Gamma_\alpha & 0
    \end{pmatrix},
\end{equation}
and $\mathcal E_\alpha$ is analytic near $(0,0)$ with
\begin{equation*}
    \|\mathcal E_\alpha(z,\varepsilon)\|
    \lesssim
    z^2+\varepsilon|z|+\varepsilon^2.
\end{equation*}
We can write $z=\varepsilon\mu$ and find 
\begin{equation*}
    \Bigl( \mu I_{m+r} - H_\alpha
    + \varepsilon \widetilde{\mathcal E}_\alpha(\mu,\varepsilon)
    \Bigr)
    \binom{a}{b} = 0,
\end{equation*}
where $\widetilde{\mathcal E}_\alpha$ is analytic and uniformly bounded for $\mu$ in bounded sets. Letting $\varepsilon$ 
be sufficiently small and applying the generalized Rouch\'e theorem for analytic matrix-valued functions, we conclude that all characteristic values near $\omega_0$ satisfy $|\omega - \omega_0| = \O(\varepsilon)$. Moreover, standard finite-dimensional perturbation theory for
analytic matrix pencils now shows that the characteristic values $\mu$ are $\O(\varepsilon)$-perturbations of the eigenvalues of $H_\alpha$. Note that $H_\alpha$ is Hermitian, its nonzero eigenvalues are
\begin{equation*}
\pm\sqrt{\lambda_1^\alpha},\dots,\pm\sqrt{\lambda_{p_\alpha}^\alpha},
\end{equation*}
and
\begin{equation*}
    \ker H_\alpha
    =
    \ker\widetilde\Gamma_\alpha\oplus \ker\widetilde\Gamma_\alpha^\ast,
    \qquad
    \dim\ker H_\alpha=m+r-2p_\alpha.
\end{equation*}
Therefore, the $2p_\alpha$ nonzero branches satisfy
\begin{equation*}
    \mu_{j,\pm}^\alpha(\varepsilon)
    =
    \pm\sqrt{\lambda_j^\alpha}+\O(\varepsilon),
    \qquad
    j=1,\dots,p_\alpha,
\end{equation*}
whereas the remaining $m+r-2p_\alpha$ branches satisfy
\begin{equation*}
    \mu^\alpha(\varepsilon)=\O(\varepsilon).
\end{equation*}
Since $z=\varepsilon\mu$, $\varepsilon^2=\delta$, and $\omega=\omega_0+z$, we complete the proof. 
\end{proof}

We now show that the singular quasiperiodic frequency-dependent capacitance matrix $\mc{C}^{\mathrm{sing}}_{\alpha}(\omega_0)$ arises as the residue of the natural meromorphic extension of the regular quasiperiodic frequency-dependent capacitance matrix $\mc{C}^{\mathrm{reg}}_{\alpha}(\omega)$ at $\omega_0$ satisfying \eqref{eq:case2-assumption}. This is the periodic analogue of the singular reduction in the one-dimensional setting; see Appendix \ref{sec:reduct}. 

\begin{lemma}[(Connection between the regular and singular quasiperiodic capacitance matrices)]
\label{lem:reg-sing-connection}
Assume \eqref{eq:case2-assumption}, and let $g_1,\dots,g_m$ and
$\psi_1^\alpha,\dots,\psi_r^\alpha$ be as above. Consider the quasiperiodic frequency-dependent capacitance matrix as in Case 1:
\begin{equation}\label{eq:Creg-meromorphic}
\bigl(\mc{C}^{\mathrm{reg}}_{\alpha}(\omega)\bigr)_{p,q} =
-\frac{v_b^2}{2\omega}\,
\bigl\langle \Lambda_{\rm ext}^{\alpha,\omega}[g_q],g_p\bigr\rangle_{\partial D},
\qquad 1\le p,q\le m,
\end{equation}
for $\omega$ in a punctured neighborhood of $\omega_0$. Then, $\mc C^{\mathrm{reg}}_{\alpha}(\,\cdot\,)$ is meromorphic at $\omega_0$
with at most a simple pole and admits the following Laurent expansion:
\begin{equation}\label{eq:reg-sing-connection}
\mc{C}^{\mathrm{reg}}_{\alpha}(\omega)
=
\frac{1}{\omega-\omega_0}\,\mc{C}^{\mathrm{sing}}_{\alpha}(\omega_0)
+
\mathcal H_\alpha(\omega),
\end{equation}
where $\mc{C}^{\mathrm{sing}}_{\alpha}(\omega_0)$ is the singular quasiperiodic
frequency-dependent capacitance matrix \eqref{eq:case2-singular-matrix}, and
$\mathcal H_\alpha:U\to \mathbb{C}^{m\times m}$ is analytic in a neighborhood $U$ of
$\omega_0$.
\end{lemma}

\begin{proof}
By the Laurent expansion \eqref{eq:Lambda-ext-pole} of
$\Lambda_{\rm ext}^{\alpha,\omega}$ at $\omega_0$ and the explicit residue \eqref{eq:residue-explicit}, pairing with $g_q$ and $g_p$, and using from \eqref{eq:Gamma-alpha} that
$\Gamma_\alpha^{sp}=\langle g_p,\psi_s^\alpha\rangle_{\partial D}$ yields
\begin{equation*}
\bigl\langle \mathcal R_{-1}^\alpha[g_q],g_p\bigr\rangle_{\partial D}
=
-\frac{v^2}{2\omega_0}\sum_{s=1}^r
\langle g_q,\psi_s^\alpha\rangle\,\overline{\langle g_p,\psi_s^\alpha\rangle}
=
-\frac{v^2}{2\omega_0}\,(\Gamma_\alpha^*\Gamma_\alpha)_{p,q}.
\end{equation*}
Substituting the above formula into \eqref{eq:Creg-meromorphic} and using that
$\omega\mapsto -v_b^2/(2\omega)$ is analytic at $\omega_0$ with the value $-v_b^2/(2\omega_0)$, we obtain
\begin{equation*}
\mc{C}^{\mathrm{reg}}_{\alpha}(\omega)
=
\frac{1}{\omega-\omega_0}\,\frac{v_b^2 v^2}{4\omega_0^2}\,\Gamma_\alpha^*\Gamma_\alpha
+
\mathcal H_\alpha(\omega),
\end{equation*}
where $\mathcal H_\alpha(\omega)$ collects the analytic contributions originating from $\mathcal R_0^\alpha$, the higher-order coefficients of the Laurent expansion of
$\Lambda_{\rm ext}^{\alpha,\omega}$ and the Taylor expansion of $-v_b^2/(2\omega)$ at $\omega_0$. By \eqref{eq:case2-singular-matrix}, the residue equals $\mc{C}^{\mathrm{sing}}_{\alpha}(\omega_0)$, finishing the proof. 
\end{proof}

\begin{remark}
\label{rem:reg-sing-scaling}
\Cref{lem:reg-sing-connection} explains the two distinct scalings that appear in \Cref{thm:case1-asymptotics,thm:case2-asymptotics}. When $\omega_0$ is regular (Case~1), $\mc{C}^{\mathrm{reg}}_{\alpha}(\omega_0)$ is finite and the resonance equation gives $$|\omega^\alpha(\delta)-\omega_0|=\O(\delta).$$ At a Case-2 frequency, the simple-pole factor $1/(\omega-\omega_0)$ in \eqref{eq:reg-sing-connection} implies that
$$(\omega-\omega_0)^2\sim \delta\,\mc{C}^{\mathrm{sing}}_{\alpha}(\omega_0),$$
on the coupled interior-trace sector, which produces the
$\delta^{1/2}$-splitting. The directions not seen by this sector require a next-order reduction; see \cref{rem:sing-matrices} below. 
\end{remark}

\begin{remark} \label{rem:sing-matrices}
It is worth distinguishing the roles of the two singular capacitance matrices \eqref{eq:case2-singular-matrix}--\eqref{eq:case2-exterior-singular-matrix} in Case~2. The matrix $\mc C^{\mathrm{sing}}_\alpha(\omega_0)$ acts on the $m$-dimensional space of interior Neumann eigenspace coefficients, whereas $\mc C^{\mathrm{sing,ext}}_\alpha(\omega_0)$ acts on the $r$-dimensional space of exterior Dirichlet eigenspace coefficients. They arise, respectively, by eliminating the $b$- and the $a$-block of the leading pencil $zI_{m+r}-\varepsilon H_\alpha$ in \eqref{eq:leading-coupling}, that is, as its two Schur complements.

In particular, $\mc C^{\mathrm{sing,ext}}_\alpha(\omega_0)$ cannot be the residue of $\mc C^{\mathrm{reg}}_\alpha(\omega)$, since the latter is an $m\times m$ matrix on the interior Neumann trace space, whereas $\mc C^{\mathrm{sing,ext}}_\alpha(\omega_0)$ is $r\times r$. By \Cref{lem:reg-sing-connection}, it is $\mc C^{\mathrm{sing}}_\alpha(\omega_0)$ that occurs as this residue, namely, up to the scalar $-v_b^2/(2\omega_0)$, the compression of the exterior DtN residue $\mathcal R_{-1}^\alpha$ to the interior Neumann trace space $\mathscr G(\omega_0)$.

The kernels of the two matrices have a transparent interpretation. A vector in $\ker\mc C^{\mathrm{sing}}_\alpha(\omega_0)=\ker\widetilde\Gamma_\alpha$ represents an interior Neumann mode that does not couple, at leading order, to the exterior Dirichlet pole; symmetrically, a vector in $\ker\mc C^{\mathrm{sing,ext}}_\alpha(\omega_0)=\ker\widetilde\Gamma_\alpha^\ast$ represents an exterior Dirichlet mode that does not couple to the interior Neumann trace space. Borrowing the language of dark states, we call these the interior \emph{dark} Neumann subspace and the exterior \emph{dark} Dirichlet subspace, respectively. Together, they form $\ker H_\alpha$ (see \eqref{eq:leading-coupling} for the definition of $H_\alpha$). Hence, the associated zero-leading cluster has no $\O(\delta^{1/2})$ splitting and consists of $m+r-2p_\alpha$ branches satisfying $\omega-\omega_0=\O(\delta)$. Their individual shifts and leading coefficient vectors are determined by the next-order reduction on $\ker H_\alpha$.
\end{remark}

\subsubsection*{Case 3: Exterior Dirichlet eigenvalue, interior regular}

Assume that
\begin{equation}\label{eq:case3-assumption}
\frac{\omega_0^2}{v_b^2}\notin \sigma_{\rm Neu}(-\Delta;D), \qquad
\frac{\omega_0^2}{v^2}\in \sigma_{\alpha,\mathrm{Dir}}(-\Delta;\Omega),
\end{equation}
for some fixed $\alpha\in Y^*$. Let $v_1^\alpha,\dots,v_r^\alpha$ and $\psi_s^\alpha=\partial_\nu v_s^\alpha|_{\partial D}$ be as in Case~2. Since $\omega_0^2/v_b^2$ is not an interior Neumann eigenvalue, $\Lambda_{\rm in}^{\omega_0}$ is invertible on $H^{1/2}(\partial D)$, and the interior Neumann problem
\begin{equation*}
\bigl(\Delta+\omega_0^2/v_b^2\bigr)\Phi_s^\alpha=0 \quad\text{in } D,
\qquad \partial_\nu \Phi_s^\alpha=\psi_s^\alpha \quad\text{on }\partial D,
\end{equation*}
admits a unique solution such that  $\Phi_s^\alpha:=(\Lambda_{\rm in}^{\omega_0})^{-1}\psi_s^\alpha$ for each $s=1,\dots,r$.

In this case, we introduce the \emph{exterior-mode $\alpha$-quasiperiodic frequency-dependent capacitance matrix}
\begin{equation}\label{eq:case3-matrix}
\mc{B}_\alpha(\omega_0)
= \bigl(\mc{B}_\alpha^{st}(\omega_0)\bigr)_{1\le s,t\le r},
\qquad
\mc{B}_\alpha^{st}(\omega_0) := -\frac{v^2}{2\omega_0}\,
\bigl\langle (\Lambda_{\rm in}^{\omega_0})^{-1}\psi_t^\alpha,\psi_s^\alpha\bigr\rangle_{\partial D}.
\end{equation}
The following results hold.

\begin{lemma}
Assuming \eqref{eq:case3-assumption}, $\mc{B}_\alpha(\omega_0)$ is Hermitian.
\end{lemma}

\begin{theorem}
\label{thm:case3-asymptotics}
Assume that \eqref{eq:case3-assumption} holds. Let
$\lambda_1^\alpha,\dots,\lambda_r^\alpha$ be the eigenvalues of
$\mc{B}_\alpha(\omega_0)$, counted with their multiplicities. Then the Bloch eigenfrequencies converging to $\omega_0$ satisfy 
\begin{equation*}
    \omega_j^\alpha(\delta) =  \omega_0+\delta\,\lambda_j^\alpha+\O(\delta^2), \qquad j=1,\dots,r, \quad \text{as $\delta\to 0$.}
\end{equation*}
\end{theorem}

\begin{proof}
Since $\Lambda_{\rm in}^{\omega}$ is invertible for $\omega$ in a neighborhood of $\omega_0$, equation \eqref{eq:Tdelta-alpha} gives
\begin{equation*}
g=\delta\,(\Lambda_{\rm in}^{\omega})^{-1}\Lambda_{\rm ext}^{\alpha,\omega}g.
\end{equation*}
Substituting the pole expansion \eqref{eq:Lambda-ext-pole}, and multiplying by $(\omega-\omega_0)$, we obtain
\begin{equation*}
(\omega-\omega_0)\,g
=\delta\,(\Lambda_{\rm in}^{\omega_0})^{-1}\mathcal R_{-1}^\alpha g
+\delta(\omega-\omega_0)\,K(\omega)\,g,
\end{equation*}
with $K(\omega)$ being analytic at $\omega_0$. Using \eqref{eq:residue-explicit}, the leading operator $(\Lambda_{\rm in}^{\omega_0})^{-1}\mathcal R_{-1}^\alpha$ has a finite-dimensional range $V:=\mathrm{span}\{\Phi_1^\alpha|_{\partial D},\dots,\Phi_r^\alpha|_{\partial D}\}\subset H^{1/2}(\partial D)$.
Decompose $H^{1/2}(\partial D)=V\oplus V^\perp$, write $g=g_V+g_\perp$ with
$g_V=\sum_{s=1}^r c_s\Phi_s^\alpha$, and apply the Lyapunov--Schmidt reduction.
Then, projecting onto $V^\perp$ yields
\begin{equation*}
\|g_\perp\|=\O\bigl(\delta\,\|g_V\|\bigr),
\end{equation*}
and projecting onto $V$ yields the finite-dimensional pencil
\begin{equation*}
\bigl((\omega-\omega_0)\,I-\delta\,\mc{B}_\alpha(\omega_0)\bigr)c
=
\O\bigl((\delta|\omega-\omega_0|+\delta^2)|c|\bigr).
\end{equation*}
The proof is complete. 
\end{proof}

\begin{remark}
We conjecture that Case~1 and Case~3 are generic in the sense that, for any fixed reference frequency $\omega_0>0$, Case~2 holds only on a set of quasiperiodicities $\alpha\in Y^*$ of Lebesgue measure zero. This conjecture would be implied by the absolute continuity of the spectrum of the Dirichlet Laplacian on the periodic domain $\R^3\setminus\bigcup_{\gamma\in\Lambda}(\overline D+\gamma)$ (\emph{i.e.}, no flat bands), together with the real-analyticity of its Bloch band functions on $Y^*$. 
\end{remark}

Finally, we describe the leading-order structure of the Bloch eigenmodes of \eqref{eq:periodic-bloch-eigenproblem}. In Case~1, the exterior problem is regular, and the leading-order mode is generated by the interior Neumann traces, just as in the finite-resonator problem. In Case~2, the positive eigenvalues of $\mc C^{\mathrm{sing}}_\alpha(\omega_0)$ produce a different behavior: the exterior Dirichlet pole amplifies the exterior component by $(\omega-\omega_0)^{-1}=\O(\delta^{-1/2})$, so that the normalized leading mode is supported in the exterior cell $\Omega$. For Case~3 branches associated with nonzero eigenvalues of $\mc B_\alpha(\omega_0)$, the same pole-expansion argument as in Case~2 shows that the normalized eigenmodes are exterior-cell-dominated: their exterior components converge to the corresponding exterior Dirichlet modes, whereas their interior components are $\O(\delta)$. We omit the details.

\begin{proposition}[(Bloch eigenmodes in Case~1)]\label{prop:bloch-eigenmodes}
Assume \eqref{eq:case1-assumption}. Let $\lambda_j^\alpha$ be a simple eigenvalue of $\mc C^{\mathrm{reg}}_\alpha(\omega_0)$, let $a=(a_p)_{p=1}^m$ be a corresponding eigenvector normalized by $\|a\|_2=1$, and let $\omega_j^\alpha(\delta)$ be the associated Bloch eigenfrequency in \eqref{approx:alphaj}. Define
\begin{equation*}
U_a^\alpha(x)
:=
\sum_{p=1}^m a_p
\begin{cases}
u_p(x),
& x\in D,\\[2pt]
\S_D^{\alpha,\,\omega_0/v}
\bigl[(\S_D^{\alpha,\,\omega_0/v})^{-1}[g_p]\bigr](x),
& x\in\Omega.
\end{cases}
\end{equation*}
If $u_{\delta,j}^\alpha$ is a corresponding Bloch eigenmode normalized by $\|u_{\delta,j}^\alpha\|_{L^2(Y)}=1$, then, after multiplying it by a phase factor if necessary,
\begin{equation}\label{eq:case1-mode-expansion}
u_{\delta,j}^\alpha
=
\frac{U_a^\alpha}{\|U_a^\alpha\|_{L^2(Y)}}
+\O(\delta)
\qquad
\text{in }H^1(D)\oplus H^1(\Omega).
\end{equation}

More generally, for an eigenvalue $\lambda^\alpha$ of arbitrary multiplicity, set
\begin{equation*}
E_{\lambda^\alpha} :=
\ker\bigl(\mc C^{\mathrm{reg}}_\alpha(\omega_0)
-\lambda^\alpha I_m\bigr)\,, \qquad \mathscr U_{\lambda^\alpha} :=
\left\{\frac{U_a^\alpha}{\|U_a^\alpha\|_{L^2(Y)}}:
a\in E_{\lambda^\alpha},\ \|a\|_2=1\right\}. 
\end{equation*}
Every $L^2(Y)$-normalized Bloch eigenmode associated with a branch satisfying $\omega^\alpha(\delta) =\omega_0+\delta\lambda^\alpha+\O(\delta^2)$ satisfies
\begin{equation*}
\operatorname{dist}_{H^1(D)\oplus H^1(\Omega)}
\bigl(u_\delta^\alpha,\mathscr U_{\lambda^\alpha}\bigr) =\O(\delta).
\end{equation*}
\end{proposition}

\begin{proof}
Let $\widetilde u_\delta^\alpha$ be a representative of the Bloch eigenmode, let $g_\delta:=\widetilde u_\delta^\alpha|_{\partial D}$, and scale $\widetilde u_\delta^\alpha$ so that the coefficient vector $c(\delta)=(c_p(\delta))_{p=1}^m$ of the component of $g_\delta$ in $\mathscr G(\omega_0)$ satisfies $\|c(\delta)\|_2=1$. The
Lyapunov--Schmidt estimate in the proof of \cref{thm:case1-asymptotics} and
$|\omega_j^\alpha(\delta)-\omega_0|=\O(\delta)$ give
\begin{equation*}
g_\delta
=
\sum_{p=1}^m c_p(\delta)g_p+\O(\delta)
\qquad\text{in }H^{1/2}(\partial D).
\end{equation*}
Moreover, \eqref{approx:alphaj} and \eqref{approx:bloch} imply
\begin{equation*}
\operatorname{dist}_{\ell^2}\bigl(c(\delta),E_{\lambda_j^\alpha}\bigr)
=\O(\delta).
\end{equation*}
If $\lambda_j^\alpha$ is simple, then, after choosing the phase, $c(\delta)=a+\O(\delta)$. Since the interior and exterior Dirichlet solution operators depend analytically on $\omega$ near $\omega_0$, it follows that
\begin{equation*}
\widetilde u_\delta^\alpha
=U_a^\alpha+\O(\delta)
\qquad\text{in }H^1(D)\oplus H^1(\Omega).
\end{equation*}
Since $u_1,\ldots,u_m$ are $L^2(D)$-orthonormal, $\|U_a^\alpha\|_{L^2(Y)}\geq1$. Normalizing in $L^2(Y)$ therefore proves \eqref{eq:case1-mode-expansion}. The same argument, with a unit vector in $E_{\lambda^\alpha}$ allowed to depend on $\delta$, proves the result for an eigenvalue of arbitrary multiplicity.
\end{proof}

\begin{proposition}[(Bloch eigenmodes in Case~2)]\label{prop:case2-bloch-eigenmodes}
Assume \eqref{eq:case2-assumption}. Let $\lambda>0$ be a simple eigenvalue of $\mc C^{\mathrm{sing}}_\alpha(\omega_0)$, and let $a=(a_p)_{p=1}^m$ be a corresponding eigenvector, normalized by $\|a\|_2=1$. Set
\begin{equation*}
    b:=\Gamma_\alpha a\in\mathbb C^r,
    \qquad
    V_b^\alpha:=\sum_{s=1}^r b_s\, v_s^\alpha ,
\end{equation*}
so that $b\neq0$ and $\|V_b^\alpha\|_{L^2(\Omega)}=\|b\|_2$. Let
\begin{equation*}
    \omega_\pm^\alpha(\delta)
    =
    \omega_0\pm \delta^{1/2}\sqrt{\lambda}+\O(\delta)
\end{equation*}
be the two Bloch eigenfrequency branches associated with $\lambda$ in \Cref{thm:case2-asymptotics}. If $u_{\delta,\pm}^\alpha$ is a corresponding Bloch eigenmode normalized by $\|u_{\delta,\pm}^\alpha\|_{L^2(Y)}=1$, then, after multiplying by a phase factor if necessary,
\begin{equation}\label{eq:case2-exterior-mode}
    u_{\delta,\pm}^\alpha
    =
    \frac{V_b^\alpha}{\|V_b^\alpha\|_{L^2(\Omega)}}
    +\O(\delta^{1/2})
    \qquad \text{in } H^1(\Omega),
\end{equation}
and
\begin{equation}\label{eq:case2-interior-small}
    u_{\delta,\pm}^\alpha
    =
    \O(\delta^{1/2})
    \qquad \text{in } H^1(D).
\end{equation}
In particular, $u_{\delta,\pm}^\alpha\big|_{\partial D} =\O(\delta^{1/2})$ in $H^{1/2}(\partial D)$.  That is, the normalized Bloch eigenmodes, corresponding to $\O(\delta^{1/2})$-split branches, are exterior-cell-dominated at leading order.
\end{proposition}

\begin{proof}
Since $a$ is an eigenvector of $\mc C^{\mathrm{sing}}_\alpha(\omega_0)=\frac{v_b^2v^2}{4\omega_0^2}\Gamma_\alpha^\ast\Gamma_\alpha$ with $\lambda>0$,
\begin{equation*}
    \frac{v_b^2v^2}{4\omega_0^2}\,\|\Gamma_\alpha a\|_2^2
    =
    a^\ast \mc C^{\mathrm{sing}}_\alpha(\omega_0)\, a
    =
    \lambda\,\|a\|_2^2
    >0,
\end{equation*}
so $b=\Gamma_\alpha a\neq0$; since $v_1^\alpha,\dots,v_r^\alpha$ are $L^2(\Omega)$-orthonormal, $\|V_b^\alpha\|_{L^2(\Omega)}=\|b\|_2$.

Let $\varepsilon:=\delta^{1/2}$ and choose a representative $\widetilde u_{\delta,\pm}^\alpha$ of the corresponding Bloch eigenspace by normalizing its interior coefficient block. More precisely, if $a_\delta$ denotes the interior coefficient vector arising from the Lyapunov--Schmidt reduction in the proof of \cref{thm:case2-asymptotics}, then the simplicity of $\lambda$ implies that $\pm\sqrt\lambda$ are simple eigenvalues of $H_\alpha$. Thus, we may rescale $\widetilde u_{\delta,\pm}^\alpha$ and choose its phase so that
\begin{equation*}
    a_\delta=a+\O(\varepsilon).
\end{equation*}
If $\widetilde g_{\delta,\pm}$ denotes the boundary trace of this
coefficient-normalized representative, then
\begin{equation*}
    \widetilde g_{\delta,\pm}
    =
    \sum_{p=1}^m a_p g_p+\O(\varepsilon)
    \qquad\text{in }H^{1/2}(\partial D).
\end{equation*}
Since the interior solution operator is regular at $\omega_0$, it follows that
\begin{equation*}
    \widetilde u_{\delta,\pm}^\alpha
    =
    \sum_{p=1}^m a_pu_p+\O(\varepsilon)
    \qquad\text{in }H^1(D).
\end{equation*}

On the exterior side, the $\alpha$-quasiperiodic Dirichlet solution operator inherits the simple pole of $\Lambda_{\rm ext}^{\alpha,\omega}$: for $f\in H^{1/2}(\partial D)$, the exterior solution with boundary value $f$ admits the expansion
\begin{equation*}
    u_{\rm ext}^{\alpha,\omega}[f]
    =
    -\frac{v^2}{2\omega_0(\omega-\omega_0)}
    \sum_{s=1}^r
    \langle f,\psi_s^\alpha\rangle_{\partial D}\,v_s^\alpha
    +\O\bigl(\|f\|_{H^{1/2}(\partial D)}\bigr)
    \qquad \text{in } H^1(\Omega),
\end{equation*}
with remainder uniformly bounded for $\omega$ near $\omega_0$; this is consistent with \eqref{eq:residue-explicit}, since $\partial_\nu v_s^\alpha|_{\partial D}=\psi_s^\alpha$. By \eqref{eq:Gamma-alpha} and the expansion of $\widetilde g_{\delta,\pm}$,
\begin{equation*}
    \langle \widetilde g_{\delta,\pm},\psi_s^\alpha\rangle_{\partial D} =  (\Gamma_\alpha a)_s+\O(\varepsilon)
    = b_s+\O(\varepsilon),
\end{equation*}
so that in $H^1(\Omega)$, 
\begin{equation*}
    \widetilde u_{\delta,\pm}^\alpha   =  A_{\delta,\pm}V_b^\alpha
    +\O(1)\,, \quad \text{with }\  A_{\delta,\pm} :=  -\frac{v^2}{2\omega_0\bigl(\omega_\pm^\alpha(\delta)-\omega_0\bigr)}.
\end{equation*}
Since $\omega_\pm^\alpha(\delta)-\omega_0  =\pm\varepsilon\sqrt\lambda+\O(\varepsilon^2)$, we have
\begin{equation*}
    \|\widetilde u_{\delta,\pm}^\alpha\|_{L^2(Y)}
    =
    |A_{\delta,\pm}|\,\|V_b^\alpha\|_{L^2(\Omega)}
    \bigl(1+\O(\varepsilon)\bigr).
\end{equation*}
The $L^2(Y)$-normalized Bloch mode is therefore
\begin{equation*}
    u_{\delta,\pm}^\alpha
    =
    e^{\mathrm i\theta_{\delta,\pm}}
    \frac{\widetilde u_{\delta,\pm}^\alpha}
    {\|\widetilde u_{\delta,\pm}^\alpha\|_{L^2(Y)}},
\end{equation*}
where the phase is chosen so that $e^{\mathrm i\theta_{\delta,\pm}}A_{\delta,\pm}/|A_{\delta,\pm}|=1$. It follows that
\begin{equation*}
    u_{\delta,\pm}^\alpha
    =
    \frac{V_b^\alpha}{\|V_b^\alpha\|_{L^2(\Omega)}}
    +\O(\varepsilon)
    \qquad\text{in }H^1(\Omega),
\end{equation*}
while $ u_{\delta,\pm}^\alpha=\O(\varepsilon)$ in $H^1(D)$.  Since $\varepsilon=\delta^{1/2}$, this proves \eqref{eq:case2-exterior-mode} and \eqref{eq:case2-interior-small}. The proof is complete. 
\end{proof}

\begin{remark}
The simplicity assumption in \Cref{prop:case2-bloch-eigenmodes} is only used to avoid extra notation: for a positive eigenvalue of higher multiplicity, the same conclusion holds for any branch whose interior coefficient vector converges to a fixed vector in the corresponding eigenspace.

By contrast, the zero-leading subspace $\ker\widetilde\Gamma_\alpha\oplus \ker\widetilde\Gamma_\alpha^\ast$ identified in \Cref{thm:case2-asymptotics} accounts for the remaining $m+r-2p_\alpha$ branches, whose frequency shifts are $\O(\delta)$. Their leading shifts and coefficient vectors are determined by the next-order effective pencil restricted to this subspace. Since this pencil may couple the two summands, we do not assert that these modes admit a universal classification as either interior- or exterior-cell-dominated.
\end{remark}

\subsection{Bandgap opening} \label{sec:bandgap}
Throughout this subsection, we restrict ourselves to the simplest setting: a lattice $\Lambda\subset\R^3$ with a \emph{single} resonator per unit cell ($N=1$). Let
\begin{equation*}
    0<\mu_1<\mu_2<\cdots
\end{equation*}
be the distinct positive Neumann eigenvalues of $-\Delta$ on $D$, and set
\begin{equation*}
    \omega_n^{\rm Neu}:=v_b\sqrt{\mu_n},\qquad n\ge 1,
\end{equation*}
so that $\omega_n^{\rm Neu}$ is the corresponding interior Neumann frequency at the material wave speed $v_b$. We assume that the first $J\ge 2$ eigenvalues $\mu_1,\dots,\mu_J$ are simple, satisfying
\begin{equation}\label{eq:bandgap-interior-regularity}
    \mu_n\notin\sigma_{\rm Dir}(-\Delta;D),
    \qquad 1\le n\le J,
\end{equation}
and impose the stronger uniform exterior nonresonance condition
\begin{equation}\label{eq:bandgap-case1-uniform}
\left[\frac{(\omega_1^{\rm Neu})^2}{v^2},\frac{(\omega_J^{\rm Neu})^2}{v^2}\right]
\cap \sigma_{\alpha,\mathrm{Dir}}(-\Delta;\Omega) = \varnothing
\qquad\text{for every }\alpha\in Y^*.
\end{equation}

\begin{remark} \label{rem:sufficient-case1}
    A simple sufficient condition for \eqref{eq:bandgap-case1-uniform} is that the first $J$ scaled interior Neumann eigenvalues $(v_b^2/v^2)\mu_n$ lie below the bottom of the exterior quasiperiodic Dirichlet spectrum. Specifically, let
\begin{equation*}
    \lambda_{\mathrm{ext},1}^D
:=
\min_{\alpha\in Y^*}\lambda_1^D(\alpha),
\end{equation*}
where $\lambda_1^D(\alpha)$ is the first eigenvalue of $-\Delta$ on $\Omega$ with the Dirichlet boundary condition on $\partial D$ and the $\alpha$-quasiperiodic boundary condition on $\partial Y$. Then,  \eqref{eq:bandgap-case1-uniform} holds provided that
\begin{equation*}
    \frac{v_b^2}{v^2}\,\mu_J<\lambda_{\mathrm{ext},1}^D,
\qquad\text{or, equivalently,}\qquad
\frac{v_b}{v}<\sqrt{\frac{\lambda_{\mathrm{ext},1}^D}{\mu_J}}.
\end{equation*}
Since $\lambda_{\mathrm{ext},1}^D>0$ and $\mu_J$ is fixed, this inequality is satisfied whenever the wave-speed ratio $v_b/v$ is sufficiently small. 
\end{remark}

Under \eqref{eq:bandgap-case1-uniform}, compactness of $Y^*$ and the continuous dependence of the exterior Dirichlet eigenvalues on $\alpha$ give a uniform positive distance between the interval in \eqref{eq:bandgap-case1-uniform} and the exterior Dirichlet spectrum. Consequently, the exterior quasiperiodic DtN maps are holomorphic with uniform local operator bounds near $\omega_n^{\rm Neu}$, $1\le n\le J$.

For each $1\le n\le J$, set $z:=\omega-\omega_n^{\rm Neu}$. By \eqref{eq:bandgap-interior-regularity}, the exact Lyapunov--Schmidt reduction of \eqref{eq:Tdelta-alpha}, normalized using \eqref{eq:interior-derivative-normalization}, gives a scalar Schur complement
\begin{equation*}
    f_n^\alpha(z,\delta)
    =
    z-\delta\,\mc C_\alpha^{\mathrm{reg}}(\omega_n^{\rm Neu})
    +r_n^\alpha(z,\delta).
\end{equation*}
Here $f_n^\alpha$ is analytic in $(z,\delta)$ and continuous in $\alpha$, and there exist constants $\rho,C_{\rm LS}>0$, independent of $n$ and $\alpha$, such that
\begin{equation*}
    |r_n^\alpha(z,\delta)|
    \leq
    C_{\rm LS}\bigl(|z|^2+|\delta||z|+|\delta|^2\bigr),
    \qquad |z|+|\delta|<\rho.
\end{equation*}
Moreover, $\omega_n^{\rm Neu}+z$ is a Bloch eigenfrequency near $\omega_n^{\rm Neu}$ if and only if $f_n^\alpha(z,\delta)=0$, with the same algebraic multiplicity. Since $\partial_zf_n^\alpha(0,0)=1$, the uniform analytic implicit-function theorem yields, in a fixed neighborhood independent of $\alpha$, a unique simple zero that is continuous in $\alpha$. Thus, for every $n=1,\dots,J$, there is a unique resonant Bloch band function $\alpha\mapsto \omega^\alpha_n(\delta)$ on $Y^*$ that converges to $\omega_n^{\rm Neu}$ as $\delta\to 0$, and
\begin{equation}\label{eq:bandgap-band}
\omega^\alpha_n(\delta)
=
\omega_n^{\rm Neu}+\delta\,\mc C^{\mathrm{reg}}_\alpha(\omega_n^{\rm Neu})+\O(\delta^2)
\end{equation}
uniformly in $\alpha\in Y^*$. Here, $\mc C^{\mathrm{reg}}_\alpha(\omega_n^{\rm Neu})\in \R$ is the scalar regular quasiperiodic frequency-dependent capacitance coefficient at $\omega_n^{\rm Neu}$, which is real-valued by \cref{lem:hermitan}.

We first state a uniform boundedness property.

\begin{lemma}
\label{lem:uniform-periodic-bound-revised}
Under \eqref{eq:bandgap-case1-uniform}, there exists a constant $M_J>0$ such that
\begin{equation*}
    \bigl|\mc C^{\mathrm{reg}}_\alpha(\omega_n^{\rm Neu})\bigr|\le M_J
\qquad \text{for every } 1\le n\le J \text{ and } \alpha\in Y^*.
\end{equation*}
\end{lemma}

\begin{proof}
Let $g_n$ be the trace of the $L^2(D)$-normalized Neumann eigenfunction associated with $\mu_n$. By \eqref{eq:case1-cap-matrix},
\begin{equation*}
    \bigl|\mc C^{\mathrm{reg}}_\alpha(\omega_n^{\rm Neu})\bigr|
    \leq
    \frac{v_b^2}{2\omega_n^{\rm Neu}}
    \bigl\|\Lambda_{\rm ext}^{\alpha,\omega_n^{\rm Neu}}\bigr\|_{H^{1/2}\to H^{-1/2}}
    \|g_n\|_{H^{1/2}(\partial D)}^2.
\end{equation*}
The DtN norms are uniformly bounded by the spectral separation in \eqref{eq:bandgap-case1-uniform}, and the set $1\le n\le J$ is finite.
\end{proof}

The next theorem gives the separation of the resonant Bloch bands bifurcating from $\omega_n^{\rm Neu}$ and $\omega_{n+1}^{\rm Neu}$ and shows that the intervening interval is a spectral bandgap. It
is the non-subwavelength analogue of the bandgap opening mechanism in
\cite{ammari2017subwavelength}.

\begin{theorem}
\label{thm:uniform-bandgap-revised}
Assume that $\mu_1,\dots,\mu_J$ are simple and that \eqref{eq:bandgap-interior-regularity} and \eqref{eq:bandgap-case1-uniform} hold, and set
\begin{equation*}
    \gamma_J:=\min_{1\le n\le J-1}\bigl(\omega_{n+1}^{\rm Neu}-\omega_n^{\rm Neu}\bigr)>0.
\end{equation*}
Then there exists $\delta_J>0$ such that, for any $0<\delta<\delta_J$ and $1\le n\le J-1$,
\begin{equation}\label{eq:bandgap-strict}
\sup_{\alpha\in Y^*}\omega^\alpha_n(\delta)
\;<\;
\inf_{\alpha\in Y^*}\omega^\alpha_{n+1}(\delta).
\end{equation}
Moreover, the open interval
\begin{equation} \label{def:gapinterval}
    \left(
    \sup_{\alpha\in Y^*}\omega_n^\alpha(\delta),
    \inf_{\alpha\in Y^*}\omega_{n+1}^\alpha(\delta)
    \right)
\end{equation}
contains no Bloch eigenfrequency of \eqref{eq:periodic-bloch-eigenproblem} for any $\alpha\in Y^*$.
\end{theorem}

\begin{proof}
By \eqref{eq:bandgap-band} and \cref{lem:uniform-periodic-bound-revised}, there exist constants $C_J>0$ and $\delta_J^{(0)}>0$ such that
\begin{equation}\label{eq:bandgap-pointwise-est}
\bigl|\omega^\alpha_n(\delta)-\omega_n^{\rm Neu}\bigr|\le M_J\,\delta+C_J\,\delta^2
\qquad \text{for every } 1\le n\le J,\;\alpha\in Y^*,\;0<\delta<\delta_J^{(0)}.
\end{equation}
Hence, for every $1\le n\le J-1$, we have 
\begin{equation*}
\inf_{\alpha\in Y^*}\omega^\alpha_{n+1}(\delta)-\sup_{\alpha\in Y^*}\omega^\alpha_n(\delta)
\ge
\bigl(\omega_{n+1}^{\rm Neu}-\omega_n^{\rm Neu}\bigr)-2M_J\,\delta-2C_J\,\delta^2
\ge
\gamma_J-2M_J\,\delta-2C_J\,\delta^2.
\end{equation*}
Setting
\begin{equation*}
\delta_J:=\min\!\left\{\delta_J^{(0)},\;\frac{\gamma_J}{4M_J},\;\sqrt{\frac{\gamma_J}{8C_J}}\right\},
\end{equation*}
we obtain, for every $0<\delta<\delta_J$, that
\begin{equation*}
\inf_{\alpha\in Y^*}\omega^\alpha_{n+1}(\delta)-\sup_{\alpha\in Y^*}\omega^\alpha_n(\delta)
\ge \frac{\gamma_J}{4}>0,
\end{equation*}
which proves \eqref{eq:bandgap-strict}.

It remains to exclude other Bloch eigenfrequencies from this interval. Choose $\eta>0$ such that the intervals
\begin{equation*}
    I_j:=\bigl(\omega_j^{\rm Neu}-\eta,\omega_j^{\rm Neu}+\eta\bigr),
    \qquad 1\le j\le J,
\end{equation*}
are pairwise disjoint and the exact scalar reduction above has in each $I_j$, for every $\alpha\in Y^*$ and all sufficiently small $\delta$, precisely one characteristic value, namely $\omega_j^\alpha(\delta)$. For $1\le n\le J-1$, the compact interval
\begin{equation*}
    K_n:=\bigl[\omega_n^{\rm Neu}+\eta,
    \omega_{n+1}^{\rm Neu}-\eta\bigr]
\end{equation*}
avoids the interior Neumann frequencies. Let $\mathcal N_{\rm in}^{\omega}$ denote the interior NtD map. On $K_n$, both $\mathcal N_{\rm in}^{\omega}$ and $\Lambda_{\rm ext}^{\alpha,\omega}$ are uniformly bounded in $\omega$ and $\alpha$. If $u$ is a Bloch eigenmode and $q:=\partial_\nu u|_-$, then
\begin{equation*}
    \bigl(I-\delta\Lambda_{\rm ext}^{\alpha,\omega}
    \mathcal N_{\rm in}^{\omega}\bigr)q=0.
\end{equation*}
For sufficiently small $\delta$, the operator in parentheses is uniformly invertible by a Neumann series, so $q=0$ and hence $u=0$. Thus, $K_n$ is spectrally empty. After decreasing $\delta_J$ over the finitely many $n$, if necessary, \eqref{eq:bandgap-pointwise-est} places the two relevant band functions in $I_n$ and $I_{n+1}$. It follows that the interval \eqref{def:gapinterval} is a spectral bandgap, which completes the proof. 
\end{proof}

\subsection{Generic Dirac degeneracies in a honeycomb lattice}

In this subsection, we consider the two-dimensional case and extend the subwavelength honeycomb analysis of \cite{honeycomb1} beyond the subwavelength regime through the frequency-dependent capacitance matrix. All lattice objects below are therefore understood in dimension two. We use the standard honeycomb lattice vectors
\begin{equation*}
    \ell_1=a\left(\frac{\sqrt3}{2},\frac12\right),
    \qquad
    \ell_2=a\left(\frac{\sqrt3}{2},-\frac12\right),
    \qquad
    \Lambda=\mathbb Z\ell_1\oplus\mathbb Z\ell_2,
\end{equation*}
where $a>0$ is the lattice constant. Following the configuration in \cite{honeycomb1}, assume that the periodic structure is honeycomb, with two congruent resonators $D_1,D_2\subset Y$ in each cell. Let $Y=Y_1\cup Y_2$ be the standard triangular decomposition of the honeycomb cell, chosen so that $D_j\subset Y_j$, and set
\begin{equation*}
    \Gamma_3:=\partial Y_1\cap\partial Y_2.
\end{equation*}
We denote by $R_3$ the reflection across $\Gamma_3$. Specifically,
\begin{assumption} \label{assumphoneycomb}
Let $\Lambda^*=\mathbb Zb_1\oplus\mathbb Zb_2$ be the reciprocal lattice,
where $b_i\cdot\ell_j=2\pi\delta_{ij}$, and choose
\begin{equation*}
    \mathrm K=\frac{2b_1+b_2}{3},
    \qquad
    \mathrm K'=\frac{b_1+2b_2}{3}
    \equiv-\mathrm K\pmod{\Lambda^*}
\end{equation*}
as representatives of the two inequivalent Brillouin-zone corners.
Choose the origin at a threefold rotation center and let $R$ be the
counterclockwise rotation by $2\pi/3$. Then $R\Lambda=\Lambda$ and
$R\mathrm K\equiv\mathrm K$, $R\mathrm K'\equiv\mathrm K'$ modulo
$\Lambda^*$.
We suppose that
\begin{enumerate}
    \item[(i)] the full periodic configuration $(D_1\cup D_2)+\Lambda$ is
    invariant under $R$, and each $D_j$ is invariant under the local
    threefold rotation $R_j$ about its center $z_j$;
    \item[(ii)] the pair $D_1\cup D_2$ is invariant under inversion through
    the center of the cell and under the reflection $R_3$, which interchanges
    the two resonators: $R_3(D_1)=D_2$;
    \item[(iii)] $\omega_0>0$ is a simple Neumann eigenfrequency of a single resonator, and 
\begin{equation}\label{eq:honeycomb-regular-assumption}
\frac{\omega_0^2}{v_b^2}\in \sigma_{\rm Neu}(-\Delta;D_1)
=\sigma_{\rm Neu}(-\Delta;D_2),
\qquad
\frac{\omega_0^2}{v_b^2}\notin
\sigma_{\rm Dir}(-\Delta;D_1)\cup\sigma_{\rm Dir}(-\Delta;D_2),
\end{equation}
and the exterior problems at the two high-symmetry points are regular:
\begin{equation}\label{eq:honeycomb-exterior-regular-assumption}
\frac{\omega_0^2}{v^2}\notin \sigma_{\alpha,\mathrm{Dir}}(-\Delta;\Omega)
\quad \text{for }\alpha\in\{\mathrm{K},\mathrm{K}'\}.
\end{equation}
\end{enumerate}
Let $\widetilde f$ denote the $\alpha$-quasiperiodic extension of $f$,
defined by $\widetilde f(x+\ell)=e^{\mathrm i\alpha\cdot\ell}f(x)$ for
$\ell\in\Lambda$. Let $U_R^\alpha$ be the unitary Bloch action obtained
by applying $f\mapsto\widetilde f\circ R^{-1}$ and restricting back to
the reference cell. On the two-dimensional resonator-indexed mode space at $\mathrm K$, we use the basis in which
\begin{equation*}
    U_R^{\mathrm K}=\diag(\tau,\overline\tau),
    \qquad
    \tau=e^{2\pi\mathrm i/3}.
\end{equation*}
The conjugate representation is used at $\mathrm K'$.
\end{assumption}

\begin{remark}\label{rem:sufficietcond}
    We emphasize that the simplicity condition in \Cref{assumphoneycomb} (iii) holds generically \cite{generic,teytel1999rare}. As in \cref{rem:sufficient-case1}, the exterior non-resonance condition \eqref{eq:honeycomb-exterior-regular-assumption} can be enforced by taking the wave-speed ratio $v_b/v$ sufficiently small. By continuity of the exterior Dirichlet eigenvalues with respect to $\alpha$,
    the condition \eqref{eq:honeycomb-exterior-regular-assumption}
  at $\mathrm{K}$ and $\mathrm{K}'$ persists in neighborhoods of these points, where the regular quasiperiodic frequency-dependent capacitance matrix $\mc C^{\mathrm{reg}}_\alpha(\omega_0)$ is therefore well defined.
\end{remark}

By condition (iii), the two copies of the simple interior mode give a two-dimensional limiting space, spanned by the corresponding boundary traces. In the basis indexed by
the two resonators, we write
\begin{equation*}
    \mc C^\alpha_{\rm hc}(\omega_0):=\mc C^{\mathrm{reg}}_\alpha(\omega_0)
    \in \mathbb C^{2\times 2}.
\end{equation*}
\begin{proposition}
\label{prop:honeycomb-normal-form}
Under the honeycomb symmetry \Cref{assumphoneycomb} above, there exist orthonormal bases of the two-dimensional mode spaces at $\mathrm{K}$ and $\mathrm{K}'$ such that
\begin{equation}\label{eq:honeycomb-degenerate-cap}
    \mc C^\mathrm{K}_{\rm hc}(\omega_0)=c_\mathrm{K}(\omega_0)\,I_2,
\qquad
    \mc C^{\mathrm{K}'}_{\rm hc}(\omega_0)=c_{\mathrm{K}'}(\omega_0)\,I_2,
\end{equation}
for some $c_\mathrm{K}(\omega_0),c_{\mathrm{K}'}(\omega_0)\in\mathbb R$. Moreover, near the $\mathrm{K}$ point, the following expansion holds:
\begin{equation}\label{eq:honeycomb-linear-normal-form}
\mc C^{\mathrm{K}+\xi}_{\rm hc}(\omega_0)
=
c_\mathrm{K}(\omega_0)\,I_2
+
\begin{pmatrix}
0 & \gamma_\mathrm{K}(\omega_0)(\xi_1-i\xi_2)\\
\overline{\gamma_\mathrm{K}(\omega_0)}(\xi_1+i\xi_2) & 0
\end{pmatrix}
+\O(|\xi|^2)
\end{equation}
for some $\gamma_\mathrm{K}(\omega_0)\in\mathbb C$. The analogous expansion holds at $\mathrm{K}'$, with a coefficient $\gamma_{\mathrm{K}'}(\omega_0)\in\mathbb C$.
\end{proposition}

\begin{proof}
The quasiperiodic exterior problem and the interior Neumann eigenspaces are covariant under the honeycomb symmetries. With the induced action $U_R^\alpha$ defined in \Cref{assumphoneycomb}, we have
\begin{equation} \label{eq:covarianeq}
    \mc C^{R\alpha}_{\rm hc}(\omega_0)
    =
    U_R^\alpha\,\mc C^\alpha_{\rm hc}(\omega_0)\,
    (U_R^\alpha)^{-1}.
\end{equation}
At the corner $\alpha=\mathrm{K}$, write $U_R:=U_R^{\mathrm K}$. Since
$R\mathrm{K}\equiv \mathrm{K}$ modulo the reciprocal lattice, we have
\begin{equation*}
\mc C^\mathrm{K}_{\rm hc}(\omega_0)\,U_R = U_R\,\mc C^\mathrm{K}_{\rm hc}(\omega_0).
\end{equation*}
In the resonator-indexed basis of \Cref{assumphoneycomb},
$U_R=\diag(\tau,\overline\tau)$. Since
$\tau\ne\overline\tau$, any matrix that commutes with
$\diag(\tau,\overline\tau)$ is itself diagonal, so
\begin{equation*}
    \mc C^\mathrm{K}_{\rm hc}(\omega_0) = \diag(c_1,c_2),
\qquad c_1,c_2\in\mathbb R,
\end{equation*}
where the diagonal entries are real by Hermiticity.

Let $P$ denote spatial inversion ($x\mapsto -x$ about the center of the unit cell) and $T$ complex conjugation. Both act on Bloch eigenspaces by $\alpha\mapsto -\alpha$, so the composition $S:=PT$ preserves the Bloch eigenspace at $\mathrm{K}$ and is anti-unitary. Moreover, $S$ commutes with both $U_R$ and $\mc C^\mathrm{K}_{\rm hc}(\omega_0)$. For any $U_R$-eigenvector $v$ with $U_R v=\tau v$, anti-linearity of $S$ gives
\begin{equation*}
U_R(Sv) = S(U_R v) = S(\tau v) = \overline\tau\,(Sv).
\end{equation*}
Hence, $S$ maps the $\tau$-eigenspace into the $\overline\tau$-eigenspace, and up to rescaling we may set $v_2 = Sv_1$. Applying the commutation $S\,\mc C^\mathrm{K}_{\rm hc}(\omega_0) = \mc C^\mathrm{K}_{\rm hc}(\omega_0)\,S$ to $v_1$ then yields
\begin{equation*}
c_2\,v_2
= \mc C^\mathrm{K}_{\rm hc}(\omega_0)\,v_2
= \mc C^\mathrm{K}_{\rm hc}(\omega_0)(Sv_1)
= S\,\mc C^\mathrm{K}_{\rm hc}(\omega_0)\,v_1
= S(c_1 v_1)
= \overline{c_1}\,(Sv_1)
= c_1\,v_2,
\end{equation*}
where the last equality uses $c_1\in\mathbb R$. Hence $c_1=c_2$, which proves \eqref{eq:honeycomb-degenerate-cap}. The same argument applies at $\mathrm{K}'$.

It remains to identify the first-order correction term in \eqref{eq:honeycomb-linear-normal-form}. Let 
\begin{equation*}
        L_\mathrm{K}(\xi):=D_\alpha\,\mc C^\alpha_{\rm hc}(\omega_0)\bigr|_{\alpha=\mathrm{K}}[\xi].
\end{equation*}
Differentiating the covariance relation \eqref{eq:covarianeq} at
$\alpha=\mathrm{K}$ and using
$\mc C^\mathrm{K}_{\rm hc}(\omega_0)=c_\mathrm{K}(\omega_0)I_2$
yields
\begin{equation*}
    L_\mathrm{K}(R\xi)=U_R\,L_\mathrm{K}(\xi)\,U_R^{-1}.
\end{equation*}
Indeed, the terms involving $D_\alpha U_R^\alpha$ cancel because the
matrix at $\mathrm K$ is scalar.
The diagonal entries of $L_\mathrm{K}(\xi)$ are real linear and invariant under $R$, so they vanish. The upper off-diagonal entry $\ell(\xi)$ satisfies
\begin{equation*}
    \ell(R\xi)=\tau^2\,\ell(\xi).
\end{equation*}
One can see that the space of complex-valued real-linear forms $\ell(\xi) = a \xi_1 + b \xi_2$, $a,b \in \mathbb{C}$, with the above transformation relation, is one-dimensional and is spanned by $\xi_1-i\xi_2$. By Hermiticity of $\mc C^\alpha_{\rm hc}(\omega_0)$ for real $\alpha$, the lower off-diagonal entry is the complex conjugate of the upper one, giving \eqref{eq:honeycomb-linear-normal-form}.
\end{proof}

\begin{lemma}
\label{lem:exact-honeycomb-crossing}
Under \Cref{assumphoneycomb}, suppose that
$\gamma_\mathrm{K}(\omega_0)\neq0$. For every sufficiently small real $\delta>0$,
the two Bloch eigenfrequency branches converging to $\omega_0$ may be labeled
so that
\begin{equation}\label{eq:exact-dirac-crossing}
    \omega_+^\mathrm{K}(\delta)=\omega_-^\mathrm{K}(\delta)
    =:\omega_{\rm D}^\mathrm{K}(\delta),
    \qquad
    \omega_{\rm D}^\mathrm{K}(\delta)
    =\omega_0+\delta c_\mathrm{K}(\omega_0)+\O(\delta^2),
\end{equation}
and, as $\xi\to0$,
\begin{equation}\label{eq:exact-dirac-cone}
    \omega_\pm^{\mathrm{K}+\xi}(\delta)
    =\omega_{\rm D}^\mathrm{K}(\delta)
    \pm v_{\rm D}^\mathrm{K}(\delta)|\xi|
    +\O_\delta(|\xi|^2),
    \qquad
    v_{\rm D}^\mathrm{K}(\delta)
    =\delta|\gamma_\mathrm{K}(\omega_0)|+\O(\delta^2)>0.
\end{equation}
The analogous statements hold at $\mathrm{K}'$.
\end{lemma}

\begin{proof}
Recall that $\mathscr G(\omega_0)  :=\ker\Lambda_{\rm in}^{\omega_0}  =\operatorname{span}\{g_1,g_2\} \subset H^{1/2}(\partial D)$. 
By \eqref{eq:Tdelta-alpha},
$\mathcal T_\delta^\alpha(\omega):H^{1/2}(\partial D)\to
H^{-1/2}(\partial D)$. Let $P$ be the $L^2(\partial D)$-orthogonal projection
onto $\mathscr G(\omega_0)$ and let $P^*$ be its dual projection on
$H^{-1/2}(\partial D)$. Thus $H^{1/2}(\partial D) =\operatorname{ran}P\oplus\ker P$ and $H^{-1/2}(\partial D) =\operatorname{ran}P^*\oplus\ker P^*$. With respect to these decompositions, we can write 
\begin{equation*}
    \mathcal T_\delta^\alpha(\omega)
    = \begin{pmatrix}A&B\\ C&D\end{pmatrix},
\end{equation*}
where $\mathcal T_\delta^\alpha(\omega)$ is defined as in \eqref{eq:Tdelta-alpha}, and 
\begin{equation*}
\begin{aligned}
    A&=P^*\mathcal T_\delta^\alpha(\omega)P,
    &B&=P^*\mathcal T_\delta^\alpha(\omega)(I-P),\\
    C&=(I-P^*)\mathcal T_\delta^\alpha(\omega)P,
    &D&=(I-P^*)\mathcal T_\delta^\alpha(\omega)(I-P).
\end{aligned}
\end{equation*}
The block $D$ is invertible near $(\mathrm{K},\omega_0,0)$, and the Schur complement $S:=A-BD^{-1}C$ satisfies
\begin{equation*}
    \begin{pmatrix}A&B\\ C&D\end{pmatrix}
    =
    \begin{pmatrix}I&BD^{-1}\\0&I\end{pmatrix}
    \begin{pmatrix}S&0\\0&D\end{pmatrix}
    \begin{pmatrix}I&0\\D^{-1}C&I\end{pmatrix}.
\end{equation*}
Thus the problem \eqref{eq:Tdelta-alpha} and $\det S=0$ have the same characteristic values and algebraic multiplicities. For real parameters, we have $C=B^*$, $D=D^*$, and $S=S^*$. The normalized matrix
\begin{equation*}
    \mathsf F_{\rm hc}^\alpha(\omega,\delta)
    :=\mathcal M(\omega_0)^{-1}S,
    \qquad
    \mathcal M(\omega_0)=-\frac{2\omega_0}{v_b^2}I_2,
\end{equation*}
is therefore also Hermitian. The Schur formula $S = A-BD^{-1}C$ and \eqref{eq:interior-derivative-normalization} give
\begin{equation}\label{eq:exact-honeycomb-schur-expansion}
    \mathsf F_{\rm hc}^\alpha(\omega,\delta)
    = (\omega-\omega_0)I_2
      -\delta\mc C_{\rm hc}^\alpha(\omega_0)
      +\mathsf E^\alpha(\omega,\delta),
\end{equation}
where, for every multi-index $\eta$ with $|\eta|\leq2$,
\begin{equation*}
    \bigl\|\partial_\alpha^\eta
    \mathsf E^\alpha(\omega,\delta)\bigr\|
    \lesssim |\omega-\omega_0|^2
    +\delta|\omega-\omega_0|+\delta^2.
\end{equation*}
This follows from joint analyticity and the fact that the $\alpha$-dependence disappears when $\delta=0$.

The covariance \eqref{eq:covarianeq} and $PT$ symmetries used in the proof of \Cref{prop:honeycomb-normal-form} are inherited by this reduction. At $\mathrm{K}$, the rotation symmetry makes $\mathsf F_{\rm hc}^\mathrm{K}$ diagonal in the $\operatorname{diag}(\tau,\overline\tau)$ basis, while the anti-unitary $PT$ symmetry exchanges the two diagonal sectors. Hence
\begin{equation*}
    \mathsf F_{\rm hc}^\mathrm{K}(\omega,\delta)
    =f_\mathrm{K}(\omega,\delta)I_2
\end{equation*}
for real $\omega$ and $\delta$; holomorphy in $\omega$ extends this identity to a complex neighborhood of $\omega_0$. By \eqref{eq:exact-honeycomb-schur-expansion},
\begin{equation*}
    f_\mathrm{K}(\omega,\delta)
    =\omega-\omega_0-\delta c_\mathrm{K}(\omega_0)
    +\O\bigl(|\omega-\omega_0|^2
    +\delta|\omega-\omega_0|+\delta^2\bigr).
\end{equation*}
Since $\partial_\omega f_\mathrm{K}(\omega_0,0)=1$, the implicit-function theorem gives the unique root \eqref{eq:exact-dirac-crossing}. Because $\det\mathsf F_{\rm hc}^\mathrm{K}=f_\mathrm{K}^2$, this root has multiplicity two.

Similarly, differentiating the covariance relation at $\mathrm{K}$ gives
\begin{equation*}
    D_\alpha\mathsf F_{\rm hc}^\alpha
    (\omega_{\rm D}^\mathrm{K}(\delta),\delta)\big|_{\alpha=\mathrm{K}}[\xi]
    =
    \begin{pmatrix}
        0 & \beta_\mathrm{K}(\delta)(\xi_1-i\xi_2)\\
        \overline{\beta_\mathrm{K}(\delta)}(\xi_1+i\xi_2) & 0
    \end{pmatrix},
\end{equation*}
where \eqref{eq:exact-honeycomb-schur-expansion} and \eqref{eq:honeycomb-linear-normal-form} give
\begin{equation*}
    \beta_\mathrm{K}(\delta)=-\delta\gamma_\mathrm{K}(\omega_0)+\O(\delta^2).
\end{equation*}
Also, $\partial_\omega\mathsf F_{\rm hc}^\mathrm{K} (\omega_{\rm D}^\mathrm{K}(\delta),\delta)=q_\mathrm{K}(\delta)I_2$ with $q_\mathrm{K}(\delta)=1+\O(\delta)>0$. Writing $s=\omega-\omega_{\rm D}^\mathrm{K}(\delta)$, Taylor's theorem and Hermiticity show that the two eigenvalues of the Schur complement are
\begin{equation*}
    q_\mathrm{K}(\delta)s\pm|\beta_\mathrm{K}(\delta)||\xi|
    +\O_\delta\bigl(s^2+|s||\xi|+|\xi|^2\bigr).
\end{equation*}
After shrinking the neighborhood, $\partial_\omega\mathsf F_{\rm hc}^\alpha$
is positive definite. Hence the two ordered eigenvalues are strictly
increasing in $\omega$, and solving their unique zeros for $s$ yields
\begin{equation*}
    \omega_\pm^{\mathrm{K}+\xi}(\delta)
    =\omega_{\rm D}^\mathrm{K}(\delta)
    \pm\frac{|\beta_\mathrm{K}(\delta)|}{q_\mathrm{K}(\delta)}|\xi|
    +\O_\delta(|\xi|^2),
\end{equation*}
which proves \eqref{eq:exact-dirac-cone}. The proof at $\mathrm{K}'$ is identical, since time-reversal symmetry gives, for a fixed unitary matrix $U$,
\begin{equation*}
    \mathsf F_{\rm hc}^{\mathrm{K}'-\xi}(\omega,\delta)
    =U\,\overline{\mathsf F_{\rm hc}^{\mathrm{K}+\xi}
    (\overline\omega,\delta)}\,U^{-1},
    \qquad |\xi|\ll1.
\end{equation*}
In particular, the Dirac resonant frequencies agree, and differentiation at $\xi=0$
gives equal slope magnitudes at $\mathrm{K}$ and $\mathrm{K}'$. The proof is complete.
\end{proof}

For the triangular decomposition $Y = Y_1 \cup Y_2$ fixed above, set
\begin{equation*}
    \Omega_j:=Y_j\setminus\overline{D_j},
    \qquad j=1,2.
\end{equation*}
Let $\Gamma_j:=\partial Y_j\cap\partial Y$, $j=1,2$. We write $\nu_Y$ for the outward normal on $\partial Y$ and the normal vector $\nu_3$ from $Y_1$ to $Y_2$. For any subdomain $G\subset\Omega$, set
\begin{equation*}
    Q_{\omega_0,G}(u,v)
    :=\int_G\bigl(\nabla u\cdot\overline{\nabla v}
    -k_0^2u\overline v\bigr)\,dx,
    \qquad k_0:=\frac{\omega_0}{v}.
\end{equation*}
Let $u_j$ be the normalized interior Neumann eigenfunction on $D_j$, and let $g_j$ be its boundary trace, extended by zero to the other component of $\partial D$. By \Cref{assumphoneycomb}(i) and the simplicity of the eigenvalue,
\begin{equation*}
    g_j\circ R_j=g_j
    \quad\text{on }\partial D_j,
    \qquad j=1,2.
\end{equation*}
By the Bloch-rotation convention in \Cref{assumphoneycomb},
\begin{equation*}
    U_R^{\mathrm K}g_1=\tau g_1,
    \qquad
    U_R^{\mathrm K}g_2=\overline\tau g_2.
\end{equation*}
By the reflection symmetry in \Cref{assumphoneycomb} and the simplicity of
the Neumann eigenvalue, the phases of the normalized eigenfunctions may be chosen so that $u_2=u_1\circ R_3$, and hence $g_2=g_1\circ R_3$; we use this convention below.

\begin{proposition}
\label{prop:frequency-dependent-lemma35}
Under \Cref{assumphoneycomb} and the above convention, the entries of $\mc C^\alpha_{\rm hc}(\omega_0)$ are differentiable in $\alpha$
at $\mathrm{K}$, with
\begin{equation}\label{eq:fd-lemma35-gradient}
    \nabla_\alpha
    \bigl(\mc C^\alpha_{\rm hc}(\omega_0)\bigr)_{11}\bigr|_{\alpha=\mathrm{K}}=0,
    \qquad
    \nabla_\alpha
    \bigl(\mc C^\alpha_{\rm hc}(\omega_0)\bigr)_{12}\bigr|_{\alpha=\mathrm{K}}
    =
    \gamma_\mathrm{K}(\omega_0)
    \binom{1}{-i}.
\end{equation}
Moreover, if $V_1^\mathrm{K}$ denotes the $\mathrm{K}$-quasiperiodic exterior solution with Dirichlet data $g_1$ on $\partial D_1$ and zero on $\partial D_2$, then
\begin{equation}\label{eq:gamma-implies-half-cell-zero}
    \gamma_\mathrm{K}(\omega_0)=0
    \quad\Longrightarrow\quad
    Q_{\omega_0,\Omega_2}(V_1^\mathrm{K},V_1^\mathrm{K})=0.
\end{equation}
The analogous conclusion holds at $\mathrm{K}'$.
\end{proposition}

\begin{proof}
The exterior regularity in
\eqref{eq:honeycomb-exterior-regular-assumption} therefore gives the differentiability of $\mc C^\alpha_{\rm hc}(\omega_0)$ near $\mathrm{K}$. The
covariance relation \eqref{eq:covarianeq} gives the same symmetry relations as in \cite[Lemma~3.5]{honeycomb1}, and differentiating these relations at
$\alpha=\mathrm{K}$ yields \eqref{eq:fd-lemma35-gradient}.
Here the simplicity of the interior eigenvalue and the threefold symmetry of each resonator imply that $g_j$ is invariant under the corresponding threefold rotation: the rotation acts on the one-dimensional real eigenspace by a real cube root of unity, which must be $1$.

It remains to prove \eqref{eq:gamma-implies-half-cell-zero}. For $\alpha$ near $\mathrm{K}$, let $V_j^\alpha$ be the exterior solution with Dirichlet data $g_j$ on $\partial D_j$ and zero on the other resonator, and write
\begin{equation*}
    \dot V_j
    :=\left.\partial_{\alpha_1}V_j^\alpha\right|_{\alpha=\mathrm{K}}.
\end{equation*}
In particular, $V_1^\mathrm{K}$ solves
\begin{equation*}
    (\Delta+k_0^2)\,V_1^\mathrm{K}=0 \quad \text{in }\Omega,
    \qquad
    V_1^\mathrm{K}=g_1 \quad \text{on }\partial D_1,
    \qquad
    V_1^\mathrm{K}=0 \quad \text{on }\partial D_2.
\end{equation*}
The differentiated fields satisfy
\begin{equation*}
    (\Delta+k_0^2)\dot V_j=0 \quad\text{in }\Omega,
    \qquad
    \dot V_j=0 \quad\text{on }\partial D,
\end{equation*}
together with the affine differentiated quasiperiodicity condition. The explicit condition is
\begin{equation*}
    \dot V_j(x+\ell)
    =e^{\mathrm iK\cdot\ell}
    \bigl(\dot V_j(x)+\mathrm i\ell_{(1)}V_j^\mathrm{K}(x)\bigr),
    \qquad \ell\in\Lambda,
\end{equation*}
where $\ell_{(1)}$ denotes the component in the direction with respect to which $\alpha_1$ is differentiated. The mass term creates no additional contribution in the Green identities. More precisely, the definition \eqref{eq:case1-cap-matrix} and Green's identity give
\begin{equation*}
    \gamma_\mathrm{K}(\omega_0)
    =\frac{v_b^2}{2\omega_0}
    \left(
    Q_{\omega_0,\Omega}(\dot V_2,V_1^\mathrm{K})
    +Q_{\omega_0,\Omega}(V_2^\mathrm{K},\dot V_1)
    \right).
\end{equation*}
For solutions $V_\ell^\mathrm{K}$ and differentiated solutions $\dot V_j$, we also have 
\begin{align*}
    Q_{\omega_0,\Omega}(\dot V_j,V_\ell^\mathrm{K})
     = \int_{\partial Y}\dot V_j\,
      \overline{\partial_{\nu_Y} V_\ell^\mathrm{K}}\,d\sigma, \qquad
    Q_{\omega_0,\Omega}(V_j^\mathrm{K},\dot V_\ell)
     = \int_{\partial Y}\partial_{\nu_Y} V_j^\mathrm{K}\,
      \overline{\dot V_\ell}\,d\sigma.
\end{align*}

To make the quasiperiodic boundary calculation explicit, set
\begin{equation*}
    I_r:=\int_{\Gamma_r}\overline{V_1^\mathrm{K}}\,
    \partial_{\nu_Y}V_2^\mathrm{K}\,d\sigma,
    \qquad
    J_r:=\int_{\Gamma_r}\overline{\partial_{\nu_Y}V_1^\mathrm{K}}\,
    V_2^\mathrm{K}\,d\sigma,
    \qquad r=1,2.
\end{equation*}
The two sides of $\Gamma_2$ are paired with the corresponding sides of $\Gamma_1$ by the translations $\ell_1$ and $\ell_2$. Let 
\begin{equation*}
   L := (\ell_1)_1=(\ell_2)_1=\frac{\sqrt3a}{2}.
\end{equation*}
 Substitution of the quasiperiodicity condition into the two boundary integrals above gives
\begin{equation*}
    \gamma_\mathrm{K}(\omega_0)
    =\mathrm i\,\frac{v_b^2L}{2\omega_0}\,(J_2-I_2).
\end{equation*}
We also note 
\begin{equation*}
    I_2=-I_1,  \qquad  J_2=-J_1,  \qquad  J_1=-I_1.
\end{equation*}
Consequently,
\begin{equation*}
    \gamma_\mathrm{K}(\omega_0)
    =2\mathrm i\,\frac{v_b^2L}{2\omega_0}\,I_1.
\end{equation*}
Thus $\gamma_\mathrm{K}(\omega_0)=0$ implies
\begin{equation*}
    I_1=I_2=J_1=J_2=0.
\end{equation*}
Moreover, \eqref{eq:honeycomb-degenerate-cap} and
\eqref{eq:case1-cap-matrix} give
\begin{equation*}
    Q_{\omega_0,\Omega}(V_2^\mathrm{K},V_1^\mathrm{K})=0.
\end{equation*}
Splitting this form over $\Omega_1$ and $\Omega_2$ and integrating by parts, we have 
\begin{equation*}
    \int_{\Gamma_3}
    \left(
    \overline{\partial_{\nu_3}V_1^\mathrm{K}}\,V_2^\mathrm{K}
    -\overline{V_1^\mathrm{K}}\,\partial_{\nu_3}V_2^\mathrm{K}
    \right)d\sigma
    =Q_{\omega_0,\Omega}(V_2^\mathrm{K},V_1^\mathrm{K})=0.
\end{equation*}
The reflection relations
\begin{equation*}
    V_2^\mathrm{K}=V_1^\mathrm{K},
    \qquad
    \partial_{\nu_3}V_2^\mathrm{K}=-\partial_{\nu_3}V_1^\mathrm{K}
    \quad\text{on }\Gamma_3
\end{equation*}
then give
\begin{equation} \label{eq:gammak}
    \operatorname{Re}\int_{\Gamma_3}
    \overline{V_1^\mathrm{K}}\,\partial_{\nu_3}V_1^\mathrm{K}\,d\sigma=0.
\end{equation}
Using the local rotation $R_1$ from \Cref{assumphoneycomb}, the induced
map changes $V_1^\mathrm{K}$ only by a unit phase, and
$\Gamma_1=(R_1\Gamma_3)\cup(R_1^2\Gamma_3)$. Thus the above
\eqref{eq:gammak}, by $R_1$ and $R_1^2$ transformations, yields
\begin{equation*}
    \operatorname{Re}\int_{\Gamma_1}
    \overline{V_1^\mathrm{K}}\,\partial_{\nu_Y}V_1^\mathrm{K}\,d\sigma=0.
\end{equation*}
Similarly, the same identity holds on $\Gamma_2$. Finally, $V_1^\mathrm{K}=0$ on $\partial D_2$. Green's identity in $\Omega_2$ now gives
\begin{equation*}
    Q_{\omega_0,\Omega_2}(V_1^\mathrm{K},V_1^\mathrm{K})
    =\operatorname{Re}\int_{\partial\Omega_2}
    \overline{V_1^\mathrm{K}}\,\partial_{\nu}V_1^\mathrm{K}\,d\sigma=0,
\end{equation*}
which proves \eqref{eq:gamma-implies-half-cell-zero}. The argument at $\mathrm{K}'$ is the same after time reversal.
\end{proof}

To establish the \emph{generic} nonvanishing of the Dirac velocity, we now study the coefficient $\gamma_\mathrm{K}$ directly. Write the reflection $R_3$ in
coordinates centered at $z_1$ and $z_2$ as
\begin{equation*}
    R_3(z_1+y)=z_2+A_3y,
    \qquad
    A_3\in O(2),\quad A_3^2=I,\quad \det A_3=-1.
\end{equation*}
Let $B\subset\mathbb R^2$ be a fixed bounded $C^\infty$ reference domain, centered at the origin and invariant under rotation by $2\pi/3$. We also
assume that $A_3B=-B$, which ensures compatibility with inversion through the center of the cell. Set
\begin{equation*}
    Q_1:=I, \qquad  Q_2:=A_3,
\end{equation*}
and consider the real-analytic honeycomb scaling family
\begin{equation}\label{eq:domainansatz}
    D^\rho:=D_1^\rho\cup D_2^\rho,
    \qquad
    D_j^\rho:=z_j+\rho Q_jB,
    \qquad j=1,2,
    \qquad 0<\rho<\rho_{\max},
\end{equation}
where $\rho_{\max}>0$ is chosen so that the inclusions remain disjoint,
$D_j^\rho\subset Y_j$, the domains $Y\setminus\overline{D^\rho}$ and
$Y_2\setminus\overline{D_2^\rho}$ remain connected. The choices above ensure
that the geometric symmetries in \Cref{assumphoneycomb}(i)--(ii) are
preserved. Fix a simple positive Neumann eigenvalue $\mu_\ast$ of $-\Delta$
on $B$, with a real
$L^2(B)$-normalized eigenfunction $U$, and assume that
\begin{equation}\label{eq:scaled-interior-regularity}
    \mu_\ast\notin\sigma_{\rm Dir}(-\Delta;B).
\end{equation}
Set
\begin{equation*}
    \Omega_\rho:=Y\setminus\overline{D^\rho},
    \qquad
    \Omega_{j,\rho}:=Y_j\setminus\overline{D_j^\rho},
    \qquad j=1,2,
    \qquad
    \omega_0(\rho):=\frac{v_b\sqrt{\mu_\ast}}{\rho},
    \qquad
    k_0(\rho):=\frac{\omega_0(\rho)}{v}.
\end{equation*}
The corresponding normalized modes and traces are
\begin{equation*}
    U_j^\rho(x)
    :=
    \rho^{-1}U\!\left(\frac{Q_j^{-1}(x-z_j)}{\rho}\right),
    \qquad
    g_j^\rho:=U_j^\rho|_{\partial D_j^\rho},
    \qquad j=1,2.
\end{equation*}
We extend each $g_j^\rho$ by zero to the other component of $\partial D^\rho$. By the definitions of $A_3$ and $Q_j$, $U_2^\rho=U_1^\rho\circ R_3$ and hence $g_2^\rho=g_1^\rho\circ R_3$. Whenever the $\alpha$-quasiperiodic exterior Dirichlet problem is regular for $\alpha\in\{\mathrm{K},\mathrm{K}'\}$, we write $c_\alpha(\rho):=c_\alpha(\omega_0(\rho))$ and $\gamma_\alpha(\rho):=\gamma_\alpha(\omega_0(\rho))$, and set $v_\alpha(\rho):=|\gamma_\alpha(\rho)|$.

\begin{proposition}
\label{prop:gammaK-generic-rho}
For $0<\rho<\rho_{\max}$, consider the lowest Laplacian
eigenvalue in $\Omega_\rho$ with the homogeneous Neumann condition on $\partial D^\rho$:
\begin{equation*}
    \lambda_{\mathrm{K},\mathrm N}(\rho)
    :=
    \inf_{0\ne w\in H_\mathrm{K}^1(\Omega_\rho)}
    \frac{\displaystyle\int_{\Omega_\rho}|\nabla w|^2\,dx}
         {\displaystyle\int_{\Omega_\rho}|w|^2\,dx},
\end{equation*}
where $H_\mathrm{K}^1(\Omega_\rho)$ denotes the space of quasiperiodic $H^1$-functions with Bloch parameter $\mathrm{K}$. Define also 
\begin{equation*}
    \lambda_{\mathrm{mix}}(\rho)
    :=
    \inf_{\substack{0\ne w\in H^1(\Omega_{2,\rho})\\
                     w|_{\partial D_2^\rho}=0}}
    \frac{\displaystyle\int_{\Omega_{2,\rho}}|\nabla w|^2\,dx}
         {\displaystyle\int_{\Omega_{2,\rho}}|w|^2\,dx}.
\end{equation*}
Then $\lambda_{\mathrm{mix}}(\rho)>0$. 

Suppose that, for some $\rho_\star\in(0,\rho_{\max})$,
\begin{equation}\label{eq:coercive-gamma-condition}
    k_0(\rho_\star)^2
    <\min\bigl\{\lambda_{\mathrm{K},\mathrm N}(\rho_\star),
    \lambda_{\mathrm{mix}}(\rho_\star)\bigr\}.
\end{equation}
Then the singular set\footnote{$\rho\in(0,\rho_{\max})\setminus\mathcal P_\mathrm{K}$ is called regular.}
\begin{equation*}
    \mathcal P_\mathrm{K}
    :=
    \left\{\rho\in(0,\rho_{\max}):
    k_0(\rho)^2\in
    \sigma_{\mathrm{K},\mathrm{Dir}}(-\Delta;\Omega_\rho)\right\}
\end{equation*}
is locally finite, and both $c_\mathrm{K}(\rho)$ and $\gamma_\mathrm{K}(\rho)$, initially defined outside $\mathcal P_\mathrm{K}$, have meromorphic continuations in $\rho$. Moreover,
\begin{equation*}
    c_\mathrm{K}(\rho_\star)>0,
    \qquad
    \gamma_\mathrm{K}(\rho_\star)\ne0,
\end{equation*}
and the regular zero set of the Dirac-slope coefficient,
\begin{equation*}
    \mathcal Z_{\gamma,\mathrm{K}}
    :=
    \left\{\rho\in(0,\rho_{\max})\setminus\mathcal P_\mathrm{K}:
    \gamma_\mathrm{K}(\rho)=0\right\},
\end{equation*}
is locally finite. Defining $\mathcal P_{\mathrm{K}'}$ analogously, time-reversal symmetry gives
\begin{equation*}
    \mathcal P_{\mathrm{K}'}=\mathcal P_\mathrm{K},
    \qquad
    c_{\mathrm{K}'}(\rho)=c_\mathrm{K}(\rho),
    \qquad
    |\gamma_{\mathrm{K}'}(\rho)|=|\gamma_\mathrm{K}(\rho)|
\end{equation*}
for real regular $\rho$. Consequently, $\mathcal E:=\mathcal P_\mathrm{K}\cup\mathcal Z_{\gamma,\mathrm{K}}$ is locally finite in $(0,\rho_{\max})$, and there exists a sequence $\rho_n\downarrow0$ with $\rho_n\notin\mathcal E$.
\end{proposition}

\begin{remark}
Similarly to \cref{rem:sufficient-case1}, for any fixed $\rho_\star$, condition \eqref{eq:coercive-gamma-condition} is satisfied whenever the wave-speed ratio $v_b/v$ is sufficiently small, since $k_0(\rho_\star)^2 =\left(\frac{v_b}{v}\right)^2\frac{\mu_\ast}{\rho_\star^2}$. 
\end{remark}

\begin{proof}
Since $\Omega_\rho$ is connected and $\mathrm{K}$ is not equivalent to $0$ modulo the reciprocal lattice, the only constant $\mathrm{K}$-quasiperiodic function is zero. The quasiperiodic Poincar\'e
inequality therefore gives $\lambda_{\mathrm{K},\mathrm N}(\rho)>0$.
The Poincar\'e inequality with zero trace on the nonempty boundary portion
$\partial D_2^\rho$ likewise gives $\lambda_{\mathrm{mix}}(\rho)>0$.
Condition
\eqref{eq:coercive-gamma-condition} makes the Helmholtz form coercive on
$H_\mathrm{K}^1(\Omega_{\rho_\star})$, so the inhomogeneous exterior Dirichlet problem
is well posed by the Lax--Milgram theorem; in particular,
$\rho_\star\notin\mathcal P_\mathrm{K}$. Let
$g_1^{\rho_\star}$ be the trace of the normalized Neumann eigenfunction on
$\partial D_1^{\rho_\star}$, extended by zero to
$\partial D_2^{\rho_\star}$, and let $V_{\rho_\star}^\mathrm{K}$ be the exterior
solution with boundary trace $g_1^{\rho_\star}$. By
\eqref{eq:coercive-gamma-condition},
\begin{align*}
    Q_{\rho_\star}(V_{\rho_\star}^\mathrm{K},V_{\rho_\star}^\mathrm{K})
    &:=
    \int_{\Omega_{\rho_\star}}
    \left(|\nabla V_{\rho_\star}^\mathrm{K}|^2
    -k_0(\rho_\star)^2|V_{\rho_\star}^\mathrm{K}|^2\right)dx\\
    &\ge
    \left(\lambda_{\mathrm{K},\mathrm N}(\rho_\star)
    -k_0(\rho_\star)^2\right)
    \|V_{\rho_\star}^\mathrm{K}\|_{L^2(\Omega_{\rho_\star})}^2>0.
\end{align*}
Here $V_{\rho_\star}^\mathrm{K}\ne0$ because its boundary trace is nonzero. Green's
identity, with the normal convention used in the exterior DtN map, yields
\begin{equation*}
    Q_{\rho_\star}(V_{\rho_\star}^\mathrm{K},V_{\rho_\star}^\mathrm{K})
    =-
    \left\langle
    \Lambda_{\rm ext}^{\mathrm{K},\omega_0(\rho_\star)}
    [g_1^{\rho_\star}],g_1^{\rho_\star}
    \right\rangle_{\partial D^{\rho_\star}}.
\end{equation*}
Using \eqref{eq:case1-cap-matrix} and
$\mc C_{\rm hc}^\mathrm{K}(\omega_0(\rho_\star))=c_\mathrm{K}(\rho_\star)I_2$, we obtain
\begin{equation}\label{eq:coercive-ck-positive}
    c_\mathrm{K}(\rho_\star)
    =
    \frac{v_b^2}{2\omega_0(\rho_\star)}
    Q_{\rho_\star}(V_{\rho_\star}^\mathrm{K},V_{\rho_\star}^\mathrm{K})>0.
\end{equation}

We next prove the nonvanishing of the Dirac-slope coefficient. If $\gamma_\mathrm{K}(\rho_\star)=0$, then
\cref{prop:frequency-dependent-lemma35}, the zero trace of
$V_{\rho_\star}^\mathrm{K}$ on $\partial D_2^{\rho_\star}$, and
\eqref{eq:coercive-gamma-condition} give
\begin{equation*}
    0=
    \int_{\Omega_{2,\rho_\star}}
    \left(|\nabla V_{\rho_\star}^\mathrm{K}|^2
    -k_0(\rho_\star)^2|V_{\rho_\star}^\mathrm{K}|^2\right)dx
    \ge
    \bigl(\lambda_{\mathrm{mix}}(\rho_\star)
    -k_0(\rho_\star)^2\bigr)
    \|V_{\rho_\star}^\mathrm{K}\|_{L^2(\Omega_{2,\rho_\star})}^2.
\end{equation*}
Thus its restriction to $\Omega_{2,\rho_\star}$ vanishes. Unique
continuation in the connected
exterior domain then gives $V_{\rho_\star}^\mathrm{K}=0$ throughout
$\Omega_{\rho_\star}$, contradicting its nonzero trace on
$\partial D_1^{\rho_\star}$. Thus
\begin{equation}\label{eq:coercive-gamma-nonzero}
    \gamma_\mathrm{K}(\rho_\star)\ne0.
\end{equation}

It remains to propagate this nonvanishing information from $\rho_\star$ along the scaling family. Using real-analytic pullbacks to fixed reference cells,  the exterior Dirichlet problem becomes a jointly holomorphic Fredholm family of index zero in $(\alpha,\rho)$, denoted by $\mathscr{A}(\alpha,\rho)$. By its invertibility at $\rho_\star$ and the analytic Fredholm theorem, we have the meromorphic inverse $\mathscr{A}(\mathrm{K},\rho)^{-1}$ in $\rho$, with the locally finite pole set $\mathcal P_\mathrm{K}$. Therefore, the associated exterior DtN map and its $\alpha_1$-derivative at $\mathrm K$ are meromorphic in $\rho$, and the boundary pairings defining $c_\mathrm{K}(\rho)$ and $\gamma_\mathrm{K}(\rho)$ admit compatible meromorphic continuations across these poles. By \eqref{eq:coercive-ck-positive} and \eqref{eq:coercive-gamma-nonzero}, neither continuation is identically zero. Hence $\mathcal Z_{\gamma,\mathrm{K}}$, and therefore $\mathcal P_\mathrm{K}\cup\mathcal Z_{\gamma,\mathrm{K}}$, is locally finite. Since $\mathrm{K}'\equiv-\mathrm{K}$ modulo the reciprocal lattice and the coefficients and the interior trace are real, complex conjugation intertwines the $\mathrm{K}$ and $\mathrm{K}'$ fibers and their DtN maps. This gives $\mathcal P_{\mathrm{K}'}=\mathcal P_\mathrm{K}$ and, on the real regular set,
\begin{equation*}
    c_{\mathrm{K}'}(\rho)=\overline{c_\mathrm{K}(\rho)}=c_\mathrm{K}(\rho),
    \qquad
    |\gamma_{\mathrm{K}'}(\rho)|=|\gamma_\mathrm{K}(\rho)|,
\end{equation*}
where the equality for $c_\mathrm{K}$ follows from Hermiticity. Finally, since a locally finite subset of $(0,\rho_{\max})$ contains no interval, starting with $\rho_0:=\rho_{\max}$, one may therefore choose recursively
\begin{equation*}
    \rho_n\in
    \left(0,\min\left\{\frac1n,\frac{\rho_{n-1}}2,\rho_{\max}\right\}\right)
    \setminus\mathcal E,
\end{equation*}
which ensures $\rho_n\downarrow0$. The proof is complete. 
\end{proof}

Combining
\cref{prop:honeycomb-normal-form,lem:exact-honeycomb-crossing,prop:frequency-dependent-lemma35,prop:gammaK-generic-rho}
now yields the main result of this subsection.

\begin{theorem}[(Generic existence of Dirac cones)]
\label{thm:honeycomb-dirac-cone}
Under \Cref{assumphoneycomb} and the assumptions of
\Cref{prop:gammaK-generic-rho}, let
$\mathcal E=\mathcal P_\mathrm{K}\cup\mathcal Z_{\gamma,\mathrm{K}}$. For every
$\rho\in(0,\rho_{\max})\setminus\mathcal E$, the non-degeneracy condition
\begin{equation*}
    v_\mathrm{K}(\rho)=|\gamma_\mathrm{K}(\rho)|>0
\end{equation*}
holds, the two eigenvalues of $\mathcal{C}^{\alpha}_{\mathrm{hc}}(\omega_0(\rho))$
near $\alpha = \mathrm{K}$ satisfy
\begin{equation}\label{eq:honeycomb-cap-eigs-cone}
    \lambda_\pm^{\mathrm{K}+\xi}(\omega_0(\rho))
    = c_\mathrm{K}(\rho) \pm v_\mathrm{K}(\rho)\,|\xi| + \O(|\xi|^2),
    \qquad \xi \to 0,
\end{equation}
and the two resonant Bloch bands converging to $\omega_0(\rho)$ satisfy
\begin{equation}\label{eq:honeycomb-bands-cone}
    \omega_\pm^{\mathrm{K}+\xi}(\delta)
    = \omega_0(\rho)
      + \delta\,c_\mathrm{K}(\rho)
      \pm \delta\,v_\mathrm{K}(\rho)\,|\xi|
      + \O\!\left(\delta\,|\xi|^2 + \delta^2\right),
\end{equation}
for each fixed admissible $\rho$, uniformly for $\xi$ in a sufficiently small
neighborhood of $0$. For every sufficiently small fixed real $\delta>0$, the
same two branches form an exact Dirac cone
\begin{equation}\label{eq:honeycomb-exact-bands-cone}
    \omega_\pm^{\mathrm{K}+\xi}(\delta)
    =\omega_{\rm D}^\mathrm{K}(\delta;\rho)
    \pm v_{\rm D}^\mathrm{K}(\delta;\rho)|\xi|
    +\O_{\delta,\rho}(|\xi|^2),
\end{equation}
where
\begin{equation*}
    \omega_{\rm D}^\mathrm{K}(\delta;\rho)
    =\omega_0(\rho)+\delta c_\mathrm{K}(\rho)+\O(\delta^2),
    \qquad
    v_{\rm D}^\mathrm{K}(\delta;\rho)
    =\delta v_\mathrm{K}(\rho)+\O(\delta^2)>0.
\end{equation*}
The same conclusions hold at $\mathrm{K}'$. Moreover, there exists a sequence
$\rho_n\downarrow0$, with $\rho_n\notin\mathcal E$, such that
$\omega_0(\rho_n)\to\infty$. The smallness threshold for $\delta$ may depend
on $n$ along this sequence.
\end{theorem}

\begin{proof}
By \cref{prop:gammaK-generic-rho}, fix $\rho\in(0,\rho_{\max})\setminus\mathcal E$. Then $\gamma_\mathrm{K}(\rho)\neq0$ and the exterior problems at $\mathrm{K}$ and $\mathrm{K}'$ are
regular. Hence $v_\mathrm{K}(\rho)=|\gamma_\mathrm{K}(\rho)|>0$. 
For the eigenvalue expansion \eqref{eq:honeycomb-cap-eigs-cone}, it suffices to note that the off-diagonal Hermitian matrix \eqref{eq:honeycomb-linear-normal-form} has eigenvalues $\pm v_\mathrm{K}(\rho)\,|\xi|$. Since exterior regularity persists in a neighborhood of $\mathrm{K}$, the local Case~1 reduction is uniform there; substituting the eigenvalue expansion into \eqref{approx:alphaj} gives \eqref{eq:honeycomb-bands-cone}. The exact expansion \eqref{eq:honeycomb-exact-bands-cone} follows from
\cref{lem:exact-honeycomb-crossing}. The argument at $\mathrm{K}'$ follows by time-reversal symmetry. 
\end{proof}

\section{Concluding remarks}
In this paper, we have extended the formalism of the capacitance matrix approximation beyond the subwavelength regime. We have considered finite and periodic systems and established approximations at leading order in the contrast of the resonant frequencies and Bloch band functions of, respectively, finite and periodic systems of high-contrast resonators in terms of the frequency-dependent capacitance matrices. We have established key properties of these matrices. In the periodic case, under the stated non-resonance and symmetry assumptions, we have proved bandgap opening for three-dimensional lattices with a single resonator per cell and the generic existence of Dirac cones for two-dimensional honeycomb lattices; the latter occur along a sequence of frequencies tending to infinity.

Our results in this paper form the basis for the analysis and computation of scattering resonances in high-contrast systems of resonators. They can be used to illustrate the effective behavior of waves in honeycomb structures as in \cite{honeycomb2,honeycomb3} beyond the subwavelength regime. They can also be extended to systems with nonlinear or time-modulated material parameters. This will be the subject of subsequent parts of our series of papers on high-contrast systems of resonators beyond the subwavelength regime. It would also be interesting to generalize our approach and results to elastic resonators with high-contrast material parameters \cite{elast1,elast2,elast3}. Finally, numerical illustrations of our findings in the periodic case are currently under consideration and will be the subject of a forthcoming work.

\section*{Acknowledgments}
This work was partially supported by the National Key R\&D Program of China grant number 2024YFA1016000, the Fundamental Research Funds for the Central Universities grant number 226-2025-00192, and the City University of Hong Kong start-up fund 7200843. It was initiated while H.A. was visiting the Hong Kong Institute for Advanced Study as a Senior Fellow.

\appendix

\section{ODE formulation of the frequency-dependent capacitance matrix in the one-dimensional case} \label{sec:reduct}

In this appendix, we consider the one-dimensional setting as in \cite{fabry1}. Although the analysis of this work is carried out in two and three dimensions, the differential equation formulation based on DtN maps in \Cref{sec:3} can be adapted to the one-dimensional case, complementing the results of our previous work \cite{fabry1}. 

Let $N\ge 2$ and
\begin{equation*}
    D=\bigcup_{j=1}^N (x_j^-,x_j^+)\subset \R,
\qquad
\ell_j:=x_j^+-x_j^-,
\qquad
\ell_{j,j+1}:=x_{j+1}^- - x_j^+ \quad (1\le j\le N-1),
\end{equation*}
where
\begin{equation*}
    x_1^-<x_1^+<x_2^-<x_2^+<\cdots <x_N^-<x_N^+.
\end{equation*}
Fix a positive limiting reference frequency $\omega_0>0$, and recall the notation
\begin{equation*}
    k_0:=\frac{\omega_0}{v},
\qquad
k_{j,0}:=\frac{\omega_0}{v_j},
\qquad
\mc{J}:=\Bigl\{j\in\{1,\dots,N\}: k_{j,0}^2\in \sigma_{\rm Neu}(-\partial_x^2;D_j)\Bigr\}.
\end{equation*}
Note that any Neumann eigenvalue of $-\partial_x^2$ on $D_j=(x_j^-,x_j^+)$ is simple, and for each $j\in \mc{J}$ there exists a unique integer $n_j\in \mathbb{N}$ such that
\begin{equation*}
    k_{j,0}=\frac{n_j\pi}{\ell_j},
\end{equation*}
with the corresponding $L^2(D_j)$-normalized Neumann eigenfunction given by
\begin{equation*}
    u_j(x):=\sqrt{\frac{2}{\ell_j}}
\cos \Bigl(\frac{n_j\pi(x-x_j^-)}{\ell_j}\Bigr),
\qquad x\in D_j.
\end{equation*}
The boundary trace $g_j$ of $u_j$ is supported on $\partial D_j=\{x_j^-,x_j^+\}$, and satisfies 
\begin{equation*}
    g_j(x_j^-)=\sqrt{\frac{2}{\ell_j}},
\qquad
g_j(x_j^+)=(-1)^{n_j}\sqrt{\frac{2}{\ell_j}},
\qquad
g_j(x_i^\pm)=0 \quad \text{for } i\neq j.
\end{equation*}
For convenience, let
\begin{equation*}
    a_j:=\sqrt{\frac{2}{\ell_j}},
\qquad
s_j:=(-1)^{n_j}.
\end{equation*}
Then $g_j(x_j^-)=a_j$ and $g_j(x_j^+)=s_j a_j$.

We endow the boundary space $\C^{2N}$ with the bilinear pairing
\begin{equation*}
    \langle f,h\rangle_{\partial D}
:= \sum_{i=1}^N \bigl(f_i^- h_i^- + f_i^+ h_i^+\bigr),
\qquad
f_i^\pm:=f(x_i^\pm),\quad h_i^\pm:=h(x_i^\pm).
\end{equation*}
Let $\mathscr T_{k_0}$ denote the exterior DtN map at the wave number $k_0$, as defined in \cite{Subwavelength_1D_High-Contrast}. The normal derivative is taken with respect to the outward normal of $D$, which equals $-1$ at $x_i^-$ and $+1$ at $x_i^+$. Explicitly,
\begin{equation*}
    \mathscr T_{k_0}[f]^-_1=\mathrm{i}k_0 f_1^-,
\qquad
\mathscr T_{k_0}[f]^+_N=\mathrm{i}k_0 f_N^+,
\end{equation*}
while for $1\le i\le N-1$,
\begin{equation} \label{eq:formuladtn}
    \binom{\mathscr T_{k_0}[f]^+_i}{\mathscr T_{k_0}[f]^-_{i+1}}
= A_{k_0}(\ell_{i,i+1})
\binom{f_i^+}{f_{i+1}^-},
\qquad
A_k(\ell):=
\begin{pmatrix}
-k\cot(k\ell) & k\csc(k\ell)\\
k\csc(k\ell) & -k\cot(k\ell)
\end{pmatrix}.
\end{equation}
Following \cref{def:freq-cap-matrix-general-updated}, we introduce the frequency-dependent capacitance matrix
$\mathcal C(\omega_0) = (\mathcal C_{ij}(\omega_0))_{i,j\in \mc J}$ associated with $\omega_0$ in the one-dimensional setting:
\begin{equation}\label{eq:1d-freq-cap-def}
\mathcal C_{ij}(\omega_0)
:=
-\frac{\delta_i v_i^2}{2\omega_0}
\bigl\langle \mathscr T_{k_0}[g_j],g_i\bigr\rangle_{\partial D},
\qquad i,j\in \mc J.
\end{equation}

The main purpose of this appendix is to explore the connections between \eqref{eq:1d-freq-cap-def} and the frequency-dependent capacitance matrix introduced in \cite{fabry1} using the propagation-matrix approach and the Newton polygon method. 

We first derive an explicit formula for $\mathcal C(\omega_0)$. We emphasize that for the exterior DtN map to be well defined, the reference wave number $k_0$ must avoid the Dirichlet wavenumbers of the exterior spacing intervals $(x_j^+, x_{j+1}^-)$. 

\begin{proposition}\label{prop:1d-freq-cap-formula}
Assume that $k_0\ell_{i,i+1}\notin \pi \Z$ for all $1\le i\le N-1$, so that $\mathscr T_{k_0}$ is well defined. Then $\mathcal C(\omega_0)$ has nearest-neighbor structure:
\begin{equation} \label{eq:tridiagonal}
    \mathcal C_{ij}(\omega_0)=0
\qquad \text{whenever } |i-j|>1,
\end{equation}
and its nonzero entries are given as follows. We have
\begin{itemize}
\item For $i\in \mc J$ with $1<i<N$,
\begin{equation} \label{eq:cii}
    \mathcal C_{ii}(\omega_0)
=
\frac{\delta_i v_i^2}{v\ell_i}
\Bigl(\cot(k_0\ell_{i-1,i})+\cot(k_0\ell_{i,i+1})\Bigr).
\end{equation}

\item At the endpoints, if $1\in \mc J$,
\begin{equation*}
    \mathcal C_{11}(\omega_0)
=
\frac{\delta_1 v_1^2}{v\ell_1}
\Bigl(\cot(k_0\ell_{1,2})-\mathrm{i}\Bigr),
\end{equation*}
and if $N\in \mc J$,
\begin{equation*}
    \mathcal C_{NN}(\omega_0) = \frac{\delta_N v_N^2}{v\ell_N}
\Bigl(\cot(k_0\ell_{N-1,N})-\mathrm{i}\Bigr).
\end{equation*}

\item For $1\le i\le N-1$ with $i,i+1\in \mc J$,
\begin{equation*}
    \mathcal C_{i,i+1}(\omega_0) = -\frac{\delta_i v_i^2}{v\sqrt{\ell_i\ell_{i+1}}}\,
s_i\,\csc(k_0\ell_{i,i+1}),
\qquad
\mathcal C_{i+1,i}(\omega_0) = -\frac{\delta_{i+1} v_{i+1}^2}{v\sqrt{\ell_i\ell_{i+1}}}\,
s_i\,\csc(k_0\ell_{i,i+1}).
\end{equation*}
\end{itemize}
\end{proposition}

\begin{proof}
Let $j\in \mc J$. Since $g_j$ is supported on $\{x_j^-,x_j^+\}$, the explicit formula \eqref{eq:formuladtn} for $\mathscr T_{k_0}$ shows that $\mathscr T_{k_0}[g_j]$ can be nonzero only at the four points $x_{j-1}^+,\, x_j^-,\, x_j^+,\, x_{j+1}^-$ (with obvious modifications when $j=1$ or $j=N$). Consequently, we have \eqref{eq:tridiagonal}. 

We now compute the nonzero entries. For $1<j<N$ with $j\in \mc J$,
\begin{equation*}
    \mathscr T_{k_0}[g_j](x_j^-) = -k_0\cot(k_0\ell_{j-1,j})\,a_j,
\qquad
\mathscr T_{k_0}[g_j](x_j^+) = -k_0\cot(k_0\ell_{j,j+1})\,s_j a_j,
\end{equation*}
and therefore, using $s_j^2=1$,
\begin{align*}
\bigl\langle \mathscr T_{k_0}[g_j],g_j\bigr\rangle_{\partial D}
&=
a_j\bigl(-k_0\cot(k_0\ell_{j-1,j})\,a_j\bigr)
+
s_j a_j \bigl(-k_0\cot(k_0\ell_{j,j+1})\,s_j a_j\bigr) \\
&=
-\frac{2k_0}{\ell_j}
\Bigl(\cot(k_0\ell_{j-1,j})+\cot(k_0\ell_{j,j+1})\Bigr).
\end{align*}
Substituting into \eqref{eq:1d-freq-cap-def} and using $k_0/\omega_0=1/v$ yields \eqref{eq:cii}. 

For $j=1$, the endpoint relation gives $\mathscr T_{k_0}[g_1](x_1^-) =\mathrm{i}k_0 a_1$, while
\begin{equation*}
    \mathscr T_{k_0}[g_1](x_1^+) = -k_0\cot(k_0\ell_{1,2})\,s_1 a_1.
\end{equation*}
Hence,
\begin{align*}
\bigl\langle \mathscr T_{k_0}[g_1],g_1\bigr\rangle_{\partial D}
&=
a_1(\mathrm{i}k_0 a_1)+s_1 a_1\bigl(-k_0\cot(k_0\ell_{1,2})\,s_1 a_1\bigr) \\
&=
\frac{2k_0}{\ell_1}\bigl(\mathrm{i}-\cot(k_0\ell_{1,2})\bigr),
\end{align*}
which gives the formula for $\mathcal C_{11}(\omega_0)$. The computation at $j=N$ is analogous.

For the off-diagonal entries, fix $1\le i\le N-1$ with $i,i+1\in \mc J$. Since $g_{i+1}$ is supported on $\{x_{i+1}^-,x_{i+1}^+\}$, the only nonzero contribution to $\langle \mathscr T_{k_0}[g_{i+1}],g_i\rangle_{\partial D}$ comes from the point $x_i^+$, where
\begin{equation*}
    \mathscr T_{k_0}[g_{i+1}](x_i^+)
=
k_0\csc(k_0\ell_{i,i+1})\,a_{i+1}.
\end{equation*}
Therefore,
\begin{equation*}
\bigl\langle \mathscr T_{k_0}[g_{i+1}],g_i\bigr\rangle_{\partial D}
=
g_i(x_i^+)\,\mathscr T_{k_0}[g_{i+1}](x_i^+)
=
\frac{2k_0}{\sqrt{\ell_i\ell_{i+1}}}\,s_i\,\csc(k_0\ell_{i,i+1}),
\end{equation*}
and \eqref{eq:1d-freq-cap-def} yields the formula for $\mathcal C_{i,i+1}(\omega_0)$. The formula for $\mathcal C_{i+1,i}(\omega_0)$ follows, by the same argument, from
\begin{equation*}
    \mathscr T_{k_0}[g_i](x_{i+1}^-)
=
k_0\csc(k_0\ell_{i,i+1})\,s_i a_i. \qedhere
\end{equation*}
\end{proof}

We now clarify the relation between the capacitance matrix \eqref{eq:1d-freq-cap-def} derived above from the ODE formulation and the one introduced in \cite{fabry1}. Throughout this comparison, as in \cite{fabry1}, we assume that all resonators share the same interior wave speed and the same contrast parameter:
\begin{equation*}
    v_j\equiv v_b,
\qquad
\delta_j\equiv \delta,
\qquad
r:=\frac{v}{v_b},
\end{equation*}
and we write
\begin{equation*}
    k:=\frac{\omega}{v},
\qquad
k_b:=\frac{\omega}{v_b}=rk.
\end{equation*}
Under these conventions, the matrix \eqref{eq:1d-freq-cap-def} takes the form
\begin{equation*}
    \mc C^{\rm ODE}_{ij}(\omega)
:=
-\frac{\delta v_b^2}{2\omega}
\bigl\langle \mathscr T_k[g_j],g_i\bigr\rangle_{\partial D},
\end{equation*}
which we refer to as the \emph{ODE-based capacitance matrix}. By contrast, the capacitance matrix $\mc C^{\rm FP}$ introduced in \cite[(2.4)--(2.7)]{fabry1}, which we call the \emph{Fabry--P\'erot capacitance matrix} in this appendix to distinguish it from $\mc C^{\rm ODE}(\omega)$, is associated with the alternating structural vector
\begin{equation*}
t=(r\ell_1,\ell_{1,2},r\ell_2,\ell_{2,3},\dots,\ell_{N-1,N},r\ell_N),
\end{equation*}
and is built from the auxiliary coefficients 
\begin{equation*}
\theta_j(k_0):=\frac{1}{t_j(k_0)\,t_{j+1}(k_0)},
\qquad 1\le j\le 2N-2,
\end{equation*}
where
\begin{equation*}
t_j(k_0):=
\begin{cases}
t_j,& k_0 t_j\in \pi\mathbb Z,\\
\infty,& k_0 t_j\notin \pi\mathbb Z;
\end{cases}
\qquad 1\le j\le 2N-1.
\end{equation*}
Here and below, we use the convention $1/\infty=0$. In particular, $\theta_j(k_0)$ is nonzero only when both $t_j$ and $t_{j+1}$ are resonant at $k_0$. The nonzero eigenvalues of $\mc C^{\rm FP}$ govern the $\O(\delta^{1/2})$ splitting of the resonances near $\omega_0$ \cite[Theorem 2.3]{fabry1}.

The two matrices are linked as follows: $\mc C^{\rm ODE}(\omega)$ is the regular effective matrix when the exterior DtN map is analytic at $\omega_0$, whereas, when $\omega_0$ is also a Dirichlet eigenfrequency for one or more of the exterior spacings, the symmetrized Fabry--P\'erot matrix determines the Laurent residue of $\mc C^{\rm ODE}(\omega)$, up to the scalar factor and sign conjugation displayed below (this is similar to Case 2 in the periodic setting of \Cref{sec:4}).  

We begin with the regular case, which corresponds to \cite[Theorem 2.3, (2.14)]{fabry1}.

\begin{proposition}[(Regular case)]\label{prop:1d-regular}
Let $\mc J\subset \{1,\dots,N\}$ satisfy
\begin{equation*}
    k_{b,0}\ell_j\in \pi \Z
\qquad\text{for every }j\in \mc J,
\end{equation*}
and assume in addition that
\begin{equation*}
    k_0\ell_{i,i+1}\notin \pi\Z
\qquad\text{for every spacing }\ell_{i,i+1}\text{ adjacent to a resonator in } \mc J.
\end{equation*}
Then $\mc C^{\rm ODE}(\omega)$ is analytic in a neighborhood of $\omega_0$ and satisfies
\begin{equation*}
    \mc C^{\rm ODE}(\omega_0)= \O(\delta).
\end{equation*}
In particular, the resonance near $\omega_0$ satisfies $\omega = \omega_0 + \O(\delta)$.
\end{proposition}

\begin{proof}
By \cref{prop:1d-freq-cap-formula}, each entry of $\mc C^{\rm ODE}(\omega)$ is a linear combination of terms of the form
\begin{equation*}
\cot(k\ell_{i,i+1})\quad\text{and}\quad \csc(k\ell_{i,i+1}),
\end{equation*}
with prefactors of order $\delta v_b^2/(v\ell_j)$ or $\delta v_b^2/(v\sqrt{\ell_j\ell_{j+1}})$. Since $\cot$ and $\csc$ are holomorphic away from $\pi\Z$, the assumption $k_0\ell_{i,i+1}\notin \pi\Z$ on the relevant spacings ensures that every entry of $\mc C^{\rm ODE}(\omega)$ is holomorphic near $\omega_0$. The $\mc C^{\rm ODE}(\omega_0)=\O(\delta)$ then follows. 
\end{proof}

We now turn to the singular case, which is the regime in \cite[Theorem 2.3, (2.13)]{fabry1}. The argument is similar to \cref{lem:reg-sing-connection}. Without loss of generality, we focus on the $C^{1,1}$ block in \cite[Section 3.1]{fabry1}; other principal blocks $C^{i,j}$, $i,j \in \{0,1\}$, can be discussed similarly.

\begin{proposition}[(Singular case)]\label{prop:1d-singular}
Assume that there exist integers $1\le p<q\le N$, $n_i\in \mathbb N$ ($p\le i\le q$), and $m_i\in \Z$ ($p\le i\le q-1$) such that
\begin{equation*}
    k_{b,0}\ell_i=n_i\pi
\quad (p\le i\le q),
\qquad
k_0\ell_{i,i+1}=m_i\pi
\quad (p\le i\le q-1),
\end{equation*}
and that $k_0\ell_{p-1,p}\notin \pi\Z$ when $p>1$, and $k_0\ell_{q,q+1}\notin \pi\Z$ when $q<N$. Set $\mc J_\ast:=\{p,p+1,\dots,q\}$, and let
\begin{equation*}
    \mc C^{\rm ODE}_\ast(\omega) := \bigl(\mc C^{\rm ODE}_{ij}(\omega)\bigr)_{i,j\in \mc J_\ast}
\end{equation*}
denote the corresponding principal block of $\mc C^{\rm ODE}(\omega)$. Let $C_\ast^{\rm FP}(k_0)$ be the associated principal Fabry--P\'erot block $C^{1,1}$ from \cite[Section 3.1]{fabry1}, with symmetrized form
\begin{equation*}
    C_{\ast,\rm sym}^{\rm FP}(k_0)
=
\begin{pmatrix}
\dfrac{1}{r\ell_p \ell_{p,p+1}}
&
-\dfrac{1}{r \ell_{p,p+1}\sqrt{\ell_p\ell_{p+1}}}
&
&&
\\[2mm]
-\dfrac{1}{r \ell_{p,p+1}\sqrt{\ell_p\ell_{p+1}}}
&
\dfrac{1}{r\ell_{p+1}\ell_{p,p+1}}+\dfrac{1}{r\ell_{p+1}\ell_{p+1,p+2}}
&
-\dfrac{1}{r \ell_{p+1,p+2}\sqrt{\ell_{p+1}\ell_{p+2}}}
&
&
\\
&
\ddots
&
\ddots
&
\ddots
&
\\
&
&
-\dfrac{1}{r \ell_{q-1,q}\sqrt{\ell_{q-1}\ell_q}}
&
\dfrac{1}{r\ell_q \ell_{q-1,q}}
\end{pmatrix}.
\end{equation*}
Finally, define the diagonal sign matrix
\begin{equation*}
    \Sigma:=\diag(\tau_p,\tau_{p+1},\dots,\tau_q),
\qquad
\tau_p=1,
\qquad
\tau_{i+1}=\tau_i\,(-1)^{n_i+m_i}
\quad (p\le i\le q-1).
\end{equation*}
Then
\begin{equation}\label{eq:main-connection-minus-plus}
    \mc C^{\rm ODE}_\ast(\omega)
=
\frac{\delta vv_b}{\omega-\omega_0}\,
\Sigma\,C_{\ast,\rm sym}^{\rm FP}(k_0)\,\Sigma
+
\delta\,H_\ast(\omega),
\end{equation}
where $H_\ast(\omega)$ is analytic near $\omega_0$.
\end{proposition}

\begin{proof}
By \cref{prop:1d-freq-cap-formula} and the uniform-resonator assumption ($v_j\equiv v_b$, $\delta_j\equiv \delta$), the diagonal entries of $\mc C^{\rm ODE}_\ast(\omega)$ are given by
\begin{equation*}
\mc C^{\rm ODE}_{ii}(\omega)
=
\frac{\delta v_b^2}{v\ell_i}
\Bigl(\cot(k\ell_{i-1,i})+\cot(k\ell_{i,i+1})\Bigr)
\qquad (p<i<q),
\end{equation*}
with the analogous expressions involving a single cotangent term at $i=p$ and $i=q$ (since the outer spacings $\ell_{p-1,p}$ and $\ell_{q,q+1}$ are non-resonant by assumption). From
\begin{equation*}
k\ell_{i,i+1}-k_0\ell_{i,i+1}=\ell_{i,i+1}(k-k_0),
\qquad
\omega-\omega_0=v(k-k_0),
\end{equation*}
and $k_0\ell_{i,i+1}=m_i\pi$, we obtain, as $\omega\to \omega_0$,
\begin{equation*}
\cot(k\ell_{i,i+1})
=\frac{1}{\ell_{i,i+1}(k-k_0)}+\O(1)
=\frac{v}{\ell_{i,i+1}(\omega-\omega_0)}+\O(1).
\end{equation*}
It follows that for $p<i<q$,
\begin{equation*}
\mc C^{\rm ODE}_{ii}(\omega)
=
\frac{\delta v_b^2}{\omega-\omega_0}
\Bigl(\frac{1}{\ell_i\ell_{i-1,i}}+\frac{1}{\ell_i\ell_{i,i+1}}\Bigr)
+\O(\delta),
\end{equation*}
and at the endpoints
\begin{equation*}
\mc C^{\rm ODE}_{pp}(\omega)
=
\frac{\delta v_b^2}{\omega-\omega_0}\,\frac{1}{\ell_p\ell_{p,p+1}}
+\O(\delta),
\qquad
\mc C^{\rm ODE}_{qq}(\omega)
=
\frac{\delta v_b^2}{\omega-\omega_0}\,\frac{1}{\ell_q\ell_{q-1,q}}
+\O(\delta).
\end{equation*}
For the off-diagonal entries, \cref{prop:1d-freq-cap-formula} gives, for $p\le i\le q-1$,
\begin{equation*}
\mc C^{\rm ODE}_{i,i+1}(\omega)
=
-\frac{\delta v_b^2}{v\sqrt{\ell_i\ell_{i+1}}}\,
(-1)^{n_i}\csc(k\ell_{i,i+1}).
\end{equation*}
Using $k_0\ell_{i,i+1}=m_i\pi$ and $\csc(m_i\pi+x)=(-1)^{m_i}/x+\O(x)$ as $x\to 0$,
\begin{equation*}
\csc(k\ell_{i,i+1})
=\frac{(-1)^{m_i}v}{\ell_{i,i+1}(\omega-\omega_0)}+\O(1),
\end{equation*}
and therefore
\begin{equation*}
\mc C^{\rm ODE}_{i,i+1}(\omega)
=
-\frac{\delta v_b^2}{\omega-\omega_0}\,
\frac{(-1)^{n_i+m_i}}{\ell_{i,i+1}\sqrt{\ell_i\ell_{i+1}}}
+\O(\delta).
\end{equation*}
The expansion of $\mc C^{\rm ODE}_{i+1,i}(\omega)$ is identical since the prefactors $\delta_i v_i^2$ and $\delta_{i+1}v_{i+1}^2$ coincide under the uniform-resonator  assumption.

Combining the above expansions, we may write
\begin{equation} \label{expabove}
\mc C^{\rm ODE}_\ast(\omega)
=
\frac{\delta v_b^2}{\omega-\omega_0}\,M_\ast
+
\delta\,H_\ast(\omega),
\end{equation}
where $H_\ast(\omega)$ is holomorphic near $\omega_0$ and $M_\ast$ is the tridiagonal matrix with
\begin{equation*}
(M_\ast)_{pp}=\frac{1}{\ell_p\ell_{p,p+1}},
\qquad
(M_\ast)_{ii}=\frac{1}{\ell_i\ell_{i-1,i}}+\frac{1}{\ell_i\ell_{i,i+1}}
\quad (p<i<q),
\qquad
(M_\ast)_{qq}=\frac{1}{\ell_q\ell_{q-1,q}},
\end{equation*}
\begin{equation*}
(M_\ast)_{i,i+1}=(M_\ast)_{i+1,i}
=
-\frac{(-1)^{n_i+m_i}}{\ell_{i,i+1}\sqrt{\ell_i\ell_{i+1}}}
\qquad (p\le i\le q-1).
\end{equation*}
It remains to identify $M_\ast$ with the symmetrized Fabry--P\'erot block. From $r=v/v_b$, we have $vv_b/r=v_b^2$, and hence
\begin{equation*}
v_b^2\cdot\frac{1}{\ell_i\ell_{i,i+1}}
=\frac{vv_b}{r\ell_i\ell_{i,i+1}},
\qquad
v_b^2\cdot\frac{1}{\ell_{i,i+1}\sqrt{\ell_i\ell_{i+1}}}
=\frac{vv_b}{r\ell_{i,i+1}\sqrt{\ell_i\ell_{i+1}}},
\end{equation*}
so the diagonal entries of $v_b^2 M_\ast$ are exactly $vv_b$ times those of $C^{\rm FP}_{\ast,\rm sym}(k_0)$. The recursion $\tau_{i+1}=\tau_i(-1)^{n_i+m_i}$ yields $\tau_i\tau_{i+1}=(-1)^{n_i+m_i}$ for $p\le i\le q-1$, so conjugation by $\Sigma$ reproduces precisely the signs of the off-diagonal entries of $M_\ast$. Therefore,
\begin{equation*}
v_b^2\,M_\ast
=
vv_b\,\Sigma\,C_{\ast,\rm sym}^{\rm FP}(k_0)\,\Sigma.
\end{equation*}
Substituting the above expression into \eqref{expabove} yields \eqref{eq:main-connection-minus-plus}.
\end{proof}

In particular, \cref{prop:1d-singular} gives
\begin{equation*}
    \operatorname*{Res}_{\omega=\omega_0}\mc C_\ast^{\rm ODE}(\omega)
    =
    \delta vv_b\,\Sigma C_{\ast,\rm sym}^{\rm FP}(k_0)\Sigma.
\end{equation*}
Thus, the normalized Laurent residue obtained by removing the scalar factor $\delta vv_b$ is a sign conjugate of the symmetrized Fabry--P\'erot block. The following corollary makes the resulting $\O(\delta^{1/2})$ splitting in the limiting resonance equation precise.

\begin{corollary}
Let $\lambda$ be a nonzero eigenvalue of $C_\ast^{\rm FP}(k_0)$. Then $\lambda$ is positive and simple, and is also an eigenvalue of $C_{\ast,\rm sym}^{\rm FP}(k_0)$ and of its sign conjugate $\Sigma C_{\ast,\rm sym}^{\rm FP}(k_0)\Sigma$. Moreover, for $a\neq0$, the limiting resonance equation
\begin{equation} \label{eq:reduced}
    (\omega-\omega_0)\,a=\mc C_\ast^{\rm ODE}(\omega)\,a
\end{equation}
has two simple branches associated with $\lambda$, satisfying
\begin{equation*}
    \omega-\omega_0
=
\pm \sqrt{\delta vv_b\,\lambda}+ \O(\delta)
=
\pm v\sqrt{\frac{\lambda}{r}}\,\delta^{1/2}+ \O(\delta),
\end{equation*}
recovering the asymptotic formula in \cite[Theorem 2.3, (2.13)]{fabry1}.
\end{corollary}

\begin{proof}
By \cite[Lemma 3.1]{fabry1}, every nonzero eigenvalue of $C_\ast^{\rm FP}(k_0)$ is positive and simple, and the nonzero spectra of $C_\ast^{\rm FP}(k_0)$ and $C_{\ast,\rm sym}^{\rm FP}(k_0)$ coincide. Since $\Sigma^{-1}=\Sigma$, the matrix $A:=\Sigma C_{\ast,\rm sym}^{\rm FP}(k_0)\Sigma$ also has $\lambda$ as a simple eigenvalue. Set $z:=\omega-\omega_0$. For $z\neq0$, substituting \eqref{eq:main-connection-minus-plus} into \eqref{eq:reduced} and multiplying by $z$ shows that a nonzero solution exists precisely when
\begin{equation*}
    \det\bigl(z^2I-\delta vv_b A-\delta zH_\ast(\omega_0+z)\bigr)=0.
\end{equation*}
Write $\delta=\varepsilon^2$ and $z=\varepsilon\mu$. Factoring $\varepsilon^2$ from the matrix pencil reduces this equation to
\begin{equation*}
    \Phi(\mu,\varepsilon)
    :=
    \det\bigl(\mu^2I-vv_b A-\varepsilon\mu H_\ast(\omega_0+\varepsilon\mu)\bigr)
    =0.
\end{equation*}
At $\varepsilon=0$, the two values $\mu_\pm^{(0)}=\pm\sqrt{vv_b\lambda}$ are simple zeros of $\Phi(\mathord\cdot,0)$. The analytic implicit function theorem therefore yields two simple branches $\mu_\pm(\varepsilon)=\pm\sqrt{vv_b\lambda}+\O(\varepsilon)$, and hence
\begin{equation*}
    z_\pm
    =
    \pm\sqrt{\delta vv_b\lambda}+\O(\delta),
\end{equation*}
and the final expression follows from $vv_b=v^2/r$.
\end{proof}


\begin{thebibliography}{10}

\bibitem{skin3d}
Habib Ammari, Silvio Barandun, Jinghao Cao, Bryn Davies, Erik~Orvehed Hiltunen,
  and Ping Liu.
\newblock The non-{H}ermitian skin effect with three-dimensional long-range
  coupling.
\newblock {\em J. Eur. Math. Soc.}, 2026.

\bibitem{ammari2020robust}
Habib Ammari, Bryn Davies, and Erik~Orvehed Hiltunen.
\newblock Robust edge modes in dislocated systems of subwavelength resonators.
\newblock {\em J. Lond. Math. Soc. (2)}, 106(3):2075--2135, 2022.

\bibitem{anderson}
Habib Ammari, Bryn Davies, and Erik~Orvehed Hiltunen.
\newblock Anderson localization in the subwavelength regime.
\newblock {\em Comm. Math. Phys.}, 405(1):Paper No. 1, 20, 2024.

\bibitem{ammari2025mathematical}
Habib Ammari, Bryn Davies, and Erik~Orvehed Hiltunen.
\newblock {\em Mathematical Theories for Metamaterials: From Condensed Matter
  Theory to Subwavelength Physics}, volume 136 of {\em CBMS Regional Conference
  Series in Mathematics}.
\newblock American Mathematical Society, Providence, RI, 2026.

\bibitem{fano}
Habib Ammari, Bryn Davies, Erik~Orvehed Hiltunen, Hyundae Lee, and Sanghyeon
  Yu.
\newblock Bound states in the continuum and {F}ano resonances in subwavelength
  resonator arrays.
\newblock {\em J. Math. Phys.}, 62(10):Paper No. 101506, 24, 2021.

\bibitem{ammari2020exceptional}
Habib Ammari, Bryn Davies, Erik~Orvehed Hiltunen, Hyundae Lee, and Sanghyeon
  Yu.
\newblock Exceptional points in parity-time-symmetric subwavelength
  metamaterials.
\newblock {\em SIAM J. Math. Anal.}, 54(6):6223--6253, 2022.

\bibitem{ammari2019topologically}
Habib Ammari, Bryn Davies, Erik~Orvehed Hiltunen, and Sanghyeon Yu.
\newblock Topologically protected edge modes in one-dimensional chains of
  subwavelength resonators.
\newblock {\em J. Math. Pures Appl. (9)}, 144:17--49, 2020.

\bibitem{ammari2024functional}
Habib Ammari, Bryn Davies, and Erik Orvehed~Hiltunen.
\newblock Functional analytic methods for discrete approximations of
  subwavelength resonator systems.
\newblock {\em Pure and Applied Analysis}, 6(3):873--939, 2024.

\bibitem{honeycomb1}
Habib Ammari, Brian Fitzpatrick, Erik~Orvehed Hiltunen, Hyundae Lee, and
  Sanghyeon Yu.
\newblock Honeycomb-lattice {M}innaert bubbles.
\newblock {\em SIAM J. Math. Anal.}, 52(6):5441--5466, 2020.

\bibitem{photonic2018}
Habib Ammari, Brian Fitzpatrick, Hyeonbae Kang, Matias Ruiz, Sanghyeon Yu, and
  Hai Zhang.
\newblock {\em Mathematical and computational methods in photonics and
  phononics}, volume 235 of {\em Mathematical Surveys and Monographs}.
\newblock American Mathematical Society, Providence, RI, 2018.

\bibitem{ammari2017subwavelength}
Habib Ammari, Brian Fitzpatrick, Hyundae Lee, Sanghyeon Yu, and Hai Zhang.
\newblock Subwavelength phononic bandgap opening in bubbly media.
\newblock {\em Journal of Differential Equations}, 263(9):5610--5629, 2017.

\bibitem{honeycomb3}
Habib Ammari, Xin Fu, and Wenjia Jing.
\newblock Wave packets propagation in the subwavelength regime near the {D}irac
  point.
\newblock {\em Arch. Ration. Mech. Anal.}, 250(2):Paper No. 12, 67, 2026.

\bibitem{alex}
Habib Ammari, Erik~Orvehed Hiltunen, Bowen Li, Ping Liu, Jiayu Qiu, Yingjie
  Shao, and Alexander Uhlmann.
\newblock Non-{H}ermitian {F}abry-{P}érot resonances.
\newblock {\em arXiv:2601.19855}, 2026.

\bibitem{honeycomb2}
Habib Ammari, Erik~Orvehed Hiltunen, and Sanghyeon Yu.
\newblock A high-frequency homogenization approach near the {D}irac points in
  bubbly honeycomb crystals.
\newblock {\em Arch. Ration. Mech. Anal.}, 238(3):1559--1583, 2020.

\bibitem{ammari2009layer}
Habib Ammari, Hyeonbae Kang, and Hyundae Lee.
\newblock {\em Layer potential techniques in spectral analysis}, volume 153 of
  {\em Mathematical Surveys and Monographs}.
\newblock American Mathematical Society, Providence, RI, 2009.

\bibitem{fabry1}
Habib Ammari, Bowen Li, Ping Liu, and Yingjie Shao.
\newblock Frequency-dependent capacitance matrix formulation for
  {F}abry-{P}érot resonances. {P}art {I}: One-dimensional finite systems.
\newblock {\em arXiv:2604.01159}, 2026.

\bibitem{fabry2}
Habib Ammari, Bowen Li, Ping Liu, and Yingjie Shao.
\newblock Frequency-dependent capacitance matrix formulation for
  {F}abry-{P}érot resonances. {P}art {II}: One-dimensional periodic systems.
\newblock {\em in preparation}, 2026.

\bibitem{barnett2}
Alex~H. Barnett, Andrew Hassell, and Melissa Tacy.
\newblock Comparable upper and lower bounds for boundary values of {N}eumann
  eigenfunctions and tight inclusion of eigenvalues.
\newblock {\em Duke Math. J.}, 167(16):3059--3114, 2018.

\bibitem{spence1}
Dean Baskin, Euan~A. Spence, and Jared Wunsch.
\newblock Sharp high-frequency estimates for the {H}elmholtz equation and
  applications to boundary integral equations.
\newblock {\em SIAM J. Math. Anal.}, 48(1):229--267, 2016.

\bibitem{spence3}
S.~N. Chandler-Wilde, E.~A. Spence, A.~Gibbs, and V.~P. Smyshlyaev.
\newblock High-frequency bounds for the {H}elmholtz equation under parabolic
  trapping and applications in numerical analysis.
\newblock {\em SIAM J. Math. Anal.}, 52(1):845--893, 2020.

\bibitem{Subwavelength_1D_High-Contrast}
Florian Feppon, Zijian Cheng, and Habib Ammari.
\newblock Subwavelength resonances in one-dimensional high-contrast acoustic
  media.
\newblock {\em SIAM J. Appl. Math}, 83(2):625--665, 2023.

\bibitem{refhigh3}
Moritz Forsch, Robert Stockill, Andreas Wallucks, Igor Marinković, Claus
  Gärtner, Richard~A. Norte, Frank van Otten, Andrea Fiore, Kartik Srinivasan,
  and Simon Gröblacher.
\newblock Microwave-to-optics conversion using a mechanical oscillator in its
  quantum ground state.
\newblock {\em Nature Physics}, 16:69--74, 2020.

\bibitem{spence2}
J.~Galkowski, P.~Marchand, and E.~A. Spence.
\newblock High-frequency estimates on boundary integral operators for the
  {H}elmholtz exterior {N}eumann problem.
\newblock {\em Integral Equations Operator Theory}, 94(4):Paper No. 36, 68,
  2022.

\bibitem{gohberg1971operator}
Israel~C Gohberg and Efim~I Sigal.
\newblock An operator generalization of the logarithmic residue theorem and the
  theorem of rouch{\'e}.
\newblock {\em Mathematics of the USSR-Sbornik}, 13(4):603--625, 1971.

\bibitem{jfa}
Xiaolong Han and Melissa Tacy.
\newblock Sharp norm estimates of layer potentials and operators at high
  frequency.
\newblock {\em J. Funct. Anal.}, 269(9):2890--2926, 2015.
\newblock With an appendix by Jeffrey Galkowski.

\bibitem{barnett1}
Andrew Hassell and Alex Barnett.
\newblock Estimates on {N}eumann eigenfunctions at the boundary, and the
  ``method of particular solutions'' for computing them.
\newblock In {\em Spectral geometry}, volume~84 of {\em Proc. Sympos. Pure
  Math.}, pages 195--208. Amer. Math. Soc., Providence, RI, 2012.

\bibitem{henry2005perturbation}
Dan Henry.
\newblock {\em Perturbation of the boundary in boundary-value problems of
  partial differential equations}, volume 318 of {\em London Mathematical
  Society Lecture Note Series}.
\newblock Cambridge University Press, Cambridge, 2005.

\bibitem{pm1}
Yi~Huang, Bowen Li, Ping Liu, and Yingjie Shao.
\newblock Resonance analysis of one-dimensional acoustic media: A propagation
  matrix approach.
\newblock {\em Studies in Applied Mathematics}, 156(5):e70242, 2026.

\bibitem{elast2}
Hongjie Li and Longjuan Xu.
\newblock Resonant modes of two hard inclusions within a soft elastic material
  and their stress estimates.
\newblock {\em J. Differential Equations}, 453:Paper No. 113822, 44, 2026.

\bibitem{elast1}
Hongjie Li and Jun Zou.
\newblock Mathematical justifications of dipolar resonances with hard
  inclusions embedded in a soft elastic material.
\newblock {\em SIAM J. Appl. Math.}, 85(4):1810--1833, 2025.

\bibitem{li2025high}
Long Li and Mourad Sini.
\newblock High contrast transmission and {F}abry-{P}érot-type resonances.
\newblock {\em arXiv:2510.19096}, 2025.

\bibitem{mclean2000strongly}
William Charles~Hector McLean.
\newblock {\em Strongly elliptic systems and boundary integral equations}.
\newblock Cambridge University Press, 2000.

\bibitem{elast3}
Yuanchun Ren, Bochao Chen, Yixian Gao, and Peijun Li.
\newblock Subwavelength {P}hononic {B}andgaps in {H}igh-{C}ontrast {E}lastic
  {M}edia.
\newblock {\em Multiscale Model. Simul.}, 24(2):429--455, 2026.

\bibitem{ultrasonic}
K.G. Scheuer, F.B. Romero, and R.G. DeCorby.
\newblock Coupling the thermal acoustic modes of a bubble to an optomechanical
  sensor.
\newblock {\em Microsystems $\&$ Nanoengineering}, 10:204, 2024.

\bibitem{tataru}
Daniel Tataru.
\newblock On the regularity of boundary traces for the wave equation.
\newblock {\em Ann. Scuola Norm. Sup. Pisa Cl. Sci. (4)}, 26(1):185--206, 1998.

\bibitem{teytel1999rare}
Mikhail Teytel.
\newblock How rare are multiple eigenvalues?
\newblock {\em Communications on Pure and Applied Mathematics: A Journal Issued
  by the Courant Institute of Mathematical Sciences}, 52(8):917--934, 1999.

\bibitem{generic}
K.~Uhlenbeck.
\newblock Generic properties of eigenfunctions.
\newblock {\em Amer. J. Math.}, 98(4):1059--1078, 1976.

\bibitem{refhigh1}
Kerry~J. Vahala.
\newblock Optical microcavities.
\newblock {\em Nature}, 424:839--846, 2003.

\bibitem{refhigh2}
Maxim~K. Zalalutdinov, Jeremy~T. Robinson, Jose~J. Fonseca, Samuel~W. LaGasse,
  Tribhuwan Pandey, Lucas~R. Lindsay, Thomas~L. Reinecke, Douglas~M. Photiadis,
  James~C. Culbertson, Cory~D. Cress, and Brian~H. Houston.
\newblock Acoustic cavities in 2d heterostructures.
\newblock {\em Nature Communications}, 12:3267, 2021.

\end{thebibliography}
\end{document}